\documentclass{amsart}

\usepackage[normalem]{ulem}
\usepackage[english]{babel}
\usepackage{amsmath,amsfonts,amsthm,mathtools, amssymb}
\usepackage{xcolor}
\usepackage{tikz-cd}
\usepackage{hyperref}
\usepackage{aliascnt}
\usepackage{xparse}
\usepackage{tensor}
\usepackage{caption}
\usepackage{enumitem}   

\usepackage{tabto}
\usepackage{imakeidx}

\def\M{{\mathcal{M}}}

\def\oC{\overline{\mathcal{C}}}

\def\oM{\overline{\mathcal{M}}}
\def\cM{{\mathcal{M}}}
\def\Om{{\Omega}}
\def\oOm{\overline{\Omega}}

\def\cH{{\mathcal{H}}}

\def\cO{\mathcal{O}}

\def\cA{\mathcal{A}}
\def\cB{\mathcal{B}}
\def\cS{\mathcal{S}}
\def\cF{\mathcal{F}}
\def\cE{\mathcal{E}}
\def\frakR{\mathfrak{R}}
\def\DR{{\rm DR}}
\def\log{{\rm log}}

\def\LG{\mathrm{LG}}
\def\Aut{\mathrm{Aut}}
\def\sh{{\rm sinh}}
\def\ch{{\rm cosh}}
\def\sp{{\rm spin}}
\def\odd{{\rm odd}}
\def\even{{\rm even}}
\def\Ind{{\rm Ind}}
\def\mod{{\rm mod}}
\def\LCM{{\rm lcm }}

\def\oGamma{{\overline{\Gamma}}}

\def\RR{\mathbb{R}}
\def\CC{\mathbb{C}}
\def\ZZ{\mathbb{Z}}

\def\QQ{\mathbb{Q}}
\def\PP{\mathbb{P}}

\theoremstyle{definition}
\newtheorem{definition}{Definition}[section]

\newtheorem{remark}[definition]{Remark}
\newtheorem{claim}[definition]{Claim}

\newtheorem{assumption}[definition]{Assumption}
\theoremstyle{plain}
\newtheorem{conjecture}[definition]{Conjecture}
\newtheorem{question}[definition]{Question}
\newtheorem{theorem}[definition]{Theorem}
\newtheorem{proposition}[definition]{Proposition}

\newtheorem{lemma}[definition]{Lemma}
\newtheorem{corollary}[definition]{Corollary}

\newcommand{\rcomment}[1]{\textcolor{black}{ #1}}
\newcommand{\rcommenttwo}[1]{\textcolor{black}{ #1}}
\newcommand{\rcommentthree}[1]{\textcolor{black}{ #1}}
\newcommand{\rcommentfour}[1]{\textcolor{black}{ #1}}
\newcommand{\rcommentfive}[1]{\textcolor{black}{ #1}}

\title[Integrals of $\psi$-classes on twisted double ramification cycles]{Integrals of $\psi$-classes on twisted double ramification cycles and spaces of differentials}
\date{\today}

\author{Matteo Costantini}
\address{Universität Duisburg-Essen
Fakultät für Mathematik
45117 Essen, , Germany}
\email{matteo.costantini@uni-due.de}

\author{Adrien Sauvaget}
\address{CNRS, Laboratoire AGM, Université de Cergy-Pontoise, 2 av. Adolphe Chauvin, 95000 Cergy-Pontoise, France}
\email{adrien.sauvaget@math.cnrs.fr}

\author{Johannes Schmitt}
\address{Department Mathematik,
R\"amistrasse 101,
CH-8092 Z\"urich,
Switzerland}
\email{johannes.schmitt@math.uzh.ch}

\begin{document}

\maketitle

\begin{abstract}
We prove a closed formula for the integral of a power of a single $\psi$-class on strata of $k$-differentials. In general, these integrals correspond to intersection numbers on twisted double ramification cycles.  

Then, we conjecture an expression of a refinement of double ramification cycles according to the parity of spin structures. Assuming that this conjecture is valid, we also compute the integral of a single $\psi$-class on the even and odd components of strata of $k$-differentials.

As an application of these results we give a closed formula for the Euler characteristic of minimal strata of abelian differentials. 
\end{abstract}

\setcounter{tocdepth}{1}
\tableofcontents

\section{Introduction}
The Hodge bundle over the moduli space of curves is the moduli space of abelian differentials, i.e.  Riemann surfaces endowed with a holomorphic 1-form. The complement of the zero section admits a stratification according to the orders of  the zeros of the one-form. These strata may be   studied from different perspectives. 

These strata appear naturally in many areas of mathematics, such as the study of flat surfaces, Teichm\"uller dynamics and associated counting problems (see \cite{EsMi},
\cite{EsMiMo}, \cite{Filip}, \cite{EFW}, \cite{EMMW} to only cite a few of many works in these directions).


Since these spaces can be described as algebraic objects, we may use intersection theory to understand their numerical invariants and dynamics. For instance, using a smooth compactification of the moduli spaces of abelian differentials (\cite{BCGGM3}, \cite{CMZarea}), it was possible to apply intersection theory to compute the Masur-Veech volumes and the sum of the Lyapunov exponents of the strata (\cite{Sau:volumeminimal}, \cite{chmosaza}), as well as their Euler characteristic (\cite{CMZ20}). 
More generally one may consider the moduli spaces of holomorphic $k$-forms, and once again numerical invariants of these spaces may be computed by intersection theory (see~\cite{CMSvol},~\cite{Zoretal} and~\cite{Sau3}).

Another interesting connection to enumerative geometry is the relation between strata of $k$-differentials and  twisted double ramification cycles (see \cite{FP, Sch18, HS19, BHPSS20}). Classical double ramification cycles were introduced to study Gromov-Witten invariants via degeneration techniques (\cite{JLi}). For instance, the Gromov-Witten theory of curves may be reduced to integrals of $\psi$-classes on double ramification cycles (\cite{OkoPan},~\cite{BSSZ}).

The aim of the present work is to generalize the results of~\cite{BSSZ} and compute integrals of the top power of a single $\psi$-class on strata of $k$-differentials and on  twisted double ramification cycles.  The results we present allow both to simplify the computations appearing in \cite{Sau3} about volumes of strata of flat surfaces with rational singularities and to give an explicit formula for the orbifold Euler characteristic of minimal strata of abelian differentials based on results from \cite{CMZ20}. This last application allows us to formulate a conjecture about the asymptotic growth of the orbifold Euler characteristic of such strata for large values of the genus. 

Moreover, if the zeros of a $1$-form are even, one may associate a canonical spin structure to it (i.e. a square root of the cotangent bundle of the curve). This spin structure defines an invariant called the parity of the spin structure. Below, we formulate a conjecture on a spin analogue of double ramification cycles. We show that this conjecture allows to compute 
intersection numbers of powers of $\psi$-classes and the orbifold Euler characteristics of strata of differentials weighted by a sign depending on the parity of the spin structure.  

We will now provide a more in depth summary of the results of this work.

\subsection{Strata of $k$-differentials}

Let $g$ and $n$ be non-negative integers satisfying $2g-2+n>0$. We denote by $\M_{g,n}$ and $\oM_{g,n}$ the moduli spaces of genus $g$, smooth and stable curves with $n$ distinct markings, respectively. 

Let $k$ be an integer, and let $a=(a_1,\ldots,a_n)\in \ZZ^n$ be a vector satisfying 
$$|a| \coloneqq \sum_{i=1}^n a_i=k(2g-2+n).$$ 
We denote\footnote{Note that we do not include $k$ in the above notation as it is determined from $a$ and $g$.} by $\M_g(a)\subset \M_{g,n}$ the locus of marked curves $(C,x_1,\ldots,x_n)$ such that 
\begin{equation} \label{eqn:omegalogcondition}
\omega_\log^{\otimes k}\simeq \cO_C\left(a_1x_1+\ldots+a_nx_n\right).
\end{equation}
Since $\M_g(-a)=\M_g(a)$, we can restrict our study to the case $k \geq 0$ from now on. For $k>0$ we call $\M_g(a)$ the {\em stratum of $k$-differentials of type $a$}, while if $k=0$ we call it the {\em Hurwitz scheme of type $a$}. 
To justify these names, observe that for \rcomment{$C$ smooth and} $k>0$,  the equality \eqref{eqn:omegalogcondition} of line bundles is equivalent to the existence of a meromorphic $k$-differential on $C$ with zeros and poles at the points $x_i$ of orders $a_i-k$, and this $k$-differential is unique up to scaling. On the other hand, for $k=0$ condition \eqref{eqn:omegalogcondition} is equivalent to the existence of a Hurwitz cover $C \to \mathbb{P}^1$ with the points $x_i$ forming the preimages of $0,\infty$ and local ramification orders $a_i$ (over $0$) and $-a_i$ (over $\infty$).
\rcomment{By \cite{Sch18}, t}he space $\M_g(a)$ is a smooth subspace of $\M_{g,n}$ of dimension
$$
\left\{
\begin{array}{cl}
3g-3+n, & \text{if $a=0$, $k=0$}\\
(2g-2+n), & \text{if $a\in \ZZ_{>0}^n,$ and $k=1$}\\
\text{mixed dimension}, & \text{if $a\in k\ZZ_{>0}^n,$ and $k>1$}\\
(2g-3+n), & \text{otherwise}
\end{array}
 \right.
$$ 
 If $k>0,$ and $a \in k\ZZ_{>0}$, then the space $\M_{g}(a/k)$ of dimension $2g-2+n$ sits naturally inside $\M_{g}(a)$ as the open and closed subspace of $k$-differentials obtained as the $k$th power of a $1$-differential. Meanwhile, all other components of $\M_{g}(a)$ are of dimension $2g-3+n$, thus explaining the mixed dimension in the third line.  We denote by $\oM_{g}(a)$ the Zariski closure of $\M_g(a)$ in $\oM_{g,n}$.
 
 \subsection{Double ramification cycles} 
 There exists a family of cycles, the {\em twisted}\footnote{Here, the word \emph{twisted} refers to the fact that the cycles compactify the condition $\omega_{\log}^{\otimes k} \cong \mathcal{O}_C(\sum_i a_i x_i)$ for arbitrary $k \in \mathbb{Z}$ instead of the more restrictive case $k=0$ appearing e.g. in the paper \cite{BSSZ}. In particular, we mention the word in the title of our paper to distinguish our results from theirs. That having been said, we will, in the interest of brevity, mostly omit the word \emph{twisted} in the following text.} {\em double ramification cycles}: 
$$\DR_g:\ZZ^n \to A^g(\oM_{g,n}),$$ satisfying the following properties:
\begin{itemize}
    \item we have $\DR_g(a)_{|\M_{g,n}} = [\M_g(a)] \in A^g(\M_{g,n},\QQ)$, if  $a\notin k\ZZ_{>0}$,
    \item the function $\DR_g(a)$ is a cycle-valued polynomial of degree $2g$ in the entries of the vector $a$.
\end{itemize}
We refer the reader to Section~\ref{Sec:ConjA} for a precise definition of these cycles. The purpose of the present paper is to compute the following polynomial defined by intersection theory:
\begin{eqnarray}
\cA_{g}: \ZZ^n &\to& \QQ, \label{eqn:Agdef}\\
 a &\mapsto& \int_{\oM_{g,n}} \DR_g(a)\cdot \psi_1^{2g-3+n}, \nonumber
\end{eqnarray}
where we recall that $\psi_i$ is the first Chern class of the cotangent line at the $i$-th marking for all $1\leq i\leq n$. 

\begin{theorem}\label{th:main}
For all $g, n \geq 0$ satisfying $2g-2+n>0$ and $a \in \mathbb{Z}^n$, we have
\begin{eqnarray}\label{eqn:main}
\cA_{g}(a)= [z^{2g}] \,\,\, {\rm exp}\left(\frac{a_1z\cdot \cS'(kz)}{\cS(kz)} \right) \frac{\prod_{i>1} \cS(a_iz)}{\cS(z)\cS(kz)^{2g-1+n}},
\end{eqnarray}
where $k=|a|/(2g-2+n)$,  $\cS(z)=\frac{{\rm sinh}(z/2)}{z/2},$ and $[z^{2g}](\cdot) $ stands for the coefficient of $z^{2g}$ in the formal series.
\end{theorem}
This result extends Theorem~1 of \cite{BSSZ}, which proved this formula in the case $|a|=k=0$.
As in their paper, we prove the formula \eqref{eqn:main} by first giving a list of properties satisfied by the functions $\cA_g$ which characterize them uniquely (see Lemma \ref{lem:identities}). Then we verify that the formula on the right-hand side of \eqref{eqn:main} indeed has all desired features and thus computes the value of $\cA_g$. The analogous properties in \cite{BSSZ} for $k=0$ were proven using a definition of the double ramification cycle via the Gromov-Witten theory of the projective line. For $k>0$ no such definition is available. Instead we use properties of the formula for $\DR_g(a)$ in terms of tautological classes studied in \cite{JPPZ17, PZ21} and results about the geometry of the strata of $k$-differentials (combined with their connection to the double ramification cycles).

One of the reasons why the case $k>0$ of Theorem \ref{th:main} is particularly interesting is that 
the corresponding values of $\cA_g(a)$ are related to intersection numbers on strata of $k$-differentials. Indeed, we have
\begin{equation}\label{eq:drstrata}
\cA_{g}(a) =\left\{
\begin{array}{cl}
\int_{\oM_{g,n}} [\oM_{g}(a)]\cdot \psi_1^{2g-3+n} & \text{if $k>0$ and $a_1\notin k\ZZ_{>0}$},\\
-a_1\int_{\oM_{g,n}} [\oM_{g}(a)]\cdot \psi_1^{2g-2+n} & \text{if $k=1$ and $a\in\ZZ^n_{>0}$.}
\end{array}
 \right.
\end{equation}
This equality follows from a relationship between strata of $k$-differentials and double ramification cycles conjectured (see Conjecture A in \cite{FP, Sch18}) which was recently established in \cite{HS19, BHPSS20} (see Section~\ref{Sec:ConjA} for details). More generally, the intersection number of a power of a single $\psi$ class on an arbitrary stratum of $k$-differentials may be recovered from the function $\cA_g$ (see~\cite{Sau3}, Section 2.2). 


\subsection{Spin refinement} \label{Sect:spin_refinement_intro} \rcommenttwo{If $k, a$ are odd (i.e. $k$ and all entries of $a$ are odd)}, then a marked curve $(C,x_1,\ldots,x_n)\in \M_{g}(a)$ carries a natural {\em spin structure} (i.e. a line bundle $L$ satisfying $L^{\otimes 2}\simeq \omega_C$) defined by 
\begin{equation}\label{eq:spinbundle}
L = \omega_C^{\otimes(-k+1)/2} \otimes \mathcal{O}_C\left(\frac{a_1-k}{2}x_1+\ldots+\frac{a_n-k}{2}x_n\right).
\end{equation}
The parity of a spin structure $L\to C$, defined as the parity of $h^0(C,L)$, is invariant in a family of spin structures over smooth curves (see~\cite{Mum}). In particular the space $\M_g(a)$ may be decomposed as the disjoint union
\[
\M_g(a) = \M_{g}(a)^{\rm odd} \sqcup \M_{g}(a)^{\rm even}
\]
of two (possibly themselves disconnected) subspaces defined by the parity of the above canonical spin structure.
\begin{remark}
The reader should be careful as there is another natural definition of the parity of a $k$-differential for $k>1$. Indeed, one can define the parity of a $k$-differential as the parity of the associated \emph{canonical cover} (this choice is made in~\cite{ChenGendron} for example).  This second definition is more natural from the point of view of flat surfaces (for example it is well defined for values of $k$ even and positive). However, it does not agree with our definition of parity. For instance, the parity of a $k$-differential of genus 0 may be odd with this alternative definition while it is always even with the definition that we use in the text. 
\end{remark}
We will consider the following cycle
$$
[\oM_g(a)]^\sp = [\oM_{g}(a)^{\rm even}]-[\oM_{g}(a)^{\rm odd}] \in A^*(\oM_{g,n},\QQ).
$$
Moreover, if $a_1\notin k\ZZ_{>0}$, we set:
\begin{eqnarray*}
\cA_{g}^\sp(a) &=& \int_{\oM_{g,n}} [\oM_{g}(a)]^\sp\cdot \psi_1^{2g-3+n}. 
\end{eqnarray*}
In the text, we will make the following assumption: 
\begin{assumption}\label{assumption}
There exists a family of cycles $\DR_g^\sp:\ZZ^n\to A^g(\oM_{g,n})$ such that:
\begin{enumerate}
    \item The function $\DR_g^\sp$ is a symmetric polynomial of degree $2g$;
    \item It holds
    \[
    \pi^* \DR_g^\sp(a_1, \ldots, a_n) = \DR_g^\sp(a_1, \ldots, a_n, k)
    \]
    where $\pi: \oM_{g,n+1} \to \oM_{g,n}$ is the forgetful morphism of the last marking and $k=(a_1+ \ldots +a_n)/(2g-2+n)$;
    \item If $a$ and $k$ are odd, and $P$ is a monomial in classes $\psi_i$ for which $a_i$ is negative or not divisible by $k$, we have 
    $$\int_{\oM_{g,n}} [\oM_g(a)]^\sp \cdot P =\int_{\oM_{g,n}}   \DR_g^\sp(a) \cdot P;$$
    in particular, when $a_1\notin k\ZZ_{>0}$, we have
    \begin{eqnarray*}
    \cA_{g}^\sp(a)&=& \int_{\oM_{g,n}} \DR_g^\sp(a)\cdot \psi_1^{2g-3+n},
    \end{eqnarray*}
    and thus the function $\cA_{g}^\sp(a)$ is a polynomial of degree $2g$;
    \item if $a$ is positive and $k=1$, the polynomial extension of $\cA_{g}^\sp(a)$ satisfies
     \begin{eqnarray*}
   \cA_{g}^\sp(a)&=& -a_1\int_{\oM_{g,n}} [\oM_{g}(a)]^\sp\cdot \psi_1^{2g-2+n} .
        \end{eqnarray*}
\end{enumerate}
\end{assumption}
Concerning the validity of the assumptions, we first note that all of them hold verbatim for the usual (i.e. non-spin) double ramification cycles and strata of differentials. 
Moreover, in Section~\ref{Sect:spinDR} we present an explicit candidate for the function $\DR_g^\sp$. It satisfies parts (1) and (2) of the assumption by unpublished work \cite{PZ21} of Pixton and Zagier. Moreover, in Conjecture \ref{conj:DRspinformula} we formulate a spin analogue of Conjecture A from~\cite{FP, Sch18} which would imply the remaining parts of the assumption (see Proposition \ref{prop:spinConjAimpliesAssumption}). We also describe the outline of a possible proof for this new conjecture.

Assuming the properties above, we show that the techniques used to prove Theorem~\ref{th:main} can be applied to give an explicit formula for $\cA_{g}^\sp(a)$.


\begin{theorem}\label{th:spin} If parts (1)-(3) of Assumption~\ref{assumption} are valid, then we have:
\begin{eqnarray}\label{eqn:mainspin}
 \cA_{g}^\sp (a)
 &=& 2^{-g} [z^{2g}] \,\,\, {\rm exp}\left(\frac{a_1 z\cdot \cS'(kz)}{\cS(kz)} \right) \frac{\ch(z/2)}{{\cS}(z)} \frac{\prod_{i>1} \cS(a_iz)}{\cS(kz)^{2g-1+n}} .
\end{eqnarray}
\end{theorem}


\subsection{Applications} In~\cite{Sau3}, the second author showed that the volume of moduli spaces of flat surfaces with rational singularities, as well as the Masur-Veech volumes of strata of $k$-differentials (if all 
$a_i$ are in $\ZZ_{>0}\setminus k\ZZ_{>0}$) may be computed. The input needed to calculate these volumes are the functions $\cA_g$. Therefore the present work simplifies these computations by making the input explicit.

In Section \ref{Sect:Eulerchar} below we present a second application of this formula. The first author with M. M\"oller and J. Zachhuber, provided an expression for the orbifold Euler characteristic of strata of abelian differentials in terms of intersection numbers on these strata (see~\cite{CMZ20}). We use this result together with Theorem~\ref{th:main} to obtain a closed formula for the orbifold Euler characteristic of the minimal strata $\M_g(2g-1)$ for $g\geq 1$.

In order to state this formula, we introduce the following formal series in $\QQ(y)[\![z]\!]$:
$$\cH(y,z)=\sum_{g\geq 1} (2g-1)!\, b_g \left(\frac{z\cS(z)}{y}\right)^{2g} \frac{\cS((2g-1)z)}{\cS(z)},
$$
where the numbers $b_g$ are defined by $\cS(z)^{-1}=1+\sum b_g z^{2g}$. Then we define:
$$
\chi(y,z)= \frac{y+1-y^2\frac{\partial \cH}{\partial y}}{y\cS(z)^{2}}\cdot {\rm exp}\left(y\left(\frac{z\cS'(z)}{\cS(z)}-{\rm ln}(\cS(z))-\mathcal{H}\right)\right).
$$

\begin{theorem}\label{thm:eulerseries} For all $g\geq 1,$ the orbifold Euler characteristic of the minimal stratum $\M_g(2g-1)$ is given by 
$$\chi(\M_g(2g-1))= (2g-1)^{2g-1} [z^{2g}] \chi(2g-1,z).$$
\end{theorem}

This result simplifies drastically the computation of the Euler characteristic of the minimal strata from the formula shown in \cite{CMZ20} involving a sum over a high number of combinatorial graphs. 
In particular, this makes it feasible to explicitly compute the value of $\chi(\M_g(2g-1))$ for large values of $g$.

Finally, we have a spin analogue of Theorem~\ref{thm:eulerseries}. 
We define  
\[\chi(\M_g(2g-1))^\sp:= \chi(\M_g(2g-1)^\even)-\chi(\M_g(2g-1)^\odd)\]
to be the difference of the orbifold Euler characteristics of the even and odd components of the minimal stratum $\M_g(2g-1)$.

\begin{theorem}\label{thm:eulerseriesspin} If Assumption~\ref{assumption} holds, then for all $g\geq 1,$ we have
\begin{eqnarray*}
\chi(\M_g(2g-1))^\sp=2^{-g}(2g-1)^{(2g-1)} [z^{2g}] \chi^\sp(2g-1,z),
\end{eqnarray*}
where we define:
\begin{eqnarray*}
\cH^\sp(y,z)&=&-\sum_{g\geq 1} (2g-1)!\, \frac{2^{2g-1}}{2^{2g-1}-1} b_g \left(\frac{z\cS(z)}{y}\right)^{2g} \frac{\cS((2g-1)z)}{\cS(z)},\\
\chi^{\sp}(y,z)&=&\frac{y+1-y^2\frac{\partial \cH^\sp}{\partial y}}{y\cS(z)^{2}{\rm cosh}(z/2)^{-1}}\cdot {\rm exp}\left(y\left(\frac{z\cS'(z)}{\cS(z)}-{\rm ln}(\cS(z))-\cH^\sp\right)\right).
\end{eqnarray*}
\end{theorem}
In Table \ref{cap:EulerHolo} we present the orbifold Euler characteristic of the even and odd spin components of minimal strata of abelian differentials in  low genera. 
  
\begin{figure}[h]
$$ \begin{array}{|c|c|c|c|c|c|c|}
\hline  &&&&&& \\ 
g  & 1 & 2 & 3 & 4 & 5 & 6 \\
\hline &&&&&& \\
\chi(\M_g(2g-1)^\even) 
& 0  & 0& -\frac{1}{84} & -\frac{269}{720} & - \frac{693}{40} & 
-\frac{76466}{63}  \\
 &&&&&&\\
\hline &&&&&& \\ 
\chi(\M_g(2g-1)^\odd) 
& -\frac{1}{12}  &  -\frac{1}{40} & -\frac{7}{72} & -\frac{5}{4} & 
-\frac{3933}{110} & -\frac{5841833}{3120}  \\
&&&&&&\\
\hline
\end{array}
$$
\captionof{table}{Euler characteristics of even and odd components of some minimal holomorphic strata.}
\label{cap:EulerHolo}
\end{figure}

\subsection{Outlook and open questions}
A first natural question to ask is how Theorem \ref{th:main} generalizes to arbitrary monomials in $\psi$-classes.
\begin{question}
Given $g,n \geq 0$ and nonnegative integers $e_1, \ldots, e_n$ summing to $2g-3+n$, compute the function
\[
\mathbb{Z}^n \to \mathbb{Q},\quad a \mapsto \int_{\oM_{g,n}} \DR_g(a) \cdot \psi_1^{e_1} \cdots \psi_n^{e_n}\,.
\]
\end{question}
Restricted to vectors $a$ with $|a|=0$, this question was answered in \cite[Theorem 2]{BSSZ}. The proof of this Theorem was achieved by using the semi-infinite wedge formalism and by computing integrals of $\psi$-classes on classical DR cycles as vacuum expectations (see \cite[Theorem 7]{BSSZ}).

While the intersection numbers above are readily computable for many pairs $(g,n)$ using the software \texttt{admcycles} \cite{admcycles}, it proved difficult to guess a formula generalizing the one presented in \cite{BSSZ}. Alternatively one could try to understand how to naturally account for the twisting parameter $k$ in the semi-infinite wedge formalism. To our knowledge, such approach does not exist in the literature yet.  A first step would be to define a ``natural'' operator acting on the infinite wedge space, whose vacuum expectation is given by the right-hand side of~\eqref{eq:drstrata}.


A second direction of study is the asymptotic behaviour of the orbifold Euler characteristic of minimal strata of differentials.
From numerical experiments based on the formula given in Theorem \ref{thm:eulerseries}, we propose the following conjecture.
\begin{conjecture} \label{Conj:Eulercharasymptotic}
For all $g \geq 1$, the orbifold Euler characteristic $\chi(\M_g(2g-1))$ is negative. Moreover, there exist positive constants $A$ and $B$ such that
\[A\frac{(2g-1)!}{(2g-1)^4}\leq -\chi(\M_g(2g-1))\leq B\frac{(2g-1)!}{(2g-1)^3}\]
for all $g$.
\end{conjecture}
Note that this asymptotic growth rate is much higher than the one of $\M_{g,1}$, which was shown in \cite{HaZa} to be $(-1)^g\frac{(2g-1)!}{2(2\pi)^g}$. The conjecture is also a strong indication that, as in the case of $\M_{g,n}$, the cohomology of minimal strata is not spanned by tautological classes. However, in order to formally prove this, one would have to compute the topological Euler characteristic, and not the orbifold one. Note that in the case of $\M_{g,n}$,  it was shown in \cite{HaZa} that the growth rates of the topological and orbifold Euler characteristics agree.

The global geometry of strata of abelian differentials is not well understood. The above conjecture fits into the setting of two natural questions that are still open:
\begin{question}
Is $\M_g(a)$ k\"ahler-hyperbolic?
\end{question}
The analogous result in the case of $\M_{g,n}$ was proven in \cite{mcmullenkahlerhyp}. This property would imply that the sign of the Euler characteristic of $\M_g(a)$ is $(-1)^{\text{dim}(\M_g(a))}$. In particular, since the minimal strata $\M_g(2g-1)$ are always of odd dimension $2g-1$, this prediction is consistent with the first part of Conjecture~\ref{Conj:Eulercharasymptotic}.

\begin{question}
Are the connected components of $\M_g(a)$ classifying spaces for their fundamental groups?
\end{question}
This question was posed as a conjecture by Kontsevich and Zorich. The image of the fundamental group of the connected components of $\M_{g}(a)$ in the mapping class group was described in~\cite{saltercalderon} to be a framed mapping class group. In the same work, it was shown to be a finitely generated infinite index subgroup of the associated mapping class group.  However, the kernel of the natural map to the mapping class group is still ill-understood. Some analogous results appeared also in \cite{Ham}. 
    One approach for \emph{disproving} the conjecture of Kontsevich and Zorich, would be to use Conjecture~\ref{Conj:Eulercharasymptotic} to show that the Euler characteristic of the minimal strata and of their the fundamental groups are different for large values of $g$. This is not possible until more information on the fundamental group of the strata is known.

From numerical experiments, using the formula given in  Theorem \ref{thm:eulerseriesspin},  we also conjecture that the Euler characteristics of the odd and even components are asymptotically equivalent for large values of  $g$. This is implied by the following stronger  conjecture.
\begin{conjecture}
There exist  positive constants $A$ and $B$ such that 
$$
A\frac{g}{2^g}< \frac{\chi(\M_g(2g-1))^\sp}{-\chi(\M_{g}(2g-1))} < B \frac{g^2}{2^g} 
$$
for all $g$.
In particular  $\chi(\M_g(2g-1))^\sp>0$ and 
\[\frac{\chi(\M_g(2g-1)^\even)}{\chi(\M_g(2g-1)^\odd)}\longrightarrow 1\quad \text{for }\ g\longrightarrow \infty.\]
\end{conjecture}

\subsection{Organization of the paper}
We begin the paper by 
recalling in~Section \ref{Sect:DRmultiscale} the definition of double ramification cycles and their relation to the strata of $k$-differentials  as well as the smooth compactification of these strata by the spaces of multi-scale differentials.
In particular, in Section \ref{Sect:spinDR} we present our proposal for the spin double ramification cycles and the corresponding generalization of Conjecture A. 
In Section~\ref{Sect:splitting}, we prove a {\em splitting formula} for $\psi$-classes on double ramification cycles, i.e. a family of relations between $\psi$-classes on double ramification cycles. The splitting formula is then used to show a list of properties of the functions $\cA_g$ in Section \ref{Sect:Agidentities} which uniquely determine these functions. We finish the proof of Theorem~\ref{th:main} by verifying these properties for the right-hand side of equation \eqref{eqn:main}. We conclude the paper by discussing the spin refinements of our results in Section \ref{Sect:spinrefinement}, and the application to the Euler characteristic of minimal strata of differentials in Section \ref{Sect:Eulerchar}.

In Appendix \ref{Sect:Polyproperties} we collect some results about polynomiality properties of some formulas in our paper, which are used in the proofs of Section \ref{Sect:Agidentities}.

\subsection*{Acknowledgements}
We would like to thank Dawei Chen for providing important insights about the spin components and the boundary of the BCGGM compactification, and Alessandro Giachetto, Reinier Kramer, and Danilo Lewa\'nski for early communications on the spin refinement of Chiodo classes. We are also grateful to Amol Aggarwal, David Holmes,
Martin Möller, Rahul Pandharipande, Aaron Pixton, Paolo Rossi,  Johannes Schwab, and Dimitri Zvonkine
for inspiring discussions.

\rcomment{We are deeply grateful to the anonymous referee whose meticulous review and insightful suggestions significantly strengthened our argumentation and enhanced the overall quality of our paper.}

The first author  has been supported by the DFG Research Training Group 2553.

The second author would like to thank the Mathematical Institute of Leiden University who hosted him during part of the research. 

The third author was supported by the Early Postdoc.Mobility grant 184245 and the grant 184613 of the Swiss National Science Foundation and thanks the Max Planck Institute for Mathematics in Bonn for its hospitality.

\section{Double ramification cycles and moduli spaces of multi-scale differentials} \label{Sect:DRmultiscale}

In this section we recall the definitions of the cycles $\DR_g(a)$ and moduli spaces of multi-scale differentials, which will be central in the rest of the paper. In particular we recall the notion of twisted and level graphs and we fix the notation that will be used in the next sections.

\rcomment{Unless stated otherwise, we assume that $k\geq 1$. Most of the results will be established in this context and then some properties will be extended to $k=0$ by using the polynomiality of double ramification cycles.}

\subsection{Twisted and level graphs} Let $g,n \geq 0$ with $2g-2+n>0$. A {\em stable graph} is the datum of $$\Gamma=(V,H,g:V\to \mathbb{Z}_{\geq 0},i:H\to H,\phi:H\to V,H^i\simeq [\![1,n]\!]),$$
where:
\begin{itemize}
\item The function $i$ is an involution of $H$.
\item  The cycles of length $2$ for $i$ are called {\em edges} while the fixed points are called {\em legs}. We fix the identification of the set of legs with the set $[\![1,n]\!]$ of integers from $1$ to $n$.
\item An element of $V$ is called a {\em vertex} and for a half-edge $h$ with $\phi(h)=v$ we say that $h$ is \emph{incident} to $v$. We denote by $H_\Gamma(v)$ the set of half-edges incident to $v$ and by $n(v)$ the {\em valency} of the vertex $v$, i.e. the cardinality of $H_\Gamma(v)$. 
\item For all vertices $v$ we have $2g(v)-2+n(v)>0$.
\item The genus of the graph, defined as
\[
g(\Gamma) = h^1(\Gamma)+\sum_{v\in V} g(v),\text{ with }h^1(\Gamma)=|E|-|V|+1
\]
is equal to $g$.
\item The graph is connected.
\end{itemize}

A stable graph determines a moduli space $\oM_\Gamma=\prod_{v\in V} \oM_{g(v),n(v)}$ and a morphism $\zeta_\Gamma: \oM_\Gamma\to \oM_{g,n}$ (see e.g. \cite[Appendix A]{GP03}).

We fix a value of $k\geq 0$ and a vector $a$ such that $|a|=k(2g-2+n)$.
\begin{definition}[\cite{FP}]
A {\em twist} (compatible with $a$) on a stable graph $\Gamma$ is a function $I:H\to \RR$ satisfying:
\begin{itemize}
\item For all $v\in V$, we have
$$\sum_{h\in \phi^{-1}(v)} I(h) =  k(2g(v)-2+n(v)).$$
\item   If $(h,h')$ is an edge of $\Gamma$, then we have $I(h)=-I(h')$.
\item If $(h_1,h_1')$ and $(h_2,h_2')$ are edges between the same vertices $v$, $v'$, then $I(h_1)\geq 0\Leftrightarrow I(h_2)\geq 0$. In which case we denote $v\geq v'$.
\item The relation $\geq $ defines a partial order on the set of vertices. 
\item The twist at the leg with label $i$ has value $a_i$.
\end{itemize}
Given a twisted graph $(\Gamma,I)$, we define its multiplicity as 
$$m(\Gamma,I)=\prod_{(h,h')\in E(\Gamma)} \sqrt{-I(h)I(h')}\,.$$

\end{definition}

We now introduce the objects that parametrize the boundary component of the moduli space of multi-scale differentials.

Let $(\Gamma,I)$ be a twisted graph. A {\em level function} of depth $L>0$ on $(\Gamma,I)$ is a surjective function $l:V(\Gamma)\to \{0,-1,\ldots,-L\}$ satisfying  $l(v_1)\geq l(v_2)$ if $v_1\geq v_2$, for any $v_1$ and $v_2$ in $V(\Gamma)$,

\begin{definition}\label{def:levelgraph}
A {\em level graph} $\overline{\Gamma}=(\Gamma,I,l)$ is the datum of a twisted graph together with a level function. We denote by ${\rm LG}(g,a)$ the set of level graphs for the vector $a$ and with no horizontal edges (edge between vertices of the same level) and ${\rm LG}_L(g,a)$ the set of such level graphs of depth $L$.

If $\oGamma$ is a level graph in ${\rm LG}_L(g,a)$, and $1\leq i\leq L$, then we denote by $m(\oGamma)^{[i]}$ and $\ell(\oGamma)^{[i]}$ the product and lcm of the twists at all the edges crossing the level-passage between levels $-i+1$ and $i$ \rcommenttwo{(the $i$-th level passage is a horizontal line just above level $-i$)}. Here, the twist of an edge $e=\{h,h'\}$ is the positive integer $|I(h)|=|I(h')|$.
Finally we define $m(\oGamma)=\prod_{i} m(\oGamma)^{[i]}$, and $\ell(\oGamma)=\prod_i \ell(\oGamma)^{[i]}$.
\end{definition}

\subsection{A formula for twisted double ramification cycles} \label{Sec:ConjA}
The double ramification cycle $\DR_g(a)$ on $\oM_{g,n}$ is a cycle compactifying the condition $\omega_\log^{\otimes k} \cong \mathcal{O}_C(a_1 x_1 + \ldots + a_n x_n)$ on $\M_{g,n}$. Many geometric approaches have been proposed to make this statement precise (see \cite[Section 1.6]{HS19} for an overview) and over the last years, all of them have been shown to be equivalent.

A particularly explicit approach is to express the cycle $\DR_g(a)$ in terms of the generators of the tautological ring of $\oM_{g,n}$.
To state the corresponding formula below, we write $\mathcal{G}_{g,n}$ for the set of stable graphs $\Gamma$ of genus $g$ with $n$ legs. 
Fix a vector $a \in \mathbb{Z}^n$ such that the number $k=|a|/(2g-2+n)$ is an integer. Then for a stable graph $\Gamma$, an \emph{admissible weighting} modulo $r$ (with respect to $a$) is a map $w : H(\Gamma) \to \{0, \ldots, r-1\}$ satisfying:
\begin{enumerate}[label=\alph*)]
    \item For every vertex $v \in V(\Gamma)$, we have
    \[
    \sum_{h \in H_\Gamma(v)} w(h) \equiv k(2g-2+n)\ \mathrm{mod}\  r\,.
    \]
    \item For every edge $e=(h,h') \in E(\Gamma)$, we have
    \[
    w(h) + w(h') \equiv 0\ \mathrm{mod}\  r\,.
    \]
    \item For every $i=1, \ldots, n$ and $h_i$ the half-edge associated to the $i$-th marking, we have
    \[
    w(h_i) \equiv a_i\ \mathrm{mod}\  r\,. 
    \]
\end{enumerate}
We write $W_{\Gamma,r}(a)$ for the set of admissible weightings modulo $r$ on $\Gamma$. Then, we define a mixed-degree tautological class $P_g^{r,\bullet}(a)$ on $\oM_{g,n}$ by the formula
\begin{align} \label{eqn:Pixton}
    P_g^{r,\bullet}(a) = r^{2g} \cdot \sum_{\Gamma \in \mathcal{G}_{g,n}} \sum_{w \in W_{\Gamma,r}(a)} \frac{r^{-h^1(\Gamma)}}{|\Aut(\Gamma)|} \xi_{\Gamma *} \mathrm{Cont}_{a,\Gamma, w, r}\,,
\end{align}
where the class $\mathrm{Cont}_{a,\Gamma, w, r}$ on the domain $\oM_\Gamma$ of the gluing map $\xi_\Gamma$ is defined by
\begin{align*}
    \mathrm{Cont}_{a,\Gamma, w, r} = &\prod_{v \in V(\Gamma)}  \exp({- \sum_{m \geq 1} (-1)^{m-1} \frac{B_{m+1}(k/r)}{m(m+1)} \kappa_m(v)})\\
    \cdot & \prod_{i=1, \ldots, n}  \exp({\sum_{m \geq 1} (-1)^{m-1} \frac{B_{m+1}(a_i/r)}{m(m+1)} \psi_i^m})\\
    \cdot &\prod_{\substack{e\in E(\Gamma)\\e=(h,h')}} \frac{1-\exp({\sum_{m \geq 1} (-1)^{m-1} \frac{B_{m+1}(w(h)/r)}{m(m+1)} (\psi_h^m-(-\psi_{h'})^m})}{\psi_h + \psi_{h'}}\,.
\end{align*}
For sufficiently large values of $r$, this class is a mixed degree tautological class on $\oM_{g,n}$ whose coefficients are polynomials in $r$ (see~\cite{JPPZ17}). We denote by $P_g^\bullet(a)$ the evaluation of this class by substituting $r=0$ in its coefficients. Then we define the double ramification cycle as the multiple
$$
\DR_g(a) = P_g^{g}(a)
$$
of the Chow degree $g$ part of this cycle.

If $k>0$ and $a\notin k\ZZ_{>0}^n$, the double ramification cycle $\DR_g(a)$ has a geometric interpretation as the weighted fundamental class of the locus of twisted $k$-differentials in $\oM_{g,n}$. This relationship was first conjectured in \cite{FP} for $k=1$ and in \cite{Sch18} for $k>1$ and was proved in \cite{HS19, BHPSS20}. We recall this expression here.

\begin{definition}
A \emph{simple star graph} $(\Gamma,I)$ is a twisted graph such that 
\begin{itemize}
    \item the vertices of $\Gamma$ consist of a unique \emph{central vertex} $v_0$ and a set $V_{\rm Out}$ of outlying vertices, such that all edges of $\Gamma$ go from $v_0$ to one of the outlying vertices,
    \item the twists at half-edges adjacent to outlying vertices are positive and divisible by $k$. 
\end{itemize}
\end{definition}








Two simple star graphs are called isomorphic if there is an isomorphism of the underlying stable graphs sending the central vertex to the central vertex and respecting the respective twists. 
Denote by ${\rm Star}_g(a)$ the set of isomorphism classes of simple star graphs for the given genus $g$ and vector $a$. 

Given $(\Gamma,I)$ and $v \in V(\Gamma)$ a vertex, we denote by $I(v)$ the vector of integers  $I(h)$ indexed by the half-edges $h$ at $v$. Then we define
\begin{equation}
    \oM_{\Gamma, I} = \oM_{g(v_0)}(I(v_0)) \times \prod_{v \in V_{\rm Out}} \oM_{g(v)}(I(v)/k)\,.
\end{equation}
The space $\oM_{\Gamma,I}$ is naturally a closed substack of $\oM_\Gamma$.

With this notation in place, we are ready to state the relationship between the strata of $k$-differentials and the double ramification cycle. 
\begin{theorem}[\cite{HS19, BHPSS20}] \label{Thm:ConjA}
Let $g,n\geq 0,$ $k>0$, and $a \in \mathbb{Z}^n$ such that $|a|=k(2g-2+n)$ and such that $a \notin  k\ZZ_{>0}^n$. Then we have
\begin{equation} \label{eqn:ConjA}
\DR_g(a) = \sum_{(\Gamma,I)\in {\rm Star}_g(a)} \frac{m(\Gamma,I)}{k^{|V_{\rm Out}|}|{\rm Aut}(\Gamma,I)| } \cdot \zeta_{\Gamma *}[\oM_{\Gamma,I}].    
\end{equation}
\end{theorem}


\subsection{The spin double ramification cycle} \label{Sect:spinDR}
In this section, we give an explicit proposal for the cycle $\DR_g^\sp$ from Assumption \ref{assumption} as well as a conjecture generalizing Conjecture A of~\cite{FP} and~\cite{Sch18} to the spin setting. The proposal for  $\DR_g^\sp$  is inspired by recent work \cite{spinhurwitz} on spin Hurwitz numbers. Its formula consists in a small modification of the original formulas for the double ramification cycle presented above.

First we define the spin analogue of Pixton's class above. Let  $r$ be an even number, and $a \in \mathbb{Z}^n$ a vector of odd integers such that the number $k=|a|/(2g-2+n)$ is an odd integer.  Given a stable graph $\Gamma$, we write $W_{\Gamma,r}(a)^\mathrm{odd}$ for the set of admissible weightings modulo $r$ on $\Gamma$ satisfying that all numbers $w(h)$, $h \in H(\Gamma)$, are odd. It is easy to see that
\begin{itemize}
    \item for $\Gamma$ a tree there exists a unique admissible weighting modulo $r$ and it is automatically odd,
    \item for arbitrary $\Gamma$, the cardinality of $W_{\Gamma,r}(a)^\mathrm{odd}$ is precisely $(r/2)^{h^1(\Gamma)}$.
\end{itemize}
Then, we define a mixed-degree tautological class $P_g^{r,\sp,\bullet}(a)$ on $\oM_{g,n}$ by the formula
\begin{align} \label{eqn:spinPixton}
    P_g^{r,\sp,\bullet}(a) = r^{2g}2^{-g} \cdot \sum_{\Gamma \in \mathcal{G}_{g,n}} \sum_{w \in W_{\Gamma,r}(a)^{\mathrm{odd}}} \frac{(r/2)^{-h^1(\Gamma)}}{|\Aut(\Gamma)|} \xi_{\Gamma *} \mathrm{Cont}_{a,\Gamma, w, r}\,,
\end{align}
where the class $\mathrm{Cont}_{a,\Gamma, w, r}$ is the class defined in the previous section. Using similar arguments as presented in \cite[Appendix A]{JPPZ17}, one sees that the cycle $P_g^{r,\sp,\bullet}(a)$ become polynomial in $r$ for sufficiently large (even) $r$. Denote by $P_g^{g,\sp}(a)$ the value of the degree $g$ part of this polynomial at $r=0$. Then we propose that the spin double ramification cycle should be given by
\begin{equation} \label{eqn:DRspinformula}
    \DR_g^\sp(a) = P_g^{g,\sp}(a)\,.
\end{equation}
As mentioned before, the formula above is inspired by the paper \cite{spinhurwitz}. There, the authors introduce a spin Chiodo class $C^\vartheta(r,k;\widetilde a)$, where the modified vector $\widetilde a$ is defined by the convention $2 \widetilde a_i +1 = a_i$. The comparison to the class $ P_g^{r,\sp,\bullet}(a)$ is
\[
C^\vartheta(r,k;\widetilde a) = r^{-1}  P_g^{r,\sp,\bullet}(a)\,,
\]
see \cite[Proposition 9.7, Proposition 9.20]{spinhurwitz}. We used an implementation by Danilo Lewa\'nski of the formula for $\DR_g^\sp(a)$ in the software \texttt{admcycles} \cite{admcycles} to verify that for the cases
\[
(g=1, n \leq 4), (g=2, n \leq 2), (g=3,n=1)
\]
the top-$\psi$ intersection number of the corresponding cycle agrees with the prediction from Theorem \ref{th:spin} for \emph{all} admissible input vectors $a$.

Concerning the statement in Assumption~\ref{assumption} that the cycle \eqref{eqn:DRspinformula} is a polynomial  in $a$ of degree $2g$, it will be proved in the paper \cite{PZ21} by Pixton and Zagier, alongside the corresponding statement for the classical double ramification cycle. Furthermore, the statement $\pi^* \DR_g(a)^\sp = \DR_g(a,k)^\sp$ can be shown from the formula above using a short direct computation.

As for the remaining parts of Assumption \ref{assumption}, namely the relations between top $\psi_1$-intersection numbers of $\DR_g^\sp(a)$ and $[\oM_g(a)]^\sp$ (for suitable $k,a$), these would be implied by a spin variant of Conjecture A, which we present in the following. For this, we denote by ${\rm Star}_g(a)^\odd$ the set of star graphs with odd values of twists at all half-edges. Given such a pair $(\Gamma, I) \in {\rm Star}_g(a)^\odd$ we write 
\begin{equation}
    [\oM_{\Gamma, I}]^\sp = [\oM_{g(v_0)}(I(v_0))]^\sp \otimes \prod_{v \in V_{\rm Out}} [\oM_{g(v)}(I(v)/k)]^\sp\,.
\end{equation}
Then we conjecture the following relationship between the spin cycles of the strata of $k$-differentials and the spin double ramification cycle above.


\begin{conjecture}\label{conj:DRspinformula} Let $g,n\geq 0,$ $k>0$, and $a \in \mathbb{Z}^n$ such that $|a|=k(2g-2+n)$ satisfying that $a \notin  k\ZZ_{>0}^n$ and that all entries of $a$ and $k$ are odd. Then we have
\begin{equation} \label{eqn:spinConjA}
  \sum_{(\Gamma,I)\in {\rm Star}_g(a)^\odd} \frac{m(\Gamma,I)}{k^{|V_{\rm Out}|}|{\rm Aut}(\Gamma,I)| } \cdot \zeta_{\Gamma *}[\oM_{\Gamma,I}]^\sp  = \DR_g^\sp(a)\,.
\end{equation}
\end{conjecture}
\begin{proposition} \label{prop:spinConjAimpliesAssumption}
Assuming Conjecture \ref{conj:DRspinformula} is true, the class $\DR_g^\sp(a)$ satisfies properties (3) and (4) of Assumption~\ref{assumption}.
\end{proposition}
\begin{proof}
To show property (3) we intersect both sides of the equality \eqref{eqn:spinConjA} with the monomial $P$ in classes $\psi_i$ for which $a_i \notin  k \mathbb{Z}_{>0}$. This last assumption on the $a_i$ forces all markings associated to classes $\psi_i$ appearing in $P$ to be on the central vertex of any graph $(\Gamma,I)\in {\rm Star}_g(a)^\odd$. If $\Gamma$ is non-trivial, a short computation shows that the class $[\oM_{g(v_0)}(I(v_0))]^\sp$ inserted at the central vertex has dimension strictly less than the degree $2g-3+n$ of $P$. Thus for the intersection number of $P$ with the left-hand side of \eqref{eqn:spinConjA}, the only surviving term comes from the trivial star graph. This gives the left-hand side of the equation from property (3) and thus finishes the proof. Finally, property (4) can be shown by similar arguments as appear in \cite[Section 2.2]{Sau3}.
\end{proof}
Apart from theoretical evidence for Conjecture \ref{conj:DRspinformula} which we discuss below, there are a few cases where it can be verified directly. 

In genus $g=0$ it is trivial: a spin bundle $\mathcal{L}$ on $\mathbb{P}^1$ is always even and so both sides of the conjecture equal the fundamental class of $\oM_{0,n}$.

In genus $g=1$, for a point $(C, x_1, \ldots, x_n) \in \mathcal{M}_1(a)$ with $a$ odd, the spin bundle 
$$L = \mathcal{O}_C\left(\frac{a_1-k}{2}x_1 + \ldots + \frac{a_n-k}{2}x_n \right)$$
on an elliptic curve $C$ is a $2$-torsion line bundle. We can distinguish two cases: for $L=\mathcal{O}_C$ it has precisely one section and thus it is odd, whereas for $L \neq \mathcal{O}_C$ it has no section and thus is even. Therefore the odd components of $\mathcal{M}_1(a)$ are precisely the loci where $L$ is trivial and thus they are given by $\mathcal{M}_{1}((a_1-k)/2, \ldots, (a_n-k)/2)$. Using that $[\rcomment{\oM_1}(a)]^\sp = [\rcomment{\oM_1}(a)]-2\cdot[\rcomment{\oM_1}(a)]^{\mathrm{odd}}$ we then have
\begin{align} \label{eqn:M1aspin}
[\oM_{1}(a_1, \ldots, a_n)]^\sp = [\oM_{1}(a_1, \ldots, a_n)] -2 \cdot [\oM_{1}(\frac{a_1-k}{2}, \ldots, \frac{a_n-k}{2})]\,.
\end{align}
The cycles $[\oM_1(a)]$ themselves are determined by the original Conjecture A.
Thus in genus $g=1$ one can explicitly compute both sides of the equality \eqref{eqn:spinConjA} since in the sum over star graphs, we only need to compute the cycles $[\oM_0(a)]^\sp = [\oM_0(a)]$ and $[\oM_1(a)]^\sp$.
A short calculation (similar to one appearing in \cite[Section 3.3.1]{Sch18}) then verifies Conjecture \ref{conj:DRspinformula}. 

Finally, using the software \texttt{admcycles} \cite{admcycles}, the conjecture can be checked for $g=2$, $a=(5,-1)$. In the sum over star graphs $(\Gamma, I)$, the only genus $2$ spin cycles that appear as factors of $[\oM_{\Gamma,I}]^\sp$ are $[\oM_2(3)]^\sp$ and  $[\oM_2(5,-1)]^\sp$. In these cases, it follows from the classification of the connected components of $\oM_2(3)$ and $\oM_2(5,-1)$  presented in \cite{KZ, Boissy} that the spaces have precisely two components: a hyperelliptic one of odd parity and a non-hyperelliptic one of even parity. This implies that
\[
[\oM_2(3)]^\sp = [\oM_2(3)] -2 \cdot  [\overline{\mathrm{Hyp}}_{2,1}], \quad 
[\oM_2(5,-1)]^\sp = [\oM_2(5,-1)] -2 \cdot  [\overline{\mathrm{Hyp}}_{2,2}]\,,
\]
where $\overline{\mathrm{Hyp}}_{2,1} \subseteq \oM_{2,1}$ and $\overline{\mathrm{Hyp}}_{2,2} \subseteq \oM_{2,2}$ are the loci of (hyperelliptic) curves with one or two marked Weierstrass points. Their fundamental classes can be computed using the methods described in \cite{SvZ}. With these inputs, all terms appearing in Conjecture \ref{conj:DRspinformula} can be computed as tautological classes and it is verified that the claimed equality follows from known tautological relations.

\rcomment{After the first publication of the present paper, Wong \cite{Wong_spin} presented an algorithm for calculating the classes $[\oM_{g}(a)]^{\rm spin}$ in many examples. Implementing this in the software \cite{diffstrata}, he checked Conjecture \ref{conj:DRspinformula} in several further cases in genus $g=2,3,4$ (see \cite[Section 7.3]{Wong_spin}).}

A natural path towards a proof of Conjecture~\ref{conj:DRspinformula} should be as follows:
\begin{enumerate}[label=(\roman*)]
    \item Consider a compactification $\overline{\mathcal{S}}_{g,n}$ of the moduli space 
    $$\mathcal{S}_{g,n} = \{(C, x_1, \ldots, x_n, \mathcal{L}) : \mathcal{L}^{\otimes 2} \cong \omega_C \}$$ 
    of spin curves (see \cite{Cornalba,C08, AJ03}). As seen before, this space decomposes into odd and even components, according to the parity of $\mathcal{L}$, and we define
    \[[\overline{\mathcal{S}}_{g,n}]^\sp = [\overline{\mathcal{S}}_{g,n}]^{\mathrm{even}} -  [\overline{\mathcal{S}}_{g,n}]^{\mathrm{odd}}\,.\]
    \item On the other hand, we can consider the universal Picard stack $\mathcal{P}ic_{g,n}$ over $\overline{\mathcal{S}}_{g,n}$ parameterizing a curve $C$ in $\overline{\mathcal{S}}_{g,n}$ together with a line bundle $\mathcal{M}$ on $C$.
    Let $e \subseteq \mathcal{P}ic_{g,n}$ be the codimension $g$ locus where $\mathcal{M}=\mathcal{O}_C$ is trivial, and denote by $\overline{e}$ its closure. 
    
    Given $k$ odd and a vector $a$ of odd integers summing to $k(2g-2+n)$, we can consider the section $\sigma_a : \overline{\mathcal{S}}_{g,n} \to \mathcal{P}ic_{g,n}$ given by
    $$
     \mathcal{M} = \mathcal{L}^\vee \otimes \omega_C^{\otimes(-k+1)/2} \otimes \mathcal{O}_C\left(\sum_{i=1}^n \frac{a_i-k}{2}x_i\right).
    $$
    We define the cycle 
    $$\widehat{\DR}_g(a) = \sigma^* [\overline{e}] \in A^{g}(\overline{\mathcal{S}}_{g,n})\,.$$
    Intuitively, this cycle compactifies the condition 
    $$
     \mathcal{L} \cong \omega_C^{\otimes(-k+1)/2} \otimes \mathcal{O}_C\left(\sum_{i=1}^n \frac{a_i-k}{2}x_i\right).
    $$    
    that we used in the definition of the spin refinement of the strata of $k$-differentials. We then expect that the machinery of the paper \cite{HS19} can be adapted to show that the cycle
    \begin{equation*} \label{eqn:FpushDR}
        F_* \left(\widehat{\DR}_g(a) \cdot [\overline{\mathcal{S}}_{g,n}]^\sp \right)
    \end{equation*}
    is precisely given by the linear combination of cycles on the left-hand side of \eqref{eqn:spinConjA}, with the numbers $m(\Gamma, I)$ being related to intersection multiplicities of the section $\sigma$ with the locus $\overline{e}$.


        \item On the other hand, generalizing the machinery of~\cite{BHPSS20}, the class $[\overline{e}] \in A^g(\mathcal{P}ic_{g,n})$ should have a formula in the tautological ring of $\mathcal{P}ic_{g,n}$ which is structurally very similar to Pixton's formula for the classical double ramification cycle. Pulling this back under $\sigma$ we obtain a formula for $\widehat{\DR}_g(a)$ and a computation analogous to \cite[Proposition 9.18]{spinhurwitz} should imply
    \begin{equation*} \label{eqn:FpushDR*}
    F_* \left(\widehat{\DR}_g(a) \cdot [\overline{\mathcal{S}}_{g,n}^\sp] \right) = \DR_g(a)^\sp,
    \end{equation*}
    which would conclude the proof.
\end{enumerate}


We plan to pursue these directions further in the future.

\subsection{The moduli space of multi-scale differentials.} \label{Sec:multiscaled}

We recall now the definition and some features of the moduli space of multi-scale differentials, a smooth modular compactification of strata of differentials with normal crossing boundary divisors.

\subsubsection{The incidence variety compactification}
Given a vector $P=(p_1,\ldots,p_n)$ of positive integers, we denote by $\oOm_{g,n}^k(P)$ the vector bundle on $\oM_{g,n}$ defined as $$\pi_{*} \omega_{\rm log}^{\otimes k}(p_1\sigma_1+\ldots,p_n\sigma_n)$$ where we recall that $\pi:\oC_{g,n}\to \oM_{g,n}$ is the universal curve and the $\sigma_i:\oM_{g,n}\to \oC_{g,n}$ are the sections associated to the markings. If $a$ is a vector of integers satisfying $|a|=k(2g-2+n)$, then for a sufficiently large value of $P$, we define $\Om_{g}(a)$ as the sub-space of $\oOm_{g,n}^k(P)$ of smooth curves with a $k$-differentials with singularities prescribed by $a$. We denote by $\PP\oOm_{g}(a)$ the Zariski closure of $\PP\Om_{g}(a)$ in $\PP\oOm_{g,n}^k(P)$ (the geometry of this space does not depend on the choice of $P$). This  space is the { incidence variety compactification of $\M_g(a)$} (which is not smooth). See \cite{BCGGM2} for the description of the boundary of this compactification.

\subsubsection{The moduli space of multi-scale differentials} The  incidence variety admits a modular desingularization $\PP\Xi_g(a)\to \PP\oOm_{g}(a)$, the {\em moduli space of multi-scale differentials}, which is defined in~\cite{BCGGM3} in the case $k=1$ and in~\cite{CMZarea} for higher values of $k$. We denote by $p:\PP\Xi_g(a)\to \oM_g(a)$ the morphism defined as the composition of the desingularization and the forgetful morphism $\PP\oOm_g(a)\to \oM_{g}(a)$. This morphism restricts to an isomorphism over $\M_g(a)$. 

The boundary $\PP\Xi_g(a)\setminus \M_g(a)$ is a simple normal crossing divisor. The boundary components of this space $\PP\Xi_g(a)$ are parametrized by level graphs, as introduced in Definition \ref{def:levelgraph}. The sum of the depth and the number of horizontal edges of a level graph is the codimension of the associated boundary component. In particular divisors are indexed by either level graphs of depth 0 with 1 (horizontal) edge, or level graphs in ${\rm LG}_1(g,a)$. We will denote by $\PP\Xi(\oGamma)\subseteq \PP\Xi_g(a)$ the boundary component corresponding to the level graph $\oGamma$. 

If $\oGamma$ is a graph in ${\rm LG}_L(g,a)$, then the associated boundary component $\PP\Xi(\oGamma)$ parametrizes {\em multi-scale differentials} compatible with $\oGamma$ modulo an equivalence relation given by the action of a torus, called {\em level rotation torus}. We explain now the details that we will need for the case of $k=1$ in Section \ref{sec:spinboundary}, where we investigate the spin parity of some boundary divisors. We refer to \cite{BCGGM3} for the full description of this compactification. 

An {\em abelian twisted differential} is a tuple $(\omega_v)_{v\in V(\oGamma)}$ of abelian differentials on each component of the stable curve defined by $\oGamma$. We say that an abelian twisted differential  is compatible with the level graph $\oGamma$ if each $\omega_v$ has vanishing orders prescribed by the twists of $\oGamma$ and if it satisfies the {\em global residue condition} (GRC)  defined in~\cite[Definition 1.4]{BCGGM2}. 
An abelian multi-scale differential compatible with $\oGamma$ is a  twisted differential  compatible with $\oGamma$ together with the choice of a prong-matching for each non-horizontal edge of $\Gamma$. A {\em prong-matching} on a vertical edge of $\oGamma$ is a \rcommenttwo{cyclic-order-reversing} bijection between the horizontal directions defined the twisted differentials on the components of the stable curve at the upper and lower ends of the edge.

If we denote by $L(\oGamma)$ the depth of $\oGamma$, the equivalence relation is given by the \rcomment{exponential} action of $\CC^{L(\oGamma)}$ 
which rescales the differentials on each level and simultaneously acts by fractional Dehn twists on the prong-matchings \rcomment{(see e.g.\cite[Sec. 3.3]{CMZarea}))}. The subgroup $\mathrm{Tw}_\oGamma\subseteq\CC^{L(\oGamma)}$ acting trivially on the prong-matchings is called the Twist group, and so the action factors through the quotient $T_\oGamma=\CC^{L(\oGamma)}/\mathrm{Tw}_\oGamma$, which is called the level rotation torus. Two multi-scale differentials compatible with $\oGamma$ are defined to be equivalent, if they differ by the action of $T_\oGamma$. 

Now that we have explained which objects are parametrized by boundary divisors, we would like to recall that the boundary divisors $\PP\Xi(\oGamma)$ are commensurable (see \cite[Prop. 4.4]{CMZ20} for the details) to the product of lower dimensional spaces of multi-scale differentials, one for each level. We will denote by $\PP\Xi(\oGamma)^{[i]}$ the moduli-space of multi-scale differentials defined by the $i$-th level of $\oGamma$. More specifically, $\PP\Xi(\oGamma)^{[i]}$ is the multi-scale differential compactification of the stratum 

\[  \PP\Big(\prod_{v\in \ell^{-1}(i)} \Om_{g(v)}(I(v))\Big)^{\frakR_i}\]
where \rcomment{we used $\frakR_i$ to indicate the locus inside of the projectivized bundle obtained by imposing the conditions on residues induced by the GRC.} 

 \rcommentthree{Recall that the datum of a Riemann surface $X$  together with a meromorphic $k$-differential $q$ determines a cyclic $k$-fold cover $\pi \colon \widehat{X} \to X$
such that $\pi^* q = \omega^k$ is the $k$-power of an abelian differential, see e.g. \cite[Sec. 2.1]{BCGGM2}. This construction is called the $k$-canonical cover construction and allows to consider a stratum of $k$-differentials as a subset of a stratum of abelian differentials. In the case of $k>1$, the multi-scale compactification of a given stratum of $k$-differentials is defined (up to stacky issues) as the closure of the embedding of the stratum into the corresponding stratum of abelian differentials obtained after the $k$-canonical cover construction. As explained in more detail in \cite[Lemma 7.1]{CMS},  a multi-scale $k$-differential can be hence defined as an abelian multi-scale differential together with an order $k$ automorphism $\tau$ for which the multi-scale differential is an eigenvector with primitive $k$-th root of unity eigenvalue, such that the prong-matching is equivariant with respect to the action of $\tau$ and such that $\tau$ acts as a permutation of the marked points as it naturally should.}

\rcommentthree{To a multi-scale $k$-differential, we can associate its quotient by $\tau$, which consists of a stable curve with a tuple of $k$-differentials on each components. The quotient by $\tau$ of the level graph of the multi-scale $k$-differential can be canonically made into an enhanced level graph for which the vanishing orders of the tuple of $k$-differentials are prescribed by the twists. We call this graph a $k$-level graph associated to the multi-scale $k$-differential and it is indeed a level graph as in Definition \ref{def:levelgraph}. Viceversa, a cover of $k$-level graphs $\widehat{\Gamma}\to \overline{\Gamma}$ together with a tuple of $k$-differentials $((C_i,[\eta_i])_{i=0,\dots,-L}$ on each components with vanishing orders prescribed by the twists of $\overline{\Gamma}$ gives rise to a unique abelian twisted differential on the canonical $k$-cover of the $k$-level graph. If the abelian twisted differential satisfies the GRC, this, up to the chioce of a prong-matching equivalence class, gives then rise to a unique multi-scale $k$-differential and hence we say that $((C_i,[\eta_i])_{i=0,\dots,-L}$ is compatible with $\widehat{\Gamma}\to \overline{\Gamma}$.}

\rcommentthree{For an enhanced $k$-level graph $\overline{\Gamma}$ with $L+1$ levels, we define the associated boundary divisor $\PP\Xi(\oGamma)$ which parametrises  multi-scale $k$-differentials associated to a tuple $((C_i,[\eta_i])_{i=0,\dots,-L},[\sigma])$, where $(C_i,[\eta_i])_{i=0,\dots,-L}$ is a tuple of $k$-differentials compatible with a cover of $k$-level graphs $\widehat{\Gamma}\to \overline{\Gamma}$, and $[\sigma]$ is \rcommentfour{the equivalence class of} an equivariant prong-matching  on the corresponding abelian $k$-cover multi-scale differential.}

The generalization to $k$-differentials of level strata can be done analogously to the $k=1$ situation (see \cite[Sec. 7.3]{CMS}).

\begin{remark}\label{rem:GRC}
We will not recall the GRC here but we mention the following facts that will be used below.
\begin{itemize}
    \item If every connected component of the subgraph above level $i$ contains at least one half-edge with a twist not in $k\ZZ_{>0}$, then the residue condition $\frakR_i$  is empty.
    \item If $k=1$, $n=1$, and $\oGamma$ is a 2-level graph of compact type, i.e. with only separating edges, then the GRC states that all residues at the poles of the  multi-scale differential on level $-1$ are trivial.
\end{itemize}
\end{remark}
\begin{remark}\label{rem:numberPMEC}
\rcommentthree{For an enhanced $k$-level graph $\overline{\Gamma}$ with $2$ levels, the number of possible prong-matching equivalence classes is}
\[ \frac{\prod_{e\in E(\overline{\Gamma})}\widehat{\kappa}_e}{\LCM_{e\in E(\overline{\Gamma})}\left(\widehat{\kappa}_e\right)},\quad  \widehat{\kappa}_e=\frac{\kappa_e}{\gcd(\kappa_e,k)}
\]
\rcommentthree{where $\kappa_e$ is the absolute value of the twist at the edge $e$. For $k=1$ this was shown in \cite[Sec. 3.4]{CMZ20} and for $k>1$ this follows from \cite[Lemma.7.4]{CMS}, where it was shown that the number of prong-matchings of the abelian cover contained in the linear submanifold defined by the stratum of $k$-differentials is exactly the product of all $\widehat{\kappa}_e$.}
\end{remark}

\subsubsection{Intersection theory on the space of multi-scale differentials} As $\PP\oOm_{g}(a)$ is subspace of a projective bundle, it has a tautologial bundle $\cO(-1)$. We denote by $\eta \in A^1(\PP\Xi_g(a))$ the Chern class of the pull-back of this line bundle. We will also denote by $\psi_i\in A^1(\PP\Xi_g(a))$ the pull-back of the $\psi$-classes from $\oM_g(a)$. 

First of all we recall a relation between the tautological class $\eta$ and the $\psi$-classes. 
In the case $k=1$, the following statement was proven in  \cite[Theorem~6(1)]{Sau} for the incidence variety compactification and in \cite[Prop. 8.2]{CMZ20} for the space of multi-scale differentials. The case of general $k \geq 1$ is treated in~\cite[Theorem~3.12]{Sau3} in the context of the incidence variety compactification, but the proof directly shows the following statement about the space of multi-scale differentials.
\begin{proposition} \label{prop:Adrienrel} For all  $\oGamma \in {{\rm LG}}{_1}(g,a),$ and all irreducible components $D$ of $\PP\Xi(\oGamma)$, there exists a rational number $m(D)$ satisfying the following conditions: 
\begin{itemize}
    \item for all $1\leq i\leq n$, we have
\begin{equation}\label{eq:xirel}
\eta    = a_i \psi_{i}\,\, -  \sum_{\oGamma \in \tensor[_i]{{\rm LG}}{_1}(g,a)} \sum_{D}
m(D) \cdot \zeta_{D *} \left([D^{[0]}]
\otimes[D^{[-1]}]\right),
\end{equation}
where the set $\tensor[_i]{{\rm LG}}{_1}(g,a)$ consists of the two-level graphs $\oGamma$ without horizontal edges where the $i$-th half-edge is on level $-1$, and $D^{[i]}$ are the $i$-th level strata induced by $D$.
\item if $k=1$, or if all vertices of $\oGamma$ have a half-edge of order not divisible by $k$, then $m(D)=m(\oGamma)/|\Aut(\oGamma)|$.
\end{itemize}
\end{proposition}

Finally we recall the relation between the class of the subspace cut out by a residue condition and other tautological classes in the case of $k=1$. As before, the following relation was proven in  \cite[Prop.~7.6]{Sau} for the incidence variety compactification and in \cite[Prop. 8.3]{CMZ20} for the space of multi-scale differentials.

\begin{proposition} \label{prop:AdrienR}
Let $k=1$ and let $\PP\Xi^\frakR_g(a)$ be a stratum cut out by a set of residue conditions $\frakR$.  Assume that $\PP\Xi^\frakR_g(a)$ is a divisor in the stratum $\PP\Xi^{\frakR_0}_g(a)$  cut out by a smaller set $\frakR_0$ of residue conditions obtained by removing one condition from $\frakR$.
Then we have the following relation
\begin{equation}\label{eq:GRCremove}
    [\PP\Xi^\frakR_g(a)] =\ - \eta \,\, - \sum_{\oGamma \in {\rm LG}_{1,\frakR}(g,a)}
\ell(\oGamma) \cdot [\PP\Xi(\oGamma)]\,
\end{equation}
in $A^1(\PP\Xi^{\frakR_0}_g(a))$,
where ${\rm LG}_{1,\frakR}(g,a)$ is the union of the sets of  non-horizontal two-level graphs where the
GRC on top level induced by~$\frakR$ does no longer introduce an extra condition
and  the set of non-horizontal two-level graphs where all the half-edges  involved in the condition forming $\frakR \setminus \frakR_0$ go to lower
level.
\end{proposition}
In order to convert the previous $\ell(\oGamma)$ coefficients when considering  the commensurability diagram involving a level graph  $\PP\Xi(\oGamma)$ without horizontal edges and the product of all its level strata  $\PP\Xi(\oGamma)^{[i]}$, we recall that we have to multiply by  
\begin{equation}\label{eq:convfactor}
    \frac{m(\oGamma)}{|{\rm Aut}(\oGamma)|\ell(\oGamma)}.
\end{equation}
This was proven in \cite[Lemma 4.5]{CMZ20}.

\section{Splitting formulas for \texorpdfstring{$\psi$}{psi}-classes on double ramification cycles} \label{Sect:splitting}

The purpose of this section is to prove the {\em splitting formula for $\psi$-classes on double ramification cycles}, i.e. a family of relations between $\psi$-classes on double ramification cycles (see Proposition~\ref{prop:psiDRformula} below). Theorem~\ref{Thm:ConjA} will play a key-role in the proof as it allows to reduce intersection with double ramification cycles to intersection with strata of $k$-differentials.

To state this formula, we introduce the following  notation:  ${\rm LG}_1^2(g,a)$ is the set of level graphs with exactly 2 vertices $v_0$ of level 0 and $v_{-1}$ of level $-1$, and no horizontal edges. 
\begin{proposition}[Splitting formula for $\psi$-classes on double ramification cycles] \label{prop:psiDRformula}
Let $g,n \geq 0$ with $2g-2+n>0$, let $a=(a_1, \ldots, a_n) \in \mathbb{Z}^n$ with $\sum_{i=1}^n a_i = k(2g-2+n)$.  Then for any two different elements $s,t \in \{1, \ldots, n\}$, we have
\begin{align}
    & (a_s \psi_s - a_t \psi_t) \DR_g(a_1, \ldots, a_n) \nonumber\\
    &= \sum_{(\oGamma,I)\in {\rm LG}_1^2(g,a)}\!\!\!\! f_{s,t}(\oGamma) \frac{m(\oGamma)}{|{\rm Aut}(\oGamma)|}   \cdot \zeta_{\Gamma *}\left(\DR_{g_0}( I(v_0)) \otimes \DR_{g_{-1}}( I(v_{-1}))\right). \label{eqn:psiDRformula}
\end{align}
Here, the function $f_{s,t}$ is defined by 
$$
f_{s,t}(\oGamma)=\left\{\begin{array}{cl}
    0 & \text{if $s$ and $t$ belong to the same vertex,} \\
    1 & \text{if $s$ belong to $v_{-1}$ and $t$ to $v_0$,}\\
    -1 & \text{otherwise}.
\end{array}\right.
$$
\end{proposition}

\subsection{Splitting formulas for strata of $k$-differentials} \label{Sec:kdiffsplitting}

In order to prove Proposition \ref{prop:psiDRformula}, we will show that a similar statement holds when we replace $\DR_g(a)$ by $[\oM_g(a)]$. An analogous splitting formula was proved in the case $k=1$ in~\cite{Sau}.

\begin{lemma}\label{Lem:kdiffsplitting}
Let $g,n\geq 0,$ $k>0$, and  $a\in \ZZ^n$ with $|a|=k(2g-2+n)$. Let $s\neq t \in \{1,\ldots,n\}$ be indices such that $k$ does not divide $a_s$ or $a_t$. Then, we have the following relation:
\begin{align}
    &(a_s \psi_s - a_t \psi_t)\cdot [\oM_g(a)] \nonumber \\
    &= \sum_{(\oGamma,I)\in {\rm LG}_1^2(g,a)}\!\!\!\! f_{s,t}(\oGamma) \frac{m(\oGamma)}{|{\rm Aut}(\oGamma)|} \cdot  \zeta_{\Gamma *}\left([\oM_{g_0}( I(v_0))] \otimes [\oM_{g_{-1}}( I(v_{-1}))]\right) \label{eqn:kdiffsplitting}
\end{align}
\end{lemma}

In order to prove this lemma we will work with the multi-scale compactification of $\M_g(a)$ introduced in Section \ref{Sec:multiscaled}.

\begin{proof}[Proof of Lemma~\ref{Lem:kdiffsplitting}]
 Recall that we denote by $p:\PP\Xi_g(a)\to \oM_g(a)$ the morphism defined as the composition of the desingularization and the forgetful morphism to $\oM_{g}(a)$.
 By the projection formula, the class $\psi_s[\oM_{g}(a)]$ is equal to $p_*(p^*\psi_i)$. Thus we now study the intersection theory on  $\PP\Xi_{g}(a)$. 
 
 Using Proposition \ref{prop:Adrienrel}, we can write $(a_s\psi_s-\eta)$ or $(a_t\psi_t-\eta)$ on $\PP\Xi_{g}(a)$ as a sum on irreducible components of the boundary divisors indexed by the graphs in ${\rm LG}_1(g,a)$ such that $s$ or $t$ respectively  are adjacent to a level $-1$ vertex. Moreover, the coefficients of an irreducible divisor appearing in the expression of $(a_s\psi_s-\eta)$ or $(a_t\psi_t-\eta)$ are equal. Thus, we write $(a_s\psi_s-a_t\psi_t)$ as $(a_s\psi_s-\eta)-(a_t\psi_t-\eta)$ to express it as a sum over irreducible components of the boundary divisors indexed by ${\rm LG}_1(g,a)$, where $s$ and $t$ are adjacent to distinct levels. 

If $\oGamma\in {\rm LG}_1(g,a)$ is a level graph with at least 2 vertices of level 0, then the fibers of $p$ restricted to the corresponding divisors have positive dimension. Thus such a graph contributes trivially to the expression of $(a_s \psi_s - a_t \psi_t)\cdot [\oM_g(a)]$.

Moreover, for all graphs $\oGamma\in {\rm LG}_1(g,a)$ involved in the expression of $(a_s\psi_s-a_t\psi_t)$, either $s$ or $t$ is adjacent to this unique vertex of level 0. Thus the global residue condition defined is empty as $k$ does not divide $a_s$, nor $a_t$ (see Remark \ref{rem:GRC}). Therefore, if such a graph has at least 2 vertices of level $-1$, then the fibers of $p$ restricted to the corresponding divisor have positive dimension and once again such graphs contribute trivially to the expression of $(a_s \psi_s - a_t \psi_t)\cdot [\oM_g(a)]$.

Hence $(a_s\psi_s-a_t\psi_t)[\oM_g(a)]$ is expressed as a sum over level graphs in ${\rm LG}_1^2(g,a)$ and the coefficient for each such graph is exactly  $f_{s,t}(\oGamma)m(\oGamma)/|\Aut(\oGamma)|$, again by Proposition \ref{prop:Adrienrel} .
\end{proof}

\subsection{Proof of Proposition \ref{prop:psiDRformula}}
Finally, we can combine the results of Sections  \ref{Sec:ConjA} and \ref{Sec:kdiffsplitting} to prove Proposition \ref{prop:psiDRformula}. As a first step, we observe that it suffices to show the proposition for particular input vectors $a$.
\begin{lemma} \label{Lem:psisplittingrestricta}
Let $g,n \geq 0$ such that $2g-2+n>0$. Then Proposition \ref{prop:psiDRformula} is true if and only if it is true for $k>0$ and $a \in (\mathbb{Z} \setminus k \mathbb{Z})^n$.
\end{lemma}
\begin{proof}
First note that  for valid input vectors $a \in \Lambda$ contained in the sublattice $\Lambda \subset \mathbb{Z}^n$ of vectors whose sum is divisible by $2g-2+n$, the parameter $k=k(a)$ can be computed from $a$. Thus the statement in Proposition \ref{prop:psiDRformula} is purely a statement about cycles depending on these input vectors. The crucial observation is that the two sides of the equality \eqref{eqn:psiDRformula} in Proposition \ref{prop:psiDRformula} are polynomial in the entries of $a$ by \cite{PZ21} (for the left-hand side) and Lemma \ref{Lem:DRsplittingpolynomiality} (for the right-hand side). Thus we conclude by observing that the statement of the proposition is vacuous for $n=1$ and that for $n \geq 2$, the set of vectors $a \in \Lambda$ with $k=k(a)>0$ and $a \in (\mathbb{Z} \setminus k \mathbb{Z})^n$ is Zariski-dense in $\mathbb{R}^n$ and so any polynomial equality satisfied for such $a$ is satisfied everywhere.
\end{proof}
\begin{remark} Note that for a level graph in ${\rm LG}_1^2(g,a)$, the level structure is uniquely determined by the underlying twisted graph (and the automorphisms of the level graph are automorphisms of the underlying twisted graph). Thus in the following proof, we will consider these objects as twisted graphs.
\end{remark}
\begin{proof}[Proof of Proposition \ref{prop:psiDRformula}]
By Lemma \ref{Lem:psisplittingrestricta} it suffices to show equality \eqref{eqn:psiDRformula} for vectors $a \in \mathbb{Z}^n$ summing to some integer multiple $k(2g-2+n)$ of $2g-2+n$ such that $a\in (\ZZ\setminus k\ZZ)^n$ . The overall strategy of our proof is as follows:
\begin{itemize}
    \item[Step 1] In the left-hand side of \eqref{eqn:psiDRformula}, we use Theorem \ref{Thm:ConjA} to replace the double ramification cycle by a sum over star graphs with strata of $k$-differentials and strata of $1$-differentials at the vertices. By the assumption on $a$, all markings must go to the central vertex here.
    \item[Step 2] Then we use the splitting formula from Lemma \ref{Lem:kdiffsplitting} on the central vertex (again made possible by the assumption on $a$) to replace it by a sum over $2$-twisted graphs glued into this vertex. At this stage, we have expressed the left-hand side of \eqref{eqn:psiDRformula} as a sum over graphs \eqref{eqn:graphsum_middle} with appropriate twists on all half-edges and strata of $k$- and $1$-differentials in the vertices.
    \item[Step 3] In the final step we interpret the top and bottom part of the graph \eqref{eqn:graphsum_middle} as simple star graphs and recombine the corresponding sub-summations into double ramification cycles using Theorem \ref{Thm:ConjA} in the opposite of the direction used before. We are left with a sum over $2$-twisted graphs with double ramification cycles at the vertices, obtaining the right-hand side of \eqref{eqn:psiDRformula}.
\end{itemize}
The outline above contains all relevant mathematical ideas going into the proof, and a reader satisfied by this outline may safely skip to the next section. The remainder of the argument below will focus on controlling the combinatorics and multiplicities involved with the above manipulations of graph sums.

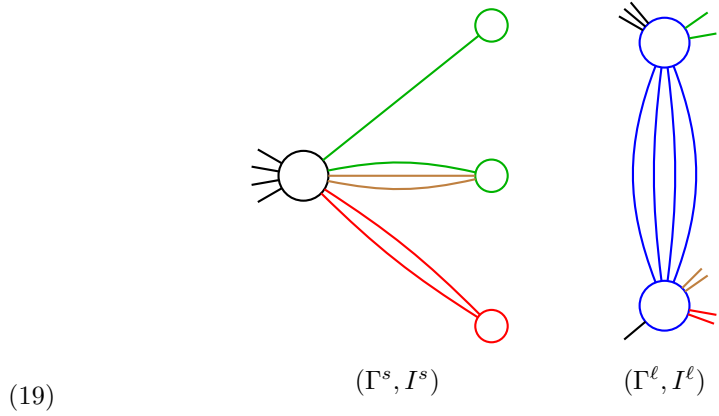
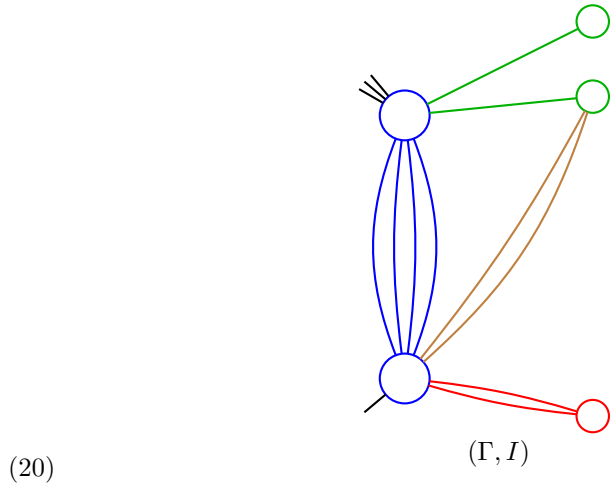
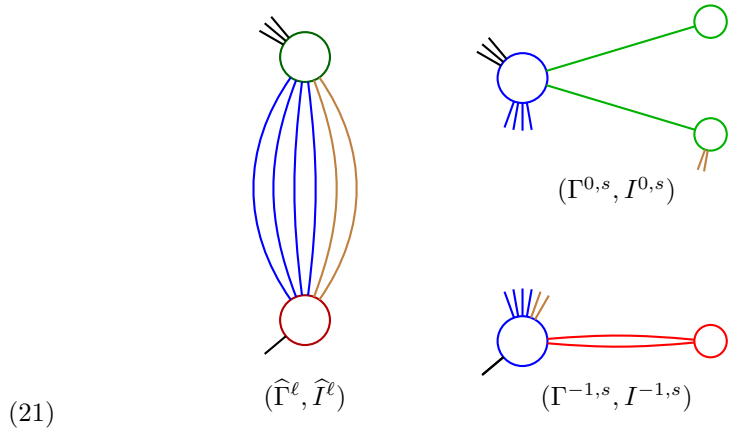
\begin{equation} \label{eqn:graphsum_top}
\begin{tikzpicture}[scale=0.5, vert/.style={circle,draw,font=\Large,scale=1.3, outer sep=0}, cvert/.style={circle,draw,font=\Large,scale=2, outer sep=0}, thick]
\node [cvert] (A) at (0,0) {};
\node [vert, black!30!green] (B) at (5,4) {};
\node [vert,black!30!green] (C) at (5,0) {};
\node [vert,red] (D) at (5,-4) {};

\draw [-,black!30!green] (A) to (B);
\draw [-,black!30!green] (A) to [bend left=12] (C);
\draw [-,brown] (A) to  (C);
\draw [-,brown] (A) to [bend right=12] (C);
\draw [-,red] (A) to [bend left=6] (D);
\draw [-,red] (A) to [bend right=6] (D);

\draw (A) -- (150:1.4cm);
\draw (A) -- (170:1.4cm);
\draw (A) -- (190:1.4cm);
\draw (A) -- (210:1.4cm);

\node at (2.5,-5.5) {$(\Gamma^s, I^s)$};
\end{tikzpicture}    
\quad  \quad \quad \quad
\begin{tikzpicture}[scale=0.5, vert/.style={circle,draw,font=\Large,scale=1.3, outer sep=0}, cvert/.style={circle,draw,font=\Large,scale=2, outer sep=0}, thick]
\node [cvert, blue] (A) at (0,3.5) {};
\node [cvert, blue] (B) at (0,-3.5) {};

\draw [-, blue] (A) to [bend left=21] (B);
\draw [-, blue] (A) to [bend left=7] (B);
\draw [-, blue] (A) to [bend left=-7] (B);
\draw [-, blue] (A) to [bend left=-21] (B);

\draw (A) -- ++(150:1.4cm);
\draw (A) -- ++(140:1.4cm);
\draw (A) -- ++(130:1.4cm);
\draw (B) -- ++(220:1.4cm);
\draw[black!30!green] (A) -- ++(10:1.4cm);
\draw[black!30!green] (A) -- ++(35:1.4cm);
\draw[brown] (B) -- ++(45:1.4cm);
\draw[brown] (B) -- ++(35:1.4cm);
\draw[red] (B) -- ++(-10:1.4cm);
\draw[red] (B) -- ++(-20:1.4cm);

\node at (0,-5.5) {$(\Gamma^\ell, I^\ell)$};
\end{tikzpicture}    
\end{equation}

\begin{equation} \label{eqn:graphsum_middle}
\begin{tikzpicture}[scale=0.5, vert/.style={circle,draw,font=\Large,scale=1.3, outer sep=0}, cvert/.style={circle,draw,font=\Large,scale=2, outer sep=0}, thick]
\node [cvert, blue] (A) at (0,3.5) {};
\node [cvert, blue] (B) at (0,-3.5) {};
\node [vert, black!30!green] (C) at (5,4) {};
\node [vert, black!30!green] (D) at (5,6) {};
\node [vert, red] (E) at (5,-4.5) {};

\draw [-, blue] (A) to [bend left=21] (B);
\draw [-, blue] (A) to [bend left=7] (B);
\draw [-, blue] (A) to [bend left=-7] (B);
\draw [-, blue] (A) to [bend left=-21] (B);
\draw [-, black!30!green] (A) to (C);
\draw [-, black!30!green] (A) to (D);
\draw [-, brown] (B) to [bend right=15] (C);
\draw [-, brown] (B) to [bend right=5] (C);
\draw [-, red] (B) to [bend left=5] (E);
\draw [-, red] (B) to [bend right=5] (E);

\draw (A) -- ++(150:1.4cm);
\draw (A) -- ++(140:1.4cm);
\draw (A) -- ++(130:1.4cm);
\draw (B) -- ++(220:1.4cm);

\node at (2.5,-5.5) {$(\Gamma, I)$};
\end{tikzpicture}    
\end{equation}

\begin{equation} \label{eqn:graphsum_bottom}
\begin{tikzpicture}[scale=0.5, vert/.style={circle,draw,font=\Large,scale=1.3, outer sep=0}, cvert/.style={circle,draw,font=\Large,scale=2, outer sep=0}, thick]
\node [cvert, black!60!green] (A) at (0,3.5) {};
\node [cvert, black!30!red] (B) at (0,-3.5) {};

\draw [-, brown] (A) to [bend left=35] (B);
\draw [-, brown] (A) to [bend left=21] (B);
\draw [-, blue] (A) to [bend left=7] (B);
\draw [-, blue] (A) to [bend left=-7] (B);
\draw [-, blue] (A) to [bend left=-21] (B);
\draw [-, blue] (A) to [bend left=-35] (B);

\draw (A) -- ++(150:1.4cm);
\draw (A) -- ++(140:1.4cm);
\draw (A) -- ++(130:1.4cm);
\draw (B) -- ++(220:1.4cm);

\node at (0,-5.5) {$(\widehat{\Gamma}^\ell, \widehat{I}^\ell)$};
\end{tikzpicture}    
\quad  \quad \quad \quad
\begin{tikzpicture}[scale=0.5, vert/.style={circle,draw,font=\Large,scale=1.3, outer sep=0}, cvert/.style={circle,draw,font=\Large,scale=2, outer sep=0}, thick]
\node [cvert, blue] (A) at (0,3.5) {};
\node [cvert, blue] (B) at (0,-3.5) {};
\node [vert, black!30!green] (C) at (5,2) {};
\node [vert, black!30!green] (D) at (5,5) {};
\node [vert, red] (E) at (5,-3.5) {};

\draw [-, black!30!green] (A) to (C);
\draw [-, black!30!green] (A) to (D);
\draw [-, red] (B) to [bend left=5] (E);
\draw [-, red] (B) to [bend right=5] (E);

\draw (B) -- ++(220:1.4cm);

\node at (2.5,0.5) {$(\Gamma^{0,s}, I^{0,s})$};
\node at (2.5,-5) {$(\Gamma^{-1,s}, I^{-1,s})$};

\draw (A) -- ++(150:1.4cm);
\draw (A) -- ++(140:1.4cm);
\draw (A) -- ++(130:1.4cm);
\draw (B) -- ++(220:1.4cm);

\draw[blue] (A) -- ++(250:1.4cm);
\draw[blue] (A) -- ++(260:1.4cm);
\draw[blue] (A) -- ++(270:1.4cm);
\draw[blue] (A) -- ++(280:1.4cm);

\draw[brown] (C) -- ++(250:1cm);
\draw[brown] (C) -- ++(260:1cm);

\draw[brown] (B) -- ++(60:1.4cm);
\draw[brown] (B) -- ++(70:1.4cm);
\draw[blue] (B) -- ++(80:1.4cm);
\draw[blue] (B) -- ++(90:1.4cm);
\draw[blue] (B) -- ++(100:1.4cm);
\draw[blue] (B) -- ++(110:1.4cm);

\end{tikzpicture}    
\end{equation}
\noindent \textbf{Step 1}: 
We start with the expression $(a_s \psi_s - a_t \psi_t) \DR_g(a)$. By the assumption that $a\in (\ZZ\setminus k\ZZ)^n$, we can apply Theorem \ref{Thm:ConjA} to replace the double ramification cycle by a sum over star graphs $(\Gamma^s, I^s)$ with strata of $k$-differentials/$1$-differentials at the central and outlying vertices (see the left side of \eqref{eqn:graphsum_top}). Moreover, all (legs corresponding to) markings must be on the central vertex since all of their weights are in $\ZZ\setminus k\ZZ$ (see the definition of a simple star graph). In particular, the markings $s,t$ belong to the same stratum $\oM_{g(v_0)}(I_s(v_0))$ of $k$-differentials that is glued into $v_0$ and the term $a_s \psi_s - a_t \psi_t$ can be pulled back to that vertex.

\noindent \textbf{Step 2}:
Again using the assumption $a\in (\ZZ\setminus k \mathbb{Z})^n$ combined with the fact that all entries of $I_s(v_0)$ are negative, we can then apply Lemma \ref{Lem:kdiffsplitting} to replace the expression 
\[
(a_s \psi_s - a_t \psi_t) \cdot \oM_{g(v_0)}(a, I_s(v_0))
\]
on the central vertex by a sum over $2$-twisted graphs $(\Gamma^\ell, I^\ell)$ glued into this vertex  (see the right side of \eqref{eqn:graphsum_top}).

Overall, we have by now written $(a_s \psi_s - a_t \psi_t) \DR_g(a)$ as a sum indexed by simple star graphs $(\Gamma^s, I^s)$ and $2$-twisted graphs $(\Gamma^\ell, I^\ell)$ gluable into the central vertex of $\Gamma^s$. The individual summands are described by the twisted graphs $(\Gamma, I)$ depicted in \eqref{eqn:graphsum_middle} obtained by performing this gluing operation and, up to rational coefficients which we make more precise below, they are given by a pushforward under $\xi_\Gamma$ of strata of $k$-differentials and $1$-differentials at the two central, respectively the outlying vertices of $\Gamma$. Let us make a couple of observations at this point:
\begin{itemize}
    \item Due to the factor $f_{s,t}(\Gamma^\ell, I^\ell)$ appearing in Lemma \ref{Lem:kdiffsplitting}, any $(\Gamma, I)$ appearing with nonzero coefficient will have to satisfy that $s,t$ appear at two different vertices (the two central vertices depicted in blue). Using this observation, we see that we can uniquely\footnote{Strictly speaking the unique reconstruction requires the additional data of an identification of the edges of $\Gamma^s$ with the edges in $\Gamma$ incident to the non-central vertices. We will suppress this detail for now and return to it during the last part of the proof, when we match the precise coefficients and multiplicities of all involved terms.} reconstruct $(\Gamma^s, I^s)$ and $(\Gamma^\ell, I^\ell)$ given $(\Gamma,I)$: we obtain $\Gamma^s$ by contracting all edges between the two central components, and we obtain $\Gamma^\ell$ by cutting all other edges and removing all other vertices.
    \item The outlying vertices of $\Gamma$ (which are in natural bijection to the outlying vertices of $\Gamma^s$) can naturally be placed at the top or bottom of $\Gamma$ (and coloured green and red). The rule is that any vertex connecting \emph{only} to the bottom central vertex will be at the bottom (red) and all others will be at the top (green). This is motivated by the fact that the twist $I$ at all half-edges at the outlying vertices is positive, and thus in the partial order on $V(\Gamma)$ coming from the twist $I$ they should be strictly above any other vertex they connect to. 
    \item We color all the edges of $\Gamma$ by
    \begin{itemize}
        \item blue for the edges between the two central vertices,
        \item green and red for the edges on the top and bottom level between central and outlying vertices,
        \item brown for the edges connecting the lower central vertex of $\Gamma$ with an upper outlying vertex.
    \end{itemize}
\end{itemize}
\noindent \textbf{Step 3}:
Finally, we regroup the terms in the sum over $(\Gamma,I)$ to obtain the right-hand side of \eqref{eqn:psiDRformula}. For this, given $(\Gamma, I)$ we compute a triple of twisted graphs as follows:
\begin{itemize}
    \item We obtain a $2$-twisted graph $(\widehat{\Gamma}^\ell, \widehat{I}^\ell)$ by \emph{contracting} all red and green edges (and preserving the brown edges).
    \item We obtain simple star graphs $(\Gamma^{i,s}, I^{i,s})$, $i=0,-1$, by \emph{cutting} all blue and brown edges and taking the connected components of the top and bottom levels.
\end{itemize}
From the construction, it is clear that $(\Gamma,I)$ can be reconstructed by gluing the graphs $(\Gamma^{i,s}, I^{i,s})$ into the vertices of $(\widehat{\Gamma}^\ell, \widehat{I}^\ell)$, so that again Step 3 defined a combinatorial bijection.

The graphs $(\widehat{\Gamma}^\ell, \widehat{I}^\ell)$ appearing like this are precisely the indices of the sum on the  right-hand side of \eqref{eqn:psiDRformula}. On the other hand, the double ramification cycles appearing in this right-hand side can be replaced by sums over star graphs by Theorem \ref{Thm:ConjA}. Indeed, since $a_s, a_t$ are assumed to be not divisible by $k$ and since $s,t$ go to opposite sides of $\widehat{\Gamma}^\ell$, the assumptions of Theorem \ref{Thm:ConjA} are satisfied and the simple star graphs that appear are precisely of the form $(\Gamma^{i,s}, I^{i,s})$.

\noindent \textbf{Comparison of coefficients}:
At this point we have described how to use Theorem \ref{Thm:ConjA} and Lemma \ref{Lem:kdiffsplitting} to expand the left-hand side of \eqref{eqn:psiDRformula} into a graph sum (with insertions being strata of $k$-differentials and $1$-differentials) and how to regroup this sum using Theorem \ref{Thm:ConjA} to obtain the right-hand side of \eqref{eqn:psiDRformula}. We established these expansions and regroupings on the levels of the involved combinatorial objects (i.e. twisted graphs), but it remains to be verified that in the final step, all summands appear with the correct rational coefficients. Looking at the relevant formulas \eqref{eqn:psiDRformula}, \eqref{eqn:ConjA} and \eqref{eqn:kdiffsplitting} most of the factors are easily matched:\footnote{Instead of performing Step 3 in the direction described above (regrouping terms $(\Gamma,I)$), the reader might find it easier to go in the opposite direction and expand the right-hand side of \eqref{eqn:psiDRformula} into a graph sum. With this interpretation, the equations \eqref{eqn:comparemfac},\eqref{eqn:comparekfac}, \eqref{eqn:compareffac}, \eqref{eqn:compareautfac} below show how the coefficient of the graph sum term $(\Gamma,I)$ from Steps 1,2 agrees with the coefficient obtained from this inverse of Step 3.}

For the multiplicities $m(\Gamma,I)$ we see that the union of half-edges (and twists) of $\Gamma^s, \Gamma^\ell$ is in natural correspondence to the union of half-edges (and twists) of $(\widehat{\Gamma}^\ell, \widehat{I}^\ell)$ and the $(\Gamma^{i,s}, I^{i,s})$, and thus
    \begin{equation} \label{eqn:comparemfac}
        m(\Gamma^s,I^s) \cdot m(\Gamma^\ell, I^\ell) = m(\widehat{\Gamma}^\ell, \widehat{I}^\ell) \cdot m(\Gamma^{0,s}, I^{0,s}) \cdot m(\Gamma^{-1,s}, I^{-1,s})\,.
    \end{equation}
    
Since the outlying vertices of $\Gamma^s$ are the union of the outlying vertices of the $\Gamma^{i,s}$, we have
    \begin{equation} \label{eqn:comparekfac}
        k^{|V_{\mathrm{Out}}(\Gamma^s)|} =  k^{|V_{\mathrm{Out}}(\Gamma^{0,s})|}\cdot k^{|V_{\mathrm{Out}}(\Gamma^{-1,s})|}\,.
    \end{equation}

As the levels in $\Gamma^\ell, \widehat{\Gamma}^\ell$ to which the markings $s,t$ go are the same, we also have
    \begin{equation} \label{eqn:compareffac}
        f_{s,t}(\Gamma^\ell, I^\ell) = f_{s,t}(\widehat{\Gamma}^\ell, \widehat{I}^\ell)\,.
    \end{equation}

It remains to match the automorphism factors. As might be expected, this is the most tricky part of the argument, and the naive equality of products of sizes of automorphism groups in fact fails. The reason is that implicit in Steps 1 to 3 we have assumed e.g. an identification between half-edges at some vertex and the legs of a graph that is glued into this vertex. 

To make precise statements here, it is advantageous to reformulate the graph sums in \eqref{eqn:psiDRformula}, \eqref{eqn:ConjA} and \eqref{eqn:kdiffsplitting} in terms of sums over twisted graphs $(\Gamma,I)$ together with a bijective numbering $o: E(\Gamma) \to [[1, \ldots, e]]$ of all involved edges, where $e = |E(\Gamma')|$. The symmetric group $S_e$ operates transitively on all such numberings and the stabilizer of an isomorphism class of a numbered graph is precisely the automorphism group of the underlying twisted graph. Using the Orbit-Stabilizer Theorem from group theory one checks that formulas \eqref{eqn:psiDRformula}, \eqref{eqn:ConjA} and \eqref{eqn:kdiffsplitting} remain valid under replacing
\begin{itemize}
    \item the sums over isomorphism classes of $(\Gamma,I)$ satisfying the respective properties with sums over isomorphism classes of $(\Gamma, I, o: E(\Gamma) \xrightarrow{\sim} [[1, \ldots, e]]$),
    \item replacing the factors $1/|\mathrm{Aut}(\Gamma,I)|$ by $1/e!$.
\end{itemize}
With this insight, we can essentially finish the argument.
To set notation, let $e_{\mathrm{bl}}, e_{\mathrm{br}}, e_{\mathrm{g}}, e_{\mathrm{r}}$ be the numbers of blue, brown, green and red edges in the pictures \eqref{eqn:graphsum_top}, \eqref{eqn:graphsum_middle} and  \eqref{eqn:graphsum_bottom}. 
Then passing to the numbered version of formulas \eqref{eqn:psiDRformula}, \eqref{eqn:ConjA} and \eqref{eqn:kdiffsplitting} has the following effects on the coefficients of $(\Gamma,I)$ in \eqref{eqn:graphsum_middle}:
\begin{itemize}
    \item In \eqref{eqn:graphsum_top} we can assume to have a numbering on the $e_{\mathrm{g}}+e_{\mathrm{br}}+e_{\mathrm{r}}$ edges of $\Gamma^s$ and legs of $\Gamma^\ell$ and that the gluing respects this ordering. Similarly we have a numbering on the $e_{\mathrm{bl}}$ edges of $\Gamma^\ell$. Overall, the glued graph $(\Gamma,I)$ appears with a coefficient of $$\frac{1}{(e_{\mathrm{g}}+e_{\mathrm{br}}+e_{\mathrm{r}})! e_{\mathrm{bl}}!}$$
    and inherits two numberings (on the green, brown and red edges and on the blue edges).
    \item In \eqref{eqn:graphsum_bottom} we can have a numbering on the $e_{\mathrm{bl}}+e_{\mathrm{br}}$ edges of $\widehat{\Gamma}^\ell$ and legs of $\Gamma^{i,s}$ with the gluing respecting this ordering. Similarly we have a numbering on the $e_{\mathrm{g}}$ edges of $\Gamma^{0,s}$ and the $e_{\mathrm{r}}$ edges of $\Gamma^{-1,s}$. Overall, the glued graph $(\Gamma,I)$ appears with a coefficient of $$\frac{1}{(e_{\mathrm{bl}}+e_{\mathrm{br}})! e_{\mathrm{g}}! e_{\mathrm{r}}!}$$
    and inherits three numberings (on the blue and brown, the green and the red edges).
\end{itemize}
To conclude we must not just compare the coefficients mentioned above but also with how many different numberings each $(\Gamma,I)$ can appear. To do this, we use the following result:

\noindent \textbf{Fact} : Given sets $M_1, \ldots, M_u$ of sizes $m_1, \ldots, m_u$, the map from the set of orderings of $M=M_1 \sqcup \cdots \sqcup M_u$ to the product of the set of orderings on each $M_i$ taking the induced order on the subsets $M_i \subset M$ has fibres of size
\[
\binom{m_1 + \ldots + m_u}{m_1, \ldots, m_u}\,.
\]

Thus given a graph $(\Gamma,I)$ with four orderings on its edges (one for each colour), there are $\binom{e_{\mathrm{g}}+e_{\mathrm{br}}+e_{\mathrm{r}}}{e_{\mathrm{g}},e_{\mathrm{br}},e_{\mathrm{r}}}$ orderings in \eqref{eqn:graphsum_top} inducing the given four orders, and $\binom{e_{\mathrm{bl}}+e_{\mathrm{br}}}{e_{\mathrm{bl}},e_{\mathrm{br}}}$ orderings from \eqref{eqn:graphsum_bottom}. Overall, the desired equality of the coefficients in \eqref{eqn:psiDRformula} then follows from the identity 
\begin{equation} \label{eqn:compareautfac}
\frac{1}{(e_{\mathrm{g}}+e_{\mathrm{br}}+e_{\mathrm{r}})! e_{\mathrm{bl}}!} \binom{e_{\mathrm{g}}+e_{\mathrm{br}}+e_{\mathrm{r}}}{e_{\mathrm{g}},e_{\mathrm{br}},e_{\mathrm{r}}} = \frac{1}{(e_{\mathrm{bl}}+e_{\mathrm{br}})! e_{\mathrm{g}}! e_{\mathrm{r}}!}\binom{e_{\mathrm{bl}}+e_{\mathrm{br}}}{e_{\mathrm{bl}},e_{\mathrm{br}}}\,. 
\end{equation}
\end{proof}

\section{Identities satisfied by \texorpdfstring{$\cA_{g}$}{Ag}} \label{Sect:Agidentities}

In this section we use Proposition~\ref{prop:psiDRformula} to determine three identities satisfied by the functions $\cA_g$. We show that these identities determine the functions $\cA_g$ to prove Theorem~\ref{th:main}.

\begin{lemma}\label{lem:identities} The functions $\cA_{g}=\cA_g(a_1, \ldots, a_n)$ are polynomials in the entries $a_i$, of total degree at most $2g$ and symmetric in the arguments $a_2, \ldots, a_n$. Moreover, they satisfy the following identities:
\begin{eqnarray}
\cA_{g}(a_1,\ldots,a_n, k) &=& \cA_{g}(a_1,\ldots,a_n)\quad \text{ for }k=\frac{a_1+\ldots + a_n}{2g-2+n}, \label{eqn:Agid1}\\
a_1 \cdot \cA_g(a_1,\ldots,a_n,0)\,\, + && \!\!\!\!\!\!\!\!\!\!\!\!\! \sum_{i>1} (a_i-k)  \cA_{g}(\ldots,a_i-k,\ldots) \label{eqn:Agid2}\\
\nonumber &=& \frac{1}{2} \sum_{j=0}^k j(k-j)\cdot \cA_{g-1}(a_1,\ldots,a_n,-j,j-k)\\
 \quad  \text{ for }k=\frac{a_1+\ldots + a_n}{2g-2+n+1},&& \nonumber\\
\cA_{g}|_{|a|=0}&=&[z^{2g}] \frac{\prod_{i=1}^n \cS(a_iz)}{\cS(z)}. \label{eqn:Agid3}
\end{eqnarray}
\end{lemma}

\begin{proof}
The polynomiality and degree bound for $\cA_g$ follow from \cite{PZ21} and the symmetry in $a_2, \ldots, a_n$ follows from the $S_n$-equivariance of the double ramification cycle in its entries and the resulting $S_{n-1}$-invariance of the definition \eqref{eqn:Agdef} of $\cA_g$.

The identity~\eqref{eqn:Agid1} follows from the fact that the class $\DR_g(a_1,\ldots,a_n,k)$ is the pull-back of $\DR_g(a_1,\ldots,a_n)$ under the forgetful morphism of the last marking (this follows from Invariance II of \cite{BHPSS20}). The identity~\eqref{eqn:Agid3} was proved in~\cite[Theorem 1]{BSSZ}. Thus it remains to prove the identity~\eqref{eqn:Agid2}. We rewrite:
$$a_1 \cA_{g}(a_1,\ldots, a_n, 0)=\int_{\oM_{g,n+1}} \psi_1^{2g-3+n}\cdot  (a_1\,\psi_1- 0 \cdot \psi_{n+1})\DR_g(a_1,\ldots,a_n,0).$$
Using Proposition~\ref{prop:psiDRformula}, we can express the right hand side as a sum on graphs with $2$ vertices, with markings $1$ and $n+1$ being on different vertices. For dimension reasons, the only graphs contributing non-trivially are the graphs such that the vertex that does not contain $1$ (and hence \emph{does} contain $n+1$) is of genus $0$ and has exactly $3$ half-edges. It occurs in 2 family of cases: 
\begin{itemize}
    \item either the $(n+1)$st marking is on a vertex with another marking $i\neq 1$, connected to the other vertex by 1 edge;
    \item or the $(n+1)$st marking is on a vertex with no other marking,  connected to the other vertex by 2 edges.
\end{itemize}
Keeping only these two contributions gives the second relation.
\end{proof}
Next, we prove that the properties above are sufficient to characterize $\cA_g$.
\begin{proposition}\label{prop:constraint}
All functions $\cA_{g}$ are uniquely determined by $\cA_0$, $\cA_1$, and the properties in Lemma~\ref{lem:identities}.  
\end{proposition}
The proof of the previous Proposition uses in a substantial way the polynomiality and symmetry properties of $\cA_g$. Before we begin the proof, we need some technical lemmas concerning such symmetric polynomial functions.
\newcommand{\sfunc}{g}
\begin{lemma}\label{Lem:symmfunc}
Let $\sfunc: \mathbb{Z}^n \to \mathbb{Q}$ be a polynomial function of total degree (at most) $D$ in the variables $y_1, \ldots, y_n$, which is symmetric in $y_2, \ldots, y_n$. Then there exist coefficients
\[
c_{d,\mu} \in \mathbb{Q} \text{ for }0 \leq d \leq D \text{ and } \mu = (m_1, \ldots, m_\ell) \text{ a partition of size at most }D-d
\]
such that
\begin{equation} \label{eqn:symmfunc}
    \sfunc(y) = \sum_{d,\mu} c_{d,\mu} \cdot y_1^d e_\mu(y)
\end{equation}
where $e_\mu(y)$ is the elementary symmetric polynomial defined by:
\begin{eqnarray*}
\prod_{i=1}^n(X+y_i)&=&\sum_{i=0}^n e_{n-i}(y) X^i\\
e_\mu&=&\prod_{i=1}^\ell e_{m_i}.
\end{eqnarray*}
For $n>D$, the coefficients $c_{d,\mu}$ in the representation \eqref{eqn:symmfunc} are unique. In other words, the functions $y_1^d e_\mu(y)$ form a basis of the space of all functions $\sfunc$ as above.

Furthermore, recalling that $e_m(y_1, \ldots, y_n)=0$ for $n<m$, the function $\sfunc$ with representation \eqref{eqn:symmfunc} satisfies
\begin{equation} \label{eqn:symmfuncrest}
\sfunc(y_1, \ldots, y_{n-1},0) = \sum_{d,\mu} c_{d,\mu} \cdot y_1^d e_\mu(y_1, \ldots, y_{n-1})\,.
\end{equation}
\end{lemma}
\begin{proof}
The set of functions $g$ as above forms the degree at most $D$ part of the ring
\[Q = \mathbb{Q}[y_1] \otimes_{\mathbb{Q}} \mathbb{Q}[y_2, \ldots, y_n]^{S_{n-1}}\,,\]
which, by the fundamental theorem of symmetric polynomials, has a canonical basis given by $y_1^d e_\mu(y_2, \ldots, y_n)$ where all parts of $\mu$ are of size at most $n-1$. We want to show that the system $\{y_1^d e_\mu(y) : d, \mu\}$ is a generating set of $Q$, for $y=(y_1,\dots,y_n)$. Since the system is closed under multiplication and contains $y_1$, it suffices to show that all functions $g= e_m(y_2, \ldots, y_n)$ have a representation \eqref{eqn:symmfunc}. This easily follows by induction on $m$ using the fact that
\[
e_m(y) - e_m(y_2, \ldots, y_n) = y_1 \cdot h \text{ for }h \in Q \text{ of degree at most }m-1\,.
\]
For $n>D$ we have that automatically all partitions $\mu$ of size at most $D$ have that all parts are bounded by $n-1$. Thus the functions $y_1^d e_\mu(y_2, \ldots, y_n)$ with $d + |\mu| \leq D$ form a basis of $Q_{\leq D}$. But since the number of such pairs $d,\mu$ equals the number of elements $y_1^d e_\mu(y)$ and since these generate $Q_{\leq D}$ by the first part of the proof, they form a basis as desired. The last statement \eqref{eqn:symmfuncrest} of the lemma is immediate from the definition of the elementary symmetric polynomials.
\end{proof}

\begin{lemma} \label{Lem:annoyingsums}
Let $n \geq 1$, $\ell \geq 0$ and $\mu=(m_1, \ldots, m_\ell)$ a partition. Let furthermore $y_1, \ldots, y_n, k$ be formal variables. Then we have
\begin{align} 
    e_\mu(y_1, \ldots, y_n, -k) &= e_\mu + O(k)\,, \label{eqn:annoyingsum1}\\
    \sum_{i=1}^n y_i e_\mu(y_1, \ldots, y_i-k, \ldots, y_n) &= e_1 e_\mu - |\mu| k e_\mu + O(k^2)\,,\label{eqn:annoyingsum2}
\end{align}
where on the right-hand sides $e_1 = e_1(y)$, $e_\mu=e_\mu(y)$.\footnote{Here the notations $O(k)$, $O(k^2)$ stand for the sets of elements of the polynomial ring $\mathbb{Q}[y_1, \ldots, y_n,k]$ divisible by $k, k^2$, respectively.}
\end{lemma}
\begin{proof}
The statement \eqref{eqn:annoyingsum1} follows immediately from the fact that
\[
e_m(y_1, \ldots, y_n,-k) = e_m(y_1, \ldots, y_n) -ke_{m-1}(y_1, \ldots, y_n) = e_m + O(k)\,.
\]
For the second property, we expand $e_\mu = \prod_{i=1}^\ell e_{m_j}$ and approximate the result up to first order in $k$:
\begin{align*}
&\sum_{i=1}^n y_i \prod_{j=1}^\ell e_{m_j}(y_1, \ldots, y_i-k, \ldots, y_n)\\
=&\sum_{i=1}^n y_i \prod_{j=1}^\ell \left( e_{m_j}(y) - k e_{m_j-1}(y_1, \ldots, \widehat{y_i}, \ldots, y_n) \right)\\
=&\sum_{i=1}^n y_i \left(e_\mu(y) - k  \sum_{j=1}^\ell e_{\widehat{\mu}^j} e_{m_j-1}(y_1, \ldots, \widehat{y_i}, \ldots, y_n) \right)+ O(k^2)\,
\end{align*}
where in the last line we use the notation $\widehat{\mu}^j = (m_1, \ldots, \widehat{m_j}, \ldots, m_\ell)$ for the partition obtained from $\mu$ by removing the $j$-th part. Pulling the sum over $i$ we obtain
\begin{align*}
&\sum_{i=1}^n y_i \left(e_\mu(y) - k  \sum_{j=1}^\ell e_{\widehat{\mu}^j} e_{m_j-1}(y_1, \ldots, \widehat{y_i}, \ldots, y_n) \right)+ O(k^2)\\
=& e_1 e_\mu - k  \sum_{j=1}^\ell e_{\widehat{\mu}^j} \sum_{i=1}^n y_i e_{m_j-1}(y_1, \ldots, \widehat{y_i}, \ldots, y_n) + O(k^2)\\
=& e_1 e_\mu - k  \sum_{j=1}^\ell e_{\widehat{\mu}^j} m_j e_{m_j} + O(k^2)\\
=& e_1 e_\mu - k |\mu| e_\mu + O(k^2)\,,
\end{align*}
which finishes the proof. Here in the second to last step we used that
\[
\sum_{i=1}^n y_i e_{m-1}(y_1, \ldots, \widehat{y_i}, \ldots, y_n) = m e_{m}(y)\,.
\]
And indeed, the two sides of this last equality are both sums of terms $y_{i_1} \cdots y_{i_{m}}$ for subsets $\{i_1, \ldots, i_m\} \subseteq \{1, \ldots, n\}$ of size $m$, and each such subset appears precisely $m$ times (e.g. on the left for each choice of $i=i_u$, $u=1, \ldots, m$).
\end{proof}

\begin{proof}[Proof of Proposition \ref{prop:constraint}] For $g>1$ we denote by $\widetilde{\cA}_{g}:\ZZ^n\to \QQ$ the function defined by $$\widetilde{\cA}_{g}(y_1,\ldots,y_n)={\cA}_{g}(y_1+k,\ldots,y_n+k) \text{ for }k=\frac{\sum_i y_i}{2g-2}.$$
The functions $\cA_g$ and $\widetilde{\cA}_g$ determine each other and so it suffices to show that $\widetilde{\cA}_g$ is determined by the analog of the properties from Lemma~\ref{lem:identities} together with the data in genus $0,1$. 

We continue to fix $g>1$ and assume for the moment that $n>2g$. Then since $\widetilde{\cA}_g$ is a polynomial in its entries, of degree at most $2g$ and symmetric in the last $n-1$ arguments, we know from Lemma \ref{Lem:symmfunc} that there exist \emph{unique} coefficients $c_{g,d,\mu} \in \mathbb{Q}$ for all $d\geq 0$, and $\mu=(m_1,\ldots,m_\ell)$ a partition of size $|\mu|$ at most $2g-d$, such that
\begin{equation} \label{eqn:Agtildecoeffs}
    \widetilde{\cA}_{g}(y)=\sum_{g,\mu} c_{g,d,\mu}\cdot  y_1^d\, e_\mu(y)\,.
\end{equation}
Moreover, the first identity \eqref{eqn:Agid1} of Lemma~\ref{lem:identities} implies that $\widetilde{\cA}_{g}(y_1,\ldots,y_n,0)=\widetilde{\cA}_{g}(y_1,\ldots,y_n)$. 
Therefore, by the last part of Lemma \ref{Lem:symmfunc} the coefficients $c_{g,d,\mu}$ are not only unique but in fact \emph{independent of $n$}, since the corresponding expression \eqref{eqn:Agtildecoeffs} restricts correctly to all lower numbers of marked points.

Now consider the third property \eqref{eqn:Agid3}. Note that the operation of restricting $\widetilde{\cA}_g$ to the locus where $e_1(y)=y_1 + \ldots + y_n$ vanishes precisely corresponds to setting $e_1(y)=0$ in the expression \eqref{eqn:Agtildecoeffs}. Given a partition $\mu$, we denote by $M(\mu)$ its number of entries equal to $1$. Then the third identity \eqref{eqn:Agid3} determines all coefficients $c_{g,d,\mu}$, with $M(\mu)=0$. Therefore, we will show that the second identity \eqref{eqn:Agid2} provides a relation allowing to compute the remaining coefficients, inductively on $g,$ $d,$ and $M(\mu)$.

Fix now $g>1$ and $n>2g+1$.  Then the second relation \eqref{eqn:Agid2} translates into the conventions of $\widetilde{\cA}_g$ as\footnote{There is a small annoying point: for $g=2$ the term $\widetilde{\cA}_{g-1}$ appearing below is not defined, since in genus $1$ the value of $k$ cannot be reconstructed from the values of the $y_i$. To fix the argument that follows, one can just replace the evaluation of $\widetilde{\cA}_{g-1}$ on the right by the expression $\cA_1(y_1+k, \ldots, y_n+k, -j,j-k)$ and the entire argument works verbatim. Indeed, the only thing we need is that this term can be written as a linear combination of $y_1^d e_\mu(y)$ and that by assumption we know the coefficients of the linear combination in genus $1$.}
\begin{align}
&(y_1+k) \cdot \widetilde{\cA}_g(y_1,\ldots,y_n,-k)\,\, +   \sum_{i>1} y_i  \widetilde{\cA}_{g}(y_1,\ldots,y_i-k,\ldots,y_n) \label{eqn:tildAgid2}\\
&\quad \quad \nonumber = \frac{1}{2} \sum_{j=0}^k j(k-j)\cdot \widetilde{\cA}_{g-1}(y_1,\ldots,y_n,-j-k,j-2k)
 \nonumber
\end{align}
for any integers $y_1, \ldots, y_n$ and where $k=(y_1 + \ldots + y_n)/(2g-1)$.
Indeed, the reader can verify that applying the definition of $\widetilde{\cA}$ to equation \eqref{eqn:tildAgid2} and substituting $y_i = a_i-k$ precisely gives the second relation \eqref{eqn:Agid2}.
Note that the formula for $k$ was specifically chosen in such a way that
\[
\frac{y_1 + \ldots + y_n -k}{2g-2} = k\,,
\]
so that the value of $k$ computed from the arguments of the functions $\widetilde{\cA}_g$ and $\widetilde{\cA}_{g-1}$ in \eqref{eqn:tildAgid2} agrees with the defining formula of $k$.

Both the left and right hand side are polynomials in
\begin{equation} \label{eqn:ambring}
\QQ[y_1]\otimes \QQ[y_2, \ldots, y_n]^{S_{n-1}}
\end{equation}
of degree at most $2g+1$. Thus by Lemma \ref{Lem:symmfunc} and the assumption $n>2g+1$ both sides are unique linear combinations of the functions $y_1^d e_\mu(y)$.
Comparing coefficients, we obtain a system of relations between the coefficients $c_{g,d,\mu}$ and $c_{g-1,d,\mu}$ of the functions $\widetilde{\cA}_g$ and $\widetilde{\cA}_{g-1}$ which appear. We will argue that this system uniquely determines the coefficients $c_{g,d,\mu}$ assuming the values of the coefficients $c_{g-1,d,\mu}$ are known by induction.
In particular, since the right-hand side of equation \eqref{eqn:tildAgid2} is assumed to be known (and thus forms the inhomogeneous part of the linear system between the $c_{g,d,\mu}$), we can focus on the left-hand side of the equation.


To control the notation below, we declare that for $A,B,C$ elements of the ring \eqref{eqn:ambring}, the equation $A=B + (C)$ means that $A-B$ is contained in the ideal generated by $C$. We begin with some preparatory computations: recall that  $k=e_1/(2g-1)$. 
Given any partition $\mu$, we write it as $\mu = (1^{M(\mu)})+\mu'$, so that $\mu'$ contains no parts $1$.
Then by Lemma \ref{Lem:annoyingsums} and using that $(k)=(e_1)$ as ideals in \eqref{eqn:ambring}, we have
\begin{eqnarray*}
e_{\mu'}(y_1, \ldots, y_n,-k)&=&e_{\mu'}+(e_1)\,,\\
\sum_{i\geq 1} y_i e_{\mu'}(y_1,\ldots,y_i-k, \ldots, y_n)&=&\left(1-\frac{|\mu'|}{2g-1}\right) e_1 e_{\mu'}+(e_1^2)\,.
\end{eqnarray*}
Now we put back the power $e_1^{M(\mu)}$ in the two equations above. For this observe that
\begin{align*}
    e_1(y_1, \ldots, y_n,-k) = e_1(y_1, \ldots, y_i-k, \ldots, y_n) = e_1 - k = \frac{2g-2}{2g-1}e_1\,.
\end{align*}
Multiplying each summand on the left hand side above with $e_1^{M(\mu)}$, we obtain
\begin{eqnarray*}
e_{\mu}(y_1, \ldots, y_n,-k)\!\!&=&\!\!\left(\frac{2g-2}{2g-1}\right)^{M(\mu)} e_{\mu}+(e_1^{M(\mu)+1})\,,\\
\sum_{i\geq 1} y_i e_{\mu}(y_1,\ldots,y_i-k, \ldots, y_n)\!\!&=&\!\!\left(\frac{2g-2}{2g-1}\right)^{M(\mu)} \left(1-\frac{|\mu'|}{2g-1}\right) e_1 e_{\mu}+e_1^{M(\mu)+2}.
\end{eqnarray*}
Now we start putting together the terms appearing in the left-hand side of \eqref{eqn:tildAgid2}. For this observe that
\begin{align*}
&(y_1+k)e_\mu(y,-k)+ \sum_{i>1} y_i e_\mu(\ldots,y_i-k,\ldots)\\
=& (y_1+k)e_\mu(y,-k)-y_1 e_{\mu}(y_1-k,\ldots) + \sum_{i\geq 1} y_i e_\mu(\ldots,y_i-k,\ldots)\\
=& \left(y_1 + \frac{e_1}{2g-1} \right) \cdot \left(\left(\frac{2g-2}{2g-1}\right)^{M(\mu)} e_{\mu}+(e_1^{M(\mu)+1}) \right) - (y_1) \\ &+ \left(\frac{2g-2}{2g-1}\right)^{M(\mu)} \left(1-\frac{|\mu'|}{2g-1}\right) e_1 e_{\mu}+(e_1^{M(\mu)+2})\\
=&\left(\frac{2g-2}{2g-1}\right)^{M(\mu)} \cdot \left(\frac{1}{2g-1}+1-\frac{|\mu'|}{2g-1} \right)e_1 e_\mu + (y_1) + (e_1^{M(\mu)+2})\,.
\end{align*}
Multiplying the two sides of the above equality by $y_1^d$, we precisely get the terms on the left-hand side of \eqref{eqn:tildAgid2} that are multiplied by the coefficient $c_{g,d,\mu}$. Thus we see that $c_{g,d,\mu}$ appears as a factor in front of a term $y_1^{\widetilde{d}} e_{\widetilde{\mu}}$ only if either
\begin{itemize}
    \item $\widetilde{d}=d$ and $\widetilde{\mu}=(1)+\mu$, in which case it appears with the coefficient
    \begin{equation} \label{eqn:cgdmucoeff}
        \left(\frac{2g-2}{2g-1}\right)^{M(\mu)} \cdot \frac{2g-|\mu'|}{2g-1}\,,
    \end{equation}
    \item or possibly when $\widetilde{d}>d$ or $M(\widetilde{\mu})>M(\mu)+1$.
\end{itemize}
To conclude, we first remark that the coefficient \eqref{eqn:cgdmucoeff} can only vanish when $|\mu'|=2g$. Since the total degree of $\widetilde{\cA}_g$ is bounded by $2g$, this can only happen when $\mu=\mu'$ and thus $M(\mu)=0$, in which case we already know the coefficient $c_{g,d,\mu}$ by the third assumption \eqref{eqn:Agid3}. Thus we can assume that \eqref{eqn:cgdmucoeff} does not vanish and looking at the total factor in front of $y_1^d e_1 e_\mu$ on the left-hand side of \eqref{eqn:tildAgid2}, we see a nonzero multiple of $c_{g,d,\mu}$ as well as some multiples of coefficients $c_{g,d'',\mu''}$ satisfying $d''<d$ or $M(\mu'')<M(\mu)$. By doing an induction on $d$ and $M(\mu)$, which we can start with the cases $d=-1$, $M(\mu)=-1$ which vanish, we can assume that we have already computed these lower coefficients. Thus we can uniquely determine $c_{g,d,\mu}$ as claimed above, finishing the proof.
\end{proof}

\begin{proof}[Proof of Theorem~\ref{th:main}]
We denote by 
\begin{eqnarray*}
B_g:\ZZ^n &\to& \QQ\\
a &\mapsto & [z^{2g}] \,\,\, {\rm exp}\left(\frac{a_1z\cdot \cS'(kz)}{\cS(kz)} \right) \frac{\prod_{i>1} \cS(a_iz)}{\cS(z)\cS(kz)^{2g-1+n}}
\end{eqnarray*}
Using  Proposition~\ref{prop:constraint}, we will prove
prove Theorem~\ref{th:main} by showing that $B_{g}$ is a polynomial of degree $2g$ that satisfies the four identites of Lemma~\ref{lem:identities}, and that $\cA_0=B_0$ and $\cA_1=B_1$.

The polynomiality and degree bound of $B_g$ follows by the rules of expansion of power series: the variable $z$ appears in the definition of $B_g$ always with coefficients that are either constant or linear in the $a_i$. Thus after substituting into power series and taking products of these, the coefficient of $z^{2g}$ is indeed a polynomial of degree at most $2g$. The symmetry of the formula in the arguments $a_2, \ldots, a_n$ is likewise obvious.

The fact that these numbers satisfy the first identity \eqref{eqn:Agid1} is straightforward, the third identity follows the fact that setting $|a|=0$ is equivalent to setting $k=0$ and the fact that $\cS(0)=1, \cS'(0)=0$.

In genus $0$ we have that $\DR_0(a)=[\oM_{0,n}]$ is independent of $a$ and so
\[
\cA_0(a)=\int_{\oM_{0,n}} \psi_1^{n-3} = 1\,,
\]
e.g. by \cite[Theorem 1]{BSSZ}, which agrees with $B_0=1$.

In genus $1$, we can reduce the check that $\cA_1 = B_1$ to the case $k=0$. Indeed, the fact that $ \omega_C \cong \mathcal{O}_C$ for smooth genus $1$ curves $C$ can be used to show that the double ramification cycle is independent of $k$, in the sense that
\[
\DR_1(a_1, \ldots, a_n) = \DR_1(a_1-k, \ldots, a_n-k)\,.
\]
Alternatively, this equality can be checked directly from the formula of $\DR_1(a)$ by applying known relations between divisor classes on $\oM_{1,n}$. In any case we can conclude that
\begin{equation} \label{eqn:A1equality}
    \cA_1(a_1, \ldots, a_n) =  \cA_1(a_1-k, \ldots, a_n-k)\,.
\end{equation}
Note here that the arguments on the right-hand side of \eqref{eqn:A1equality} indeed sum to zero. As for the function $B_1$, an explicit computation using the substition $a_1 = kn - \sum_{i>1} a_i$ and the expansion $\cS(u)=1+u^2/24+O(u^4)$ shows that
\begin{align*}
    B_1(a) &= [z^2] {\rm exp}\left(\frac{a_1z\cdot \cS'(kz)}{\cS(kz)} \right) \frac{\prod_{i>1} \cS(a_iz)}{\cS(z)\cS(kz)^{n+1}}\\
    &=\frac{1}{24} \left(-1 + \sum_{i=2}^n (a_i -k)^2  \right)\,.
\end{align*}
From this form we then immediately see that
\begin{equation} \label{eqn:B1equality}
    B_1(a_1, \ldots, a_n) =  B_1(a_1-k, \ldots, a_n-k)\,.
\end{equation}
Combining equations \eqref{eqn:A1equality} and \eqref{eqn:B1equality} we see that the functions $\cA_1$ and $B_1$ are determined (via the same transformation) by their values on vectors $a$ with $|a| = 0$. But for such vectors, the equality $\cA_1(a)=B_1(a)$ follows from \eqref{eqn:Agid3}.

Thus, we only need to show that the numbers $B_{g}(a)$ satisfy the identity~\eqref{eqn:Agid2}. We fix a vector $a\in \ZZ^n$, and we introduce the following formal series in $z$:
$$
B_{g}(a)(z)= {\rm exp}\left(\frac{a_1z\cdot \cS'(kz)}{\cS(kz)} \right)\cdot \frac{\prod_{i>1} \cS(a_iz)}{\cS(z)\cS(kz)^{2g-2+n}}.
$$
The derivative of this power series is given by:
\begin{eqnarray*}
\frac{d}{dz}  {B_{g}(a)(z)}  &=&  {B_{g}(a)(z)} \cdot \Bigg( \frac{a_1}{kz}-\frac{a_1}{k { z} \cS(kz)^2} - (k(2g-2+n)-a_1) \frac{\cS'(kz)}{\cS(kz)} \\   &&- \frac{\cS'(z)}{\cS(z)} + \sum_{i>1} a_i \frac{\cS'(a_iz)}{\cS(a_iz)}\Bigg). 
\end{eqnarray*}
\rcomment{Here we have used the expressions of the derivatives of $\cS$:
\begin{equation}
\cS'=\frac{{\rm cosh}(z/2)-\cS}{z}, \,\,\text{ and } \,\, \cS''=\frac{1}{4}\cS - \frac{2}{z}\cS'.
\end{equation}
}
On the other hand, we write the terms involved identity~\eqref{eqn:Agid2} using coefficient extraction formulas:
\begin{eqnarray}
\label{eqn:term1} a_1 B_{g}(a,0)&=& a_1 [z^{2g}] \cS(kz)^{-2} B_{g}(a)(z), \\
\label{eqn:term2} \sum_{i>1} (a_i-k) B_{g}(\ldots,a_i-k,\ldots)
&=&[z^{2g}] B_{g}(a)(z) \Bigg( (2g-a_1/k)k  \\ \nonumber
&& \!\!\!\!\!\!\!\!\!\!\!\!\!\!\!\!\!\!\!\!\!\!\!\!\!\!\!\!\!\!\!\!\!\!\!\!\!\!\!\!+(k(2g-1+n)-a_1)kz\frac{\cS'(kz)}{\cS(kz)}- k \sum_{i>1}  \frac{a_i z\cS'(a_i z)}{\cS(a_i z)}\Bigg),\\
\label{eqn:term3} \sum_{0<j<k} \frac{j(k-j)}{2}  \!B_{g-1}(a,-j,-k+j)
&=& k [z^{2g}]  B_{g}(a)(z) \left(kz\frac{\cS'(kz)}{\cS(kz)} - \frac{z\cS'(z)}{\cS(z)} \right),
\end{eqnarray}
The last expression follows from the formula
\begin{equation} \label{eqn:sinhsumformula}
\sum_{0<j<k} \frac{j(k-j)}{2} \cS(jz) \cS((k-j)z) = \frac{1}{z^2} \left(k \cosh(kz/2)  - \frac{\sinh(kz/2) \cosh(z/2)}{\sinh(z/2)}\right)\,,
\end{equation}
which can be proven by inserting the definition of $\cS$, multiplying both sides with $\sinh(z/2)$, expanding everything in terms of exponential functions and comparing coefficients of each terms $e^{cz/2}$ for $c \in \mathbb{Z}$.
\bigskip

In order to show that $B_g$ satisfies identity~\eqref{eqn:Agid2}, we will use the fact that
$$
-k [z^{2g}] z^{2g+1} \frac{d}{dz}  z^{-2g} {B_{g}(a)(z)}=0,
$$
for all values of $a$, since the $z^{-1}$-coefficient of the derivative of any Laurent series vanishes. Using the above expression of the derivative of $B_g(a)(z)$, we obtain:
\begin{eqnarray*}
\rcomment{0}&\rcomment{=}&-k [z^{2g}] z^{2g+1} \frac{d}{dz}  z^{-2g} {B_{g}(a)(z)} \\
&=& \rcomment{[z^{2g}]} {B_{g}(a)(z)} \cdot \Bigg( (2g{-a_1/k})k+\frac{a_1}{ { } \cS(kz)^2} + (k^2(2g-2+n)-a_1k) \frac{z\cS'(kz)}{\cS(kz)} \\   &&+ \frac{kz\cS'(z)}{\cS(z)} - \sum_{i>1} a_ikz \frac{\cS'(a_iz)}{\cS(a_iz)}\Bigg). \\
&=& a_1 B_{g}(a,0) - \frac{1}{2}\sum_{0<j<k} j(k-j) B_{g-1}(a,-j,-k+j)\\
&& + \sum_{i>1} (a_j-k) B_{g}(\ldots,a_j-k,\ldots),
\end{eqnarray*}
\rcomment{(to go from the second to the third line we have used the expressions~\eqref{eqn:term1}, ~\eqref{eqn:term2} and ~\eqref{eqn:term3} of the terms involved in~\eqref{eqn:Agid2}),
thus showing that the function $B_{g}$ satisfy the identity~\eqref{eqn:Agid2} for all values of $a$.}
\end{proof}

\section{Top-\texorpdfstring{$\psi$}{psi} for spin components} \label{Sect:spinrefinement}

In this section we analyze integrals of $\psi$-classes on spin components. In order to prove Theorem~\ref{th:spin}, we will follow the same strategy as for Theorem~\ref{th:main}. That is, we will first show that the intersection of $\psi$-classes with strata of $[\oM_g(a)]^\sp$ may be expressed in terms of classes supported on graphs with 2 vertices. These formulas hold for odd values of $k$ and $a$, but we will use the polynomiality  of the function $\cA_{g}^\sp$ implied by Assumption~\ref{assumption}  to find constraints on this function that determine it uniquely (in the spirit of Lemma~\ref{lem:identities} and Proposition  ~\ref{prop:constraint} in the non-spin case).

\subsection{Parity of boundary divisors}
\label{sec:spinboundary} 

Let $k$ be an odd integer and $a$ an odd signature. We recall  from Section \ref{Sect:spin_refinement_intro}  that
for a pair $(C, \eta)$ of a smooth curve $C$ and a meromorphic $k$-differential $\eta$ on $C$ with divisor $\mathrm{div}(\eta) = \sum_i (a_i-k) x_i$, the parity of  $(C, \eta)$  is defined as the parity of the spin structure
$$
\Phi(C,\eta)\coloneqq h^0\left(\omega_C^{\otimes(-k+1)/2} \otimes \mathcal{O}_C\left(\frac{a_1-k}{2}x_1+\ldots+\frac{a_n-k}{2}x_n\right)\right) ({\rm mod}\  2).
$$
\medskip

In this section, we work over the moduli space of multi-scale $k$-differentials $\PP\Xi_g(a)$. Recall  from Section~\ref{Sec:multiscaled} that for an enhanced level graph $\overline{\Gamma}$ with $L+1$ levels, the associated boundary stratum $\PP\Xi(\oGamma)$ parametrizes  multi-scale $k$-differentials associated to $((C_i,[\eta_i])_{i=0,\dots,-L},[\sigma])$, where $(C_i,[\eta_i])_{i=0,\dots,-L}$ is a twisted $k$-differential compatible with a $k$-cover of graphs $\widehat{\Gamma}\to\overline{\Gamma}$ and $[\sigma]$ is an equivalence class of a global prong-matching  on the canonical $k$-cover abelian multi-scale differential which is equivariant with respect to the $k$-cyclic deck transformation. 

\medskip

The moduli space of multi-scale $k$-differentials is smooth, containing $\M_g(a)$ as a dense open subset, see \cite{BCGGM3} and \cite{CMZarea}. Thus the connected components of  $\PP\Xi_g(a)$ are closures of connected components of $\M_g(a)$\footnote{This fact is not true for the incidence variety compactification as a point in the boundary may be in the closure of several connected components of $\M_g(a)$.}. This implies that the parity function can be extended to the boundary by continuity. In the present subsection, we fix a twisted $k$-differential  $(C_i,[\eta_i])_{i=0,-1}$ compatible with a cover of graphs $\widehat{\Gamma}\to\overline{\Gamma}$ where $\overline{\Gamma}$ has only two levels with one vertex each. Then for a global prong-matching equivalence class $[\sigma]$, we denote by $\Phi([\sigma])$ the parity of the associated point in the moduli space of multi-scale $k$-differentials. By the discussion above, this parity is  the parity of the generic point of \emph{any} smoothing of $((C_i,[\eta_i])_{i=0,-1},[\sigma])$ in  $\PP\Xi_g(a)$. The following proposition provides an (almost) explicit description of this function.

\begin{proposition}\label{prop:multibanana} If $\oGamma$ has one vertex at each level then the following holds:
\begin{enumerate}
\item If $\oGamma$ has only odd twists, then $\Phi([\sigma])=\Phi(C_0,\eta_0)+\Phi(C_{-1},\eta_{-1})$ for all global prong-matchings $[\sigma]$.
    \item If $\sigma$ and $\sigma'$ are two global prong-matchings (but we no longer assume that all edges $e$ of $\oGamma$ have only odd twists $\kappa_e$), then 
    \begin{equation} \label{eqn:parity_difference_product}
        (-1)^{\Phi([\sigma])-\Phi([\sigma'])}= \underset{\text{s.t. ${\kappa_e}$ 
 even}}{\prod_{e \in \overline{\Gamma},}} \prod_{\widehat{e}\to e} \left(\frac{\sigma_{\widehat{e}}}{\sigma_{\widehat{e}}'}\right)^{\frac{\widehat{\kappa}_{e}}{2}},
    \end{equation}
    where $\widehat{\kappa}_{e}\coloneqq \frac{\kappa_e}{{\rm gcd}(\kappa_e,k)}$. Here 
 $\sigma_{\widehat{e}}/\sigma_{\widehat{e}}'$ is the ratio of the two prong matchings at the node of the canonical cover associated  with $\widehat{e}$, which is a $\widehat{\kappa}_{e}$-th root of unity (so the values in this product are either $1$ or $-1$). 
\end{enumerate}
\end{proposition}

We make some comments on formula~\eqref{eqn:parity_difference_product}. First, we remark that each edge $e$ of $\oGamma$ with even twist $\kappa_e$ has an odd number $\gcd(\kappa_e, k)$ of pre-images in $\widehat{\Gamma}$. Furthermore, each of these pre-images contribute to the product with the same factor $\pm 1$ due to the $k$-equivariance of the prong matchings. Thus the term  $\prod_{\widehat{e}\to e} \left(\sigma_{\widehat{e}}/{\sigma_{\widehat{e}}'}\right)^{\frac{\widehat{\kappa}_{e}}{2}}$ can be replaced by $ \left({\sigma_{\widehat{e}}}/{\sigma_{\widehat{e}}'}\right)^{\frac{\widehat{\kappa}_{e}}{2}}$ for \emph{any} choice of pre-image $\widehat{e}\to e$.

Secondly, we verify as a sanity-check that if $\sigma$ and $\sigma'$ are representatives of the same prong-matching equivalence class $[\sigma]=[\sigma']$ then the RHS of \eqref{eqn:parity_difference_product} gives a factor $1$. To see this note that the level rotation group (which generates the equivalence relation between prong-matchings) is generated by the operation 
\[
\sigma = (\sigma_{\widehat e})_{\widehat e \in E(\widehat \Gamma)} \mapsto \sigma' = (\exp(2 i \pi/\widehat{\kappa}_{e}) \cdot  \sigma_{\widehat e})_{\widehat e \in E(\widehat \Gamma)}\,.
\]
Plugging this into \eqref{eqn:parity_difference_product} we obtain a factor $-1$ for each edge $e$ with even twist $\kappa_e$. Since $\oGamma$ has an even number of such edges\footnote{Each such edge corresponds to an odd number of odd zeros or poles of the abelian differential on the components of the canonical cover.}, the whole product evaluates to $+1$. 

Finally, the main application of formula~\eqref{eqn:parity_difference_product} in the rest of the text is the following corollary.  

\begin{corollary} \label{Cor:prong_balance}
    If $\oGamma$ has at least one edge $e_0$ with even twist $\kappa_{e_0}$, then half of the prong-matching equivalence classes have even parity and half have odd parity.
\end{corollary}

\begin{proof}
Let $\mathcal{P}$ be the set of prong-matching equivalence classes, then we define a bijective map $t:\mathcal{P} \to \mathcal{P}$ as follows: given $[\sigma]\in \mathcal{P}$ we let $[\sigma']=t([\sigma])$ be the prong-matching equivalence class with
\[
\sigma'_{\widehat e} = \begin{cases}
    \exp(2 i \pi/\widehat{\kappa}_{e_0}) \sigma_{\widehat e} & \text{ if $\widehat e$ maps to $e_0$,}\\
    \sigma_{\widehat e} &\text{ otherwise}.
\end{cases}
\]
Then $\sigma'$ is invariant under the cyclic action of $\mathbb{Z}/k\mathbb{Z}$ on $\widehat \Gamma$ and thus gives a well-defined prong-matching for the multi-scale $k$-differential. Since the action of the level-rotation group is compatible, this leads to a well-defined map $t : \mathcal{P} \to \mathcal{P}$ on the prong-equivalence classes. Its inverse $t^{-1}$ is defined by multiplying by $\exp(-2 i \pi/\widehat{\kappa}_{e_0})$ in the formula above, so $t$ is bijective. Finally, it follows from \eqref{eqn:parity_difference_product} that $$(-1)^{\Phi([\sigma])-\Phi([\sigma'])} = \prod_{\widehat e \to e_0} \underbrace{\exp(2 i \pi/\widehat{\kappa}_{e_0})^{\widehat{\kappa}_{e_0}/2}}_{=-1} = (-1)^{\gcd(k,\kappa_{e_0})} = -1.$$
Thus $t$ exchanges the odd and even prong matching equivalence classes, which thus must be of the same cardinality.
\end{proof}

 To prove Proposition \ref{prop:multibanana}, we first establish the case $k=1$, and then reduce the general statement to $k=1$.

\begin{proof}[Proof of Proposition~\ref{prop:multibanana} for k=1]   Recall that in the case of abelian differentials, the parity of a spin structure may also be defined in the following way using the Arf invariant (see \cite{EskMasZor} or \cite{KZ}).
Consider a flat surface $(C,\eta)$ and fix a symplectic basis of homology $\{A_i,B_i\}_{i=\,\dots,g}$. When the zeroes of the abelian differential $\eta$ have only even orders, we can define the parity of the spin structure via
\[\Phi(C,\eta)=\sum_{i=1}^g (\Ind_\eta(A_i)+1)(\Ind_\eta(B_i)+1) \ (\mod\ 2),\]
where $\Ind_\eta(\alpha)$ is the index of the tangent vector field of a closed curve $\alpha$ with respect to the flat metric induced by $\eta$. One can show that this parity is independent of the choices.

Let $\phi_{[\sigma]}\colon B\to \PP\Xi_g(a)$ be a smoothing of the multi-scale differential associated to a prong matching $[\sigma]$, i.e. a family of differentials 
parametrized by a disk $B$ centered at $0$ such that the differential over $b\neq 0$ is supported on a smooth curve, and 0 is mapped to $((C_i,[\eta_i])_{i=0,-1},[\sigma])$. 
In order to use the above expression of the parity, we will first fix a homology basis of the smooth fibers over a simply connected open subset of $B\setminus\{0\}$ whose closure contains 0.
We start by fixing an edge $e_0$.  We  choose a symplectic basis $(A_i,B_i)_{1\leq i\leq g(C_0)+g(C_{-1})}$ of the homology of  $C_0$ and $C_{-1}$ and complete this basis by cycles $v_e, u_e$ for each edge $e\neq e_0$ where: 
\begin{itemize}
    \item The cycle $v_e$ is the vanishing cycle at the node associated with $e$ (it is represented by a simple loop with limit contracted at the node).
    \item The cycle $u_e$ is represented by a loop  whose limit is a union of two simple curves in $C_0$ and $C_{-1}$ that join $e_0$ and $e$ (and we impose that $u_e$ and $v_e$ meet once). This cycle is not uniquely determined, as a Dehn twist along a simple loop representing $v_e$ or $v_{e_0}$ produces a loop $u_e'$ satisfying the same constraints. However, the following computation is not affected by this choice.
 \end{itemize}
The cycles $(A_i,B_i,u_e,v_e)$ form a symplectic basis on each smooth fiber  $(C,\eta)$. With this notation we have:
\begin{eqnarray*}
    \Phi([\sigma])&=& \sum_{i=1}^{g(C_0)+g(C_{-1})}(\Ind_\eta(A_i)+1)(\Ind_\eta(B_i)+1)  \\ && +\sum_{e\neq e_0} (\Ind_\eta(v_e)+1)(\Ind_\eta(u_e)+1)  \ (\mod\ 2).
\end{eqnarray*}
The first sum is locally constant, and it is independent from the choice of different prong-matchings  because the cycles $A_i, B_i$ can be represented by simple loops that do not intersect some neighborhood of the nodes in the universal curve. 
By considering the standard form of the family of differentials near the node shown in \cite[Theorem 4.3]{BCGGM3}, the parity of the index $\Ind_\eta(v_e)$ of the vanishing cycle equals the parity of the twists at the node $e$ (which is the same as the number of prongs at $e$).

If an edge $e$ has an odd twist, then $\Ind_\eta(v_e)+1$ is even and so the contribution of this edge to the parity is trivial. In particular, if we are in case (1), then $\Ind_\eta(v_e)+1$ is even for every $e$, so the parity of $(C,[\eta])$ is the sum of the parities of the two levels.

Let $e\neq e_0$ be an edge of $\oGamma$. We denote now by $t_e(\sigma)$ the global prong-matching obtained from $[\sigma]$ by multiplying the local prong-matching at $e$ with $e^{2i\pi/\kappa_e}$ while leaving the other values unchanged.

We claim that  if $e$ has an odd twist, then the parity of $\sigma$ and $t_e(\sigma)$ are the same, while if $e$ has an even twist the parity of $\sigma$ and $t_e(\sigma)$ are different, consistent with formula \eqref{eqn:parity_difference_product} stating that
\[
(-1)^{\Phi([\sigma])-\Phi([t_e(\sigma)])}=\left(\frac{\sigma_e}{e^{2i \pi/\kappa_e} \sigma_e} \right)^{\kappa_e/2} = e^{-i\pi}=-1 \,.
\]

The proposition then follows as any prong-matching $\sigma'$ can be obtained from $\sigma$ by a sequence of such operations $t_e$ at different edges, and the parity of $\sigma'$ will differ from the parity of $\sigma$ by the number of such operations (modulo 2).

To prove the above claim, we consider a smoothing family $\phi_{[t_e(\sigma)]}\colon B\to \PP\Xi_g(a)$ of the multi-scale differential $((C_i, [\eta_i])_{i=0,-1}, [t_e(\sigma)])$ obtained from the family $\phi_{[\sigma]}$ by  twisting the upper part of the node associated to $e$ by $e^{2i\pi/\kappa_e}$ (see the proof of the paragraph~\ref{subsub:twisting}  for the more in depth definition of twisted families for general $k$). Note that we are not changing the differentials in neighborhoods of the nodes in the total space of the family, but only the gluing functions attaching the plumbing fixtures to the  upper part of the nodes.
Among the $A,B,u,$ and $v$ cycles, only the index of the cycle $u_e$ is transformed non-trivially by this operation. The image of $u_e$ is now a path that does not close. However, it can be completed into a closed curve $u_e'=u_e+u_e^C$ if we join the two extremities of this path by a curve $u_e^C$. Even though there is not a unique choice to construct $u_e'$,  the difference between two choices is given by a multiple of $v_e$, and this will not affect the argument. 

If an edge $e$ has an odd twist, then $\Ind_\eta(v_e)+1$ is even and so $u_e$ or $u_e'$ give no contribution to the parity, which implies that the parity of $\sigma$ and $t_e(\sigma)$ are the same. If $e$ has an even twist, then $\Ind_\eta(v_e)+1=\Ind_{\eta'}(v_e)+1$ is equal to $1$ (modulo 2), where $\eta'$ is an abelian differential on a smooth fiber of the twisted family. This implies that  the contribution of the edge $e$ to the parity $\Phi([\sigma])$, resp. $\Phi([t_e(\sigma)])$, is equal to $\Ind_\eta(u_e)+1$, resp. $\Ind_{\eta'}(u'_e)+1$ (and this value is independent from the choice of $u_e$ or $u_e'$).  Since the differential $\eta'$ is equal to $\eta$ in a neighborhood of the node in the total space of the family, the claim then follows since ${\rm Ind}_{\eta'}(u_e')={\rm Ind}_{\eta'}(u_e+u_e^C)={\rm Ind}_{\eta}(u_e)+1$ (modulo 2).
\end{proof}

For $k=1$, the same method can be applied to  prove the analog of the first part of Proposition~\ref{prop:multibanana} for graphs of compact type.
\begin{proposition} \label{pro:parityct_improved}
If $k=1$ and  $\oGamma$ is a tree, then all edges have odd twist (and thus are vertical). If $((C_v,[\eta_v])_{v \in V(\oGamma)},[\sigma])$ is a point of $\PP\Xi(\oGamma)$, then the parity of this point is independent of $\sigma$, and given by $ \sum_{v \in V(\oGamma)} \Phi(C_v,\eta_v)\,.$

\end{proposition}



\subsubsection{Constructing an auxiliary family of one-differentials} We no longer assume that $k=1$. Let $[\sigma]$ be an equivalence class of a global prong-matching, and let $\phi_{[\sigma]}\colon B\to \PP\Xi_g(a)$ be a smoothing of the multi-scale differential  $\phi_{[\sigma]}(0)=((C_i,[\eta_i])_{i=0,-1},[\sigma])$ for the special point $0 \in B$. 
We denote by $\pi_{[\sigma]}\colon \mathcal{C}_{[\sigma]}\to B$ the underlying family of pointed curves. \rcommentfour{Up to a choice of a smaller $B$, we fix a trivialization $\eta_{[\sigma]}\colon B \to \mathcal{O}(-1)^*$ of the (pull back from the IVC compactification of the) universal line bundle. In other words, we fix a family of $k$-differentials $(\eta_{b,[\sigma]})_{b\in B}$ that  recovers the family $\phi_{[\sigma]}$ after quotient by $\CC^*$. On the central fiber we can normalize this section to have $\eta_{0,[\sigma]}=\eta_0$ on $C_0$, while $\eta_{0,[\sigma]}$ vanishes on the component $C_{-1}$.} 

In order to reduce Proposition~\ref{prop:multibanana} to the special case $k=1$, we first show that we can find a family of meromorphic stable $1$-differentials $(\eta'_{b,[\sigma]})_{b\in B}$ on the family of curves $\mathcal{C}_{[\sigma]}\to B$ with a particular set of properties (see Lemma \ref{lem:spin0}). 
We then show that for $b \neq 0$, the parity of the $k$-differential $\eta_{b,[\sigma]}$ 
agrees with the parity of the $1$-differential 
$\eta_{b,[\sigma]}/(\eta'_{b,[\sigma]})^{k-1}$, allowing us eventually to reduce to the proven case of $1$-differentials.


\begin{lemma}\label{lem:spin0}
    There exists a meromorphic $1$-form  $\eta'$ on the central fiber, i.e. the nodal curve $\pi_{[\sigma]}^{-1}(0) = C_0\cup C_{-1}$, satisfying the following constraints:
    \begin{enumerate}
    \item[a)] it has a pole of order $M\gg0$ at a point $Q_i \in C_i$ which is different from the nodes and the singularities of $\eta_i$, for $i=0$ and $-1$,
    \item[b)] it has simple poles at the branches of the nodes, with nonzero residues,
    \item[c)] it has no other poles, and the zeros are simple and away from nodes, the points $Q_0$, $Q_{-1}$, and the singularities of $\eta_0$ and $\eta_{-1}$.
\end{enumerate}
Moreover, after possibly shrinking to an \'etale neighborhood of $0 \in B$, there exists a meromophic section $(\eta'_{b,[\sigma]})_{b\in B}$ of $\omega_{\mathcal{C}_{[\sigma]}/B}$ such that on the central fiber $\eta_{0,[\sigma]}'=\eta'$, and at any smooth fiber the differential $\eta'_{b,[\sigma]}$ has the same pattern of zeros and poles (at smooth points) as $\eta'$, and the singularities of $\eta_{b,[\sigma]}'$ are not at the markings. 
\end{lemma}
\begin{proof}
    This lemma follows from classical arguments. We first choose points $Q_i \in C_i$ away from other special points. Then for sufficiently large $M$, the line bundle $\omega_{C}(M(Q_0 + Q_{-1}))$ is very ample on the central fiber $C_0\cup C_{-1}$.  This embeds $C_0\cup C_{-1}$ in some projective space. Choosing a hyperplane section that intersects $C_0\cup C_{-1}$ transversally away from all special points, we obtain the desired section $\eta' \in H^0(C_0\cup C_{-1}, \omega_{C_0\cup C_{-1}}(M(Q_0 + Q_{-1})))$. Indeed, transversality implies that zeros are simple, and these zeros missing the nodes implies that the residues are nonzero.

    To extend this meromophic differential to a family, we note that after taking an \'etale cover of a neighborhood of $0 \in B$, we can assume that the two chosen points $Q_0, Q_{-1}$ extend to sections $q_0, q_{-1}$ of $\pi_{[\sigma]}$ which still continue to miss all the special points. Moreover, by vanishing of the higher cohomology, the sheaf $(\pi_{[\sigma]})_* \omega_{\pi_{[\sigma]}}(M(q_0 + q_{-1}))$ is a vector bundle. We let $(\eta'_{b,[\sigma]})_{b \in B}$ be (the values of) a section of this bundle extending the chosen section $\eta'=\eta'_{0,[\sigma]}$ on the central fiber above. The differentials $\eta'_{b,[\sigma]}$ continue to have simple zeros which are away from all other special points, since this is an open condition in the family.
\end{proof}

We use this lemma to define a family of multi-scale $1$-differentials on the original family of curves $(\widetilde{\pi}_{[\sigma]}:\widetilde{\mathcal{C}}_{[\sigma]} \to B,[\widetilde{\eta}_{b,[\sigma]}],[\widetilde{\sigma}])$ on $B$ as follows. We set $\widetilde{\mathcal{C}}_{[\sigma]}\setminus  \widetilde{\pi}_{[\sigma]}^{-1}(0)= \mathcal{C}_{[\sigma]}\setminus  \pi_{[\sigma]}^{-1}(0)$ and there we define 
$$\widetilde{\eta}_{b,[\sigma]}=\eta_{b,[\sigma]}/(\eta_{b,[\sigma]}')^{k-1}.$$ 
Since at any smooth fiber the differential $\eta'_{b,[\sigma]}$ has the same pattern of zeros and poles and they are away from all the markings defined by $\eta_{b,[\sigma]}$, the abelian differentials $\widetilde{\eta}_{b,[\sigma]}$ have the same pattern of zeros and poles for all $b \in B \setminus \{0\}$. Hence, after replacing  $B$ by an \'etale cover, we can find sections enumerating all the simple zeros and obtain a family in a (fixed) stratum of meromorphic differentials $\M_g(\widetilde{a})$.

By the properness of the multi-scale compactification of this stratum, we obtain a multi-scale 1-differential over the special point $0\in B$. In this construction, we did not change the underlying family of smooth curves $\varphi_{[\sigma]}: B \setminus \{0\} \to \oM_{g}$ so the special fiber is also uniquely determined as the curve  glued from components $C_0, C_{-1}$ as above, and we obtain that $\widetilde{\mathcal{C}}_{[\sigma]}\to B$ is equivalent to $\mathcal{C}_{[\sigma]}$ (we only added some sections associated with markings).

Moreover, we claim that the twisted differentials $\widetilde{\eta}_0, \widetilde{\eta}_{-1}$ on $C_0, C_{-1}$ obtained by this procedure are given by $\widetilde{\eta}_i=\eta_i/(\eta'_i)^{k-1}$, where $\eta_i'$ is the restriction of $\eta'$ to $C_i$. Indeed, it follows from \cite[Definition 11.2]{BCGGM3} that (in the complex-analytic category) the twisted differentials $\widetilde{\eta}_i$ on the components of the central fiber can be obtained as rescaled limits of the differentials $\widetilde{\eta}_{b,[\sigma]}=\eta_{b,[\sigma]}/(\eta'_{b,[\sigma]})^{k-1}$ on the points $b \in B \setminus \{0\}$. However, using the same scaling parameters as for the family $\eta_{b,[\sigma]}$ of $k$-differentials (which converge to the twisted $k$-differentials $(\eta_{0},\eta_{-1})$ by \cite{BCGGM2}), the $1$-differentials $\widetilde{\eta}_{b,[\sigma]}$ converge to $\eta_i/(\eta'_i)^{k-1}$, since the rescaled numerator and the denominator of this fraction extend in the family to $b=0$.

To summarize, we have now constructed a family of multi-scale $1$-differentials \rcommentfour{(with a trivialization of the universal line bundle)}, smoothing the central fiber $((C_0, [\widetilde{\eta}_0]), (C_{-1},[\widetilde{\eta}_{-1}]), [\widetilde \sigma])$, where we use the notation $\widetilde{\sigma}$ for the prong-matching equivalence class determined by the above limit\footnote{A priori, the prong-matching $\widetilde{\sigma}$ depends on the choice the family $\eta'_b$ extending the meromorphic stable differential from Lemma~\ref{lem:spin0}. As we show below, it is in fact independent of this choice.}. Moreover, on the singular fiber, if $x_e$ is the branch of a node on the component $C_0$ or $C_{-1}$, then
$$
{\rm ord}_{x_e}(\widetilde{\eta_i})={\rm ord}_{x_e}({\eta_i})+(k-1).
$$
Thus  $\widetilde{\eta_i}$ and $\eta_i$ have the same enhancements $\kappa_e$ (or log-orders) at the branches of the node associated to the edge $e$ of $\overline{\Gamma}$, with the same parities since $k$ is odd. 
Here we used the fact that, by property b) above, the differentials $\eta_i'$ have simple poles at the branches of the nodes. Note also that since the zero and pole orders of $1/(\eta_i')^{k-1}$ are multiples of the even number $k-1$,  all entries of the (log-convention) signature $\widetilde a$ of $\widetilde{\eta}_{b,[\sigma]}$ are odd.

Finally, we claim that the spin structures defined by $\widetilde{\eta}_{b,[\sigma]}$ and $\eta_{b,[\sigma]}$ are equal on smooth fibers $\mathcal C_{[\sigma],b}$ for $b \neq 0$, i.e. 
\begin{equation} \label{eqn:spin_parity_on_smooth_fibers}
\Phi(\mathcal C_{[\sigma],b}, \eta_{b,[\sigma]}) = \Phi(\mathcal C_{[\sigma],b}, \widetilde \eta_{b,[\sigma]})\,.
\end{equation}
Indeed, the spin structure associated to $\widetilde{\eta}_{b,[\sigma]}$ is obtained by forming its divisor
\[
\mathrm{div}(\widetilde \eta_{b,[\sigma]}) = \mathrm{div}(\eta_{b,[\sigma]}) - (k-1) \mathrm{div}(\eta_{b,[\sigma]}') = \left(\sum_{i=1}^n (a_i-k) (x_i)\right) - (k-1) \mathrm{div}(\eta_{b,[\sigma]}'),
\]
so we have the equality
\[
\mathcal{O}_{\mathcal C_b}\left(\frac{\mathrm{div}(\widetilde \eta_{b,[\sigma]})}{2}\right) =  \mathcal{O}_{\mathcal C_b}\left(\sum_{i=1}^n\frac{a_i-k}{2}x_i\right) - \frac{k-1}{2}\omega_{C_{[\sigma],b}} 
\]
which is  precisely  the spin structure associated with $\eta_{b,[\sigma]}$ by~\eqref{eq:spinbundle}. Therefore, we need to determine the dependence of the parity of the family of abelian differentials $(\widetilde{\eta}_{b,[\sigma]})$ in terms of $[\sigma]$. In order to use the proven case $k=1$ of Proposition~\ref{prop:multibanana} for this purpose, we will describe the dependence of the prong-matching $[\widetilde{\sigma}]$ in terms of $[\sigma]$.

\subsubsection{Twisting a family of $k$-differentials}\label{subsub:twisting} We explain in more detail the process of twisting a prong-matching at a node in the case of $k$-differentials.  
Consider the family of abelian differentials $(\widehat{\eta}_{b,[\sigma]})_{b\in B}$ coming from the canonical cover of the family $(\eta_{b,[\sigma]})_{b\in B}$. We denote by $\widehat{\pi}_{[\sigma]}\colon \widehat{\mathcal{C}}_{[\sigma]}\to B$ the associated family of curves, i.e. the canonical cover of the family $\pi_{[\sigma]}$. Let $e$ be an edge of $\oGamma$, and let  $\widehat{e}\to e$ be  an edge of the covering graph mapping to $e$. Let $\widehat{\mathcal{X}}_{\widehat e}\subset \widehat{\mathcal{C}}_{[\sigma]}$ be an open neighborhood of the node $q_{\widehat e}$ corresponding to $\widehat{e}$. We fix a local isomorphism
\[
\widehat{\phi}\colon\widehat{\mathcal{X}}_{\widehat e}\overset{\cong}{\longrightarrow} \widehat{\mathbb{V}}_{e}=\{(\widehat u,\widehat v,b) \in B_\epsilon(0)^2 \times B : \widehat u \widehat v = \widehat f_e(b)\}
\]
to a plumbing fixture  such that the restriction of the family $(\widehat{\eta}_{b,[\sigma]})_{b\in B}$ to $\widehat{\mathcal{X}}_{\widehat e}$ is the pull-back of the family of differentials in the standard form
\begin{equation} \label{eqn:normal_form_etahat_bsigma}
\widehat{\eta}^+_{b,[\sigma]} = (\widehat u^{\widehat{\kappa}_e} + \widehat f_e^{\widehat{\kappa}_e} r) \frac{d\widehat u}{\widehat u} \text{ and }\widehat{\eta}^-_{b,[\sigma]} = - (\widehat v^{-\widehat{\kappa}_e} + r) \frac{d\widehat v}{\widehat v }\,
\end{equation}
where $\widehat{\kappa}_e=\frac{\kappa_e}{\gcd(\kappa_e,k)}$ and $r\in \cO_{B}$. Here we mean that $\widehat{\eta}^+_{b,[\sigma]}$ and $ \widehat{\eta}^-_{b,[\sigma]}$ are two rescaled versions of the same differential which pull back to $\widehat{\eta}_{b,[\sigma]}$ and $\widehat f_e^{-\widehat{\kappa}_e} \widehat{\eta}_{b,[\sigma]}$ under the isomorphism $\widehat{\phi}$ above. Putting $\widehat{\eta}_b$ in this normal form \eqref{eqn:normal_form_etahat_bsigma} is possible thanks to the standard form shown in \cite[Theorem 4.3]{BCGGM3} for abelian differentials. 

In these coordinates,  the prong matching  has the form 
\begin{equation}\label{eq:prongvalue}\sigma_{\widehat{e}} = \widehat{\phi}^*(d\widehat u \otimes d\widehat v) \in T_{q^{(0)}}^* \widehat{C}_0 \otimes T_{q^{(-1)}}^* \widehat{C}_{-1}\end{equation}
(see e.g. ~\cite[Section 11.1, above Definition 11.5]{BCGGM3}), where $q^{(i)}$ is the preimage of the node  $q_{\widehat e}$ under normalisation, and supported on $\widehat{C}_i$ for $i=0$ or $-1$.

Let $G=\mathbb{Z}/k \mathbb{Z}$ be the cyclic group acting on the curve $\widehat{\mathcal C}_{[\sigma]}$ and let $\tau \in G$ be a generator of the stabilizer group of $q_{\widehat e}$. We distinguish two cases:
\begin{enumerate}
    \item[(i)] If $k | \kappa_e$, then the $G$-stabilizer of $q_{\widehat e}$ is trivial (so $\tau = \mathrm{id}$), and the preimage of the node associated to $e$ has $k$ elements.
    \item[(ii)] If $k \nmid \kappa_e$, then we claim first that $r=0$. Indeed, for $b \in B$ the value $r(b)$ is the integral of  $\widehat{f}^{-\widehat{\kappa}_e}_e \widehat{\eta}_{b,[\sigma]}$ along a simple vanishing cycle $\gamma$. If $\tau$ is a non-trivial element in the stabilizer of $q_{\widehat{e}}$ then it maps $\gamma$ to a loop in the same homotopy class. However, $\tau^*\widehat{\eta}_{b,[\sigma]}=\zeta_k\widehat{\eta}_{b,[\sigma]}$ (for a non trivial $k$-th root of unity $\zeta_k$), so we have $r(b)=\zeta_k r(b)$, thus $r(b)=0$. Then as in the proof of \cite[Lemma 4.4]{BCGGM2} the action of $\tau$ is given by $\tau(\widehat u, \widehat v, b) = (\zeta_k\widehat u, \zeta_k^{-1} \widehat v, b)$ for a suitable $k$-th root of unity $\rho_k$. Up to possibly shrinking $\widehat{\mathcal X}_{\widehat e}$, we can ensure that it is invariant under the action of $\tau$.
\end{enumerate}

We fix a primitive $\widehat{\kappa}_e$-th root of unity $\rho\coloneqq e^{2i\pi/\widehat{\kappa}_e}$. \rcommentfour{We will now construct a new family of multi-scale $k$-differentials by equivariantly twisting the plumbing fixtures associated to $e$ by $\rho$, i.e. we describe the equivariant version of the twisting operation described in the $k=1$ case.} Let $\widehat{\mathcal{U}}^+$ be the family of open annuli $\widehat{\mathcal{U}}^+\subset \widehat{\mathcal{X}}_{\widehat e}$ defined as $\widehat{\phi}^{-1}(\{r^-<|\widehat{u}|<r^+\})$ for sufficiently small real numbers $r^-,r^+$. We denote by $\partial^{+},$ and $\partial^{-}$ the boundaries of these annuli.  Note that in both cases (i) and (ii) above, the annuli $\widehat{\mathcal{U}}^+$ and their boundaries $\partial^\pm$ are invariant under $\tau$. Let $G \cdot \widehat{\mathcal{U}}^+$ and $G \cdot \partial^\pm$ be the respective $G$-orbits, which are disjoint unions of $\gcd(k, \kappa_e)$ copies of $\widehat{\mathcal{U}}^+$ and $\partial^\pm$, respectively.

\rcommentfive{
We now want to construct a new family $\widehat{\mathcal{C}}_{t_{e}[\sigma]}$ by performing a local twisting operation around each of the nodes in the orbit $G \cdot q_{\widehat e}$. We first describe this operation in an informal way, then we provide a rigorous construction.}

\rcommentfive{
For the informal picture, note that we can see a small analytic neighbourhood of the closure of $\widehat{\mathcal U}^+$ as being glued from an upper and a lower part as in Figure \ref{Fig:twistedgluing}. Here the lower part $\widehat{\mathcal C}^-$ are points of that local neighbourhood which are below $\delta^+$, and similarly $\widehat{\mathcal C}^+$ are local points above $\delta^-$, where "above" and "below" reference the levels $0,-1$ of the central fiber. These two parts overlap along the common open set $\widehat{\mathcal U}^+$, and gluing these two copies along the identity of $\widehat{\mathcal U}^+$ we would recover the original family. To construct the twisted family $\widehat{\mathcal{C}}_{t_{e}[\sigma]}$, we instead glue the copy $\widehat{\mathcal U}^+ \subseteq \widehat{\mathcal C}^-$ to the copy $\widehat{\mathcal U}^+ \subseteq \widehat{\mathcal C}^+$ along the isomorphism
\[
\widehat{\mathcal U}^+ \to \widehat{\mathcal U}^+, (\widehat{u}, \widehat{v}, b) \mapsto (\rho \widehat{u}, \rho^{-1} \widehat{v}, b)\,.
\]
To ensure that the resulting family is still $G$-equivariant, we perform the same operation around all $G$-orbits of $\widehat{\mathcal U}^+$.}

\rcommentfive{\begin{remark}
    This operation is similar to the partial Dehn twists used, in particular, to define the Frenchel-Nielsen on moduli spaces of hyperbolic surfaces. The difference is that the surgery is performed along an open annulus in our setting instead of a simple closed curve for partial Dehn twists (e.g. $\delta^+$ or $\delta^-$). One can prove that both construction are equivalent, however patching surfaces along open sets is more natural from the point of view of descents, and provides easily the definition of the new complex structure and ($k$)-differential.
\end{remark}}

\rcommentfive{
To give a fully formal and rigorous construction, one can define the twisted family $\widehat{\mathcal{C}}_{t_{e}[\sigma]}$  as $$\left(\left(\widehat{\mathcal{C}}_{[\sigma]}\setminus G \cdot \partial^{+} \right) \sqcup \left(\widehat{\mathcal{C}}_{[\sigma]}\setminus G \cdot \partial^{-}\right)\right) \bigg/ \sim_\rho.$$
}
\rcommentfive{To specify the equivalence relation $\sim_\rho$, we denote elements in the disjoint union by tuples $(x,i)$ where $i=1,2$ is the index  specifying whether we are in $\widehat{\mathcal{C}}_{[\sigma]}\setminus G \cdot \partial^{+}$ or $\widehat{\mathcal{C}}_{[\sigma]}\setminus G \cdot \partial^{-}$, respectively. Then $\sim_\rho$ is the equivalence relation defined by 
\begin{itemize}
    \item $(x,1) \sim_\rho (x,2)$ if $x$ is in 
$\widehat{\mathcal{C}}_{[\sigma]}\setminus  \overline{G \cdot \widehat{\mathcal{U}}^+}$, and
\item $((\widehat{u}, \widehat{v}, b),1) \sim_\rho ((\rho \widehat{u}, \rho^{-1} \widehat v, b),2)$ if $(\widehat{u}, \widehat{v}, b)$ is a point in $G \cdot \widehat{\mathcal{U}}^+$.\footnote{Here we abuse notation by using the same symbols $\widehat{u}, \widehat{v}$ for coordinates on all elements of the $G$-orbit of $\widehat{\mathcal{U}}^+$ - we simply mean that one performs the same identification procedure on each open set in this orbit.}
\end{itemize}
Note that with this convention, we indeed have 
\begin{itemize}
    \item precisely one equivalence class $(x,1) \sim_\rho (x,2)$ for any point $x$ not in the orbit of the closure of $\widehat{\mathcal U}^+$,
    \item precisely one equivalence class $(x,1)$ for each $x \in G \cdot \partial^-$ and one equivalence class $(x,2)$ for each $x \in G \cdot \partial^+$,
    \item precisely one equivalence class $[(x,1)]$ for any $x$ in the orbit of $\widehat{\mathcal U}^+$.
\end{itemize}
Overall, one verifies that this formal quotient under $\sim_\rho$ exactly realizes the informal picture that is depicted in Figure \ref{Fig:twistedgluing}.
}

\begin{figure}[htb]
	\centering
\begin{tikzpicture} [scale=.50] 
		\draw    [](0,0) to[out=-55,in=55] (0,-5);
		\draw    [](4,0) to[out=-125,in=125] (4,-5);
		\node[] at (4, -3) {$\widehat{\mathcal{X}}_{\widehat{e}}$};
		\node[] at (2, 0.6) {$\delta^+$};
		\draw[dotted] (0,0) to[bend  right=10] (4,0);
		\draw[dotted] (0,0) to[bend  left=10] (4,0);
		\draw[dotted] (0.52,-1) to[bend  right=10] (3.48,-1);
		\draw[dotted] (0.52,-1) to[bend  left=10] (3.48,-1);
		\draw    [](0,-5) to[out=-125,in=0] (-2,-6);
	\draw    [](4,-5) to[out=-55,in=-180] (6,-6);
	\draw[] (0,-6.5) to[bend  right=30] (1.5,-6.5);
	\draw[] (0.2,-6.6) to[bend  left=20] (1.3,-6.6);
	\draw[] (3,-6.5) to[bend  right=30] (4.5,-6.5);
\draw[] (3.2,-6.6) to[bend  left=20] (4.3,-6.6);
		\node[] at (-2, -5) {$\widehat{\mathcal{C}}^-$};
		\draw    [](15,0) to[out=-55,in=108] (15.52,-1);
		\draw    [](19,0) to[out=-125,in=63]  (18.45,-1);
		\draw[dotted] (15,0) to[bend  right=10] (19,0);
			\draw[dotted] (15,0) to[bend  left=10] (19,0);
		\draw[dotted] (15.52,-1) to[bend  right=10] (18.45,-1);
		\draw[dotted] (15.52,-1) to[bend  left=10] (18.45,-1);
		\draw    [](15,0) to[out=125,in=0] (13,1);
		\draw    [](19,0) to[out=55,in=180] (21,1);
				\node[] at (17, -1.5) {$\delta^-$};
					\draw[] (15,1.5) to[bend  right=30] (16.5,1.5);
				\draw[] (15.2,1.4) to[bend  left=20] (16.3,1.4);
				\draw[] (18,1.5) to[bend  right=30] (19.5,1.5);
				\draw[] (18.2,1.4) to[bend  left=20] (19.3,1.4);
		\node[] at (21, 0) {$\widehat{\mathcal{C}}^+$};
		\draw [<->] (5.5,-0.5) -- (13,-0.5);
		\node[] at (9.3, 0) {$(\widehat{u}, \widehat{v}, b) \sim (\rho \widehat{u}, \rho^{-1} \widehat v, b)$};
  \node[] at (5, -0.3) {$\widehat{\mathcal{U}}^+$};
  \node[] at (13.8, -0.3) {$\widehat{\mathcal{U}}^+$};
	\end{tikzpicture}	
 \caption{The local picture of the gluing construction at one node associated to a preimage $\widehat{e}$ of the edge $e$.} \label{Fig:twistedgluing}
\end{figure}

We claim that for this new twisted family of curves $\widehat{\mathcal{C}}_{t_{e}[\sigma]}$, we obtain a family of multi-scale $k$-differentials\rcommentfour{, i.e. a family of multi-scale 1-differentials with an action of the cyclic group $G$ for which the family is equivariant}. For this, first observe that the family of  differentials $(\widehat{\eta}_{b,[\sigma]})$ descends to the twisted family since the normal form \eqref{eqn:normal_form_etahat_bsigma} is invariant under the equivalence relation $\sim_\rho$. Furthermore, the action of $G$ descends to $\widehat{\mathcal{C}}_{t_{e}[\sigma]}$. Indeed, the only thing to check is that $\sim_\rho$ is compatible with the action of the stabilizer $\tau$ of the node $q_{\widehat e}$ (since any element of $G$ acts by a power of $\tau$ composed with an element simply permuting the copies of $\widehat{\mathcal{U}}^+$). In case (i) we have that $\tau$ is the identity so the compatibility is trivial. In case (ii) we have that $\tau$ acts by multiplication with $\zeta_k^{\pm 1}$ on $\widehat u, \widehat v$ and $\sim_\rho$ is defined via a similar multiplication with $\rho^\pm$, which commute and thus are compatible.

We also remark that the central fiber of the twisted family $\widehat{\mathcal{C}}_{t_{e}[\sigma]}$ is still isomorphic to $C_0\cup C_1$, the central fiber of the original family. In this isomorphism, we simply rotated a union of disks in the upper component $C_0$   around the preimages of $G \cdot q_{\widehat e}$.

Finally, we claim that the local prong-matchings for the twisted family $(t_{e}(\sigma))_{\widehat{e}}$ are $\rho \sigma_{\widehat{e}}$ at $q_{\widehat e}$ for all preimages $\widehat e$ of $e$. Indeed, the new family still contains an open set $\widehat{\mathcal X}^t_{\widehat{e}}$ isomorphic to the plumbing fixture $ \widehat{\mathbb{V}}_e$. Since on the plumbing fixture the new prong matching is still given by $d \widehat u \otimes d \widehat v$ and   the restriction to the $b=0$ locus of the new identification $\widehat{\phi}^t\colon\widehat{\mathcal{X}}^t_{\widehat e}\overset{\cong}{\rightarrow}  \widehat{\mathbb{V}}_{e}$  is changed by the above rotation,  we have  
\begin{equation}\label{eq:twistcomp}
(t_{e}(\sigma))_{\widehat{e}}=(\widehat{\phi}^t)^*(d \widehat u \otimes d \widehat v)=\rho \cdot \widehat{\phi}^*(d \widehat u \otimes d \widehat v)=\rho\sigma_{\widehat e}.
\end{equation}

Let $\widehat{\mathcal{C}}_{[t_e(\sigma)]}\to {\mathcal{C}}_{[t_e(\sigma)]}$ be the $k$-cover obtained by taking the quotient under the $G$-action, and consider the induced family of $k$-differentials $({\eta}_{b,[t_e(\sigma)]})$, whose canonical root on the canonical cover is $\widehat{\eta}_{b,[t_e(\sigma)]}$. 


To simplify the notation, we fix an edge $e$ of $\oGamma$ and denote by $t:=t_e$ the above twisting operation at $e$. As the central fibers of ${\mathcal{C}}_{[\sigma]}$ and ${\mathcal{C}}_{[t(\sigma)]}$ are isomorphic, we can choose the same differential $\eta'$ satisfying the constraints of the first part of  Lemma~\ref{lem:spin0}. By applying the second part of Lemma~\ref{lem:spin0} to these two families, and by repeating the above construction, we obtain two families of $1$-differentials $(\widetilde{\eta}_{b,[\sigma]})$ and $(\widetilde{\eta}_{b,[t(\sigma)]})$ deforming the same twisted differential $(C_i,\widetilde{\eta}_i)_{i=0,-1}$. We denote by $\widetilde{\sigma}$ and $\widetilde{t(\sigma)}$ the associated prong-matchings.

\begin{lemma}\label{lem:twistePM} 
The ratios of local prong-matchings for the families constructed above satisfy  
$$
\left(\frac{\sigma_{\widehat{e}}}{t(\sigma)_{\widehat{e}}}\right)^{k/\gcd(\kappa_e, k)}= \left(\frac{\widetilde{\sigma}_{{e}}}{\widetilde{t(\sigma)_e}}\right).
$$
for any preimage $\widehat{e}$ of $e$.
\end{lemma}
In particular, by taking both sides of this equality to the power $\kappa_e/2$, this lemma implies that $$\left(\frac{\sigma_{\widehat{e}}}{t(\sigma)_{\widehat{e}}}\right)^{\frac{\widehat{\kappa}_e}{2}}= \left(\frac{\widetilde{\sigma}_{{e}}}{\widetilde{t(\sigma)}_e}\right)^{\frac{{\kappa}_e}{2}}$$
because $(\sigma_{\widehat{e}}/t(\sigma)_{\widehat e})^{\widehat \kappa_e/2} = \rho^{\widehat \kappa_e/2}=-1$ and $k$ is odd (so $(-1)^k = -1$). Before proving this lemma, we finish the proof of Proposition~\ref{prop:multibanana} assuming that it holds.

\begin{proof}[Proof of Proposition~\ref{prop:multibanana} for general $k$] To prove the first part of the proposition, we simply write the chain of equalities
\[
\Phi(\mathcal C_{[\sigma],b}, \eta_{b,[\sigma]}) = \Phi(\mathcal C_{[\sigma],b}, \widetilde \eta_{b,[\sigma]}) \overset{(k=1)}{=} \sum_{i=-1}^{0} \Phi(C_i, \widetilde \eta_{i}) = \sum_{i=-1}^{0} \Phi(C_i, \eta_{i})\,,
\]
where the first equality is \eqref{eqn:spin_parity_on_smooth_fibers}, the  second  follows from the case $k=1$ of the proposition, and the last one follows from the identity $\Phi(C_i, \widetilde \eta_{i})=\Phi(C_i, \eta_{i})$ for $i=0,$ or $-1$.
The second part of the proposition follows from the following second chain of equalities 
\begin{align*}
(-1)^{\Phi(\sigma)-\Phi(\sigma')}=(-1)^{\Phi(\widetilde{\sigma})-\Phi(\widetilde{\sigma}')} & = \underset{\text{s.t. ${\kappa_e}$ 
 even}}{\prod_{e \in \Gamma,}}  \left(\frac{\widetilde{\sigma}_{{e}}}{\widetilde{\sigma}_{{e}}'}\right)^{\frac{{\kappa}_{e}}{2}}\\
 & =\underset{\text{s.t. ${\kappa_e}$ 
 even}}{\prod_{e \in \Gamma,}} \prod_{\widehat{e}\to e} \left(\frac{\sigma_{\widehat{e}}}{\sigma_{\widehat{e}}'}\right)^{\frac{\widehat{\kappa}_{e}}{2}}.
\end{align*}
The first equality follows from \eqref{eqn:spin_parity_on_smooth_fibers}, the second equality follows from the case $k=1$ of the proposition, while the last equality is a consequence of Lemma~\ref{lem:twistePM} and the fact that each $e$ has $k/{\rm gcd}(\kappa_e,k)$  (which is odd) pre-images that contribute with the same sign in the last product.
\end{proof}

\begin{proof}[Proof of Lemma~\ref{lem:twistePM}] We consider a neighborhood $\mathcal{X}_e$ of the node corresponding to $e$ in $\mathcal{C}_{[\sigma]}$. We can assume that this open set is contained in the image of $\widehat{\mathcal{X}}_{\widehat{e}}$ under the canonical cover. Then we can choose a first chart 
\[
\phi\colon {\mathcal{X}}_e \overset{\cong}{\longrightarrow} \mathbb{V}_{e}= \{(u, v,b) \in B_{\epsilon'}(0)^2 \times B : u v = f_e(b)\}
\]
such that the canonical cover at the level of plumbing fixtures is given by 
$$(\widehat u, \widehat v, b) \mapsto (u,v,b)= ( \widehat u^{k/\gcd(\kappa_e, k)},  \widehat v^{k/\gcd(\kappa_e, k)}, b)\,.$$ 
Indeed this follows from the form of the action of $\tau$ in the cases (i) and (ii) discussed above.
Next, we denote by 
$$\mathcal{U} = \phi^{-1}\left(\{ (r^-)^{k/\gcd(\kappa_e, k)}<|u|<(r^+)^{k/\gcd(\kappa_e, k)}\}\right)$$
the image of  the annulus $\widehat{\mathcal{U}}$ under the canonical cover.  
In these coordinates, the equivalence relation $(\widehat{u}, \widehat{v}, b) \sim_\rho (\rho \widehat{u}, \rho^{-1} \widehat v, b)$ descends to the equivalence relation $(u, v, b) \sim (\rho^{k/\gcd(\kappa_e, k)}  u, \rho^{-k/\gcd(\kappa_e, k)}  v, b)$.
Therefore the twisted family $\mathcal{C}_{[t(\sigma)]}$ is obtained from $\mathcal{C}_{[\sigma]}$ by a rotation of the annuli $\mathcal{U}$ by $\rho^{k/\gcd(\kappa_e, k)}$ in the  $u$-coordinate. 
In order to conclude, we need to show that the prong-matching $\widetilde{\sigma}$ at the node corresponding to $e$ is given by $$\widetilde{\sigma}_{{e}} = \phi^*(d u \otimes d v) \in T_{q^{(0)}}^* {C}_0 \otimes T_{q^{(-1)}}^* {C}_{-1}.$$
Indeed, if this is the case, then the same argument as in \eqref{eq:twistcomp} implies that $\widetilde{t(\sigma)}_{{e}}=\rho^{k/{\rm gcd}(\kappa_e,k)}\widetilde{\sigma}_{{e}}$ because of the twisting operation, and the Lemma follows. 

The difficulty is that the chart $\phi$ above does not put the family of differentials $(\widetilde{\eta}_{b,[\sigma]})_{b \in B}$ in the standard form so we cannot apply the local characterization of the prong-matching~\eqref{eq:prongvalue} directly. To remedy this issue, we fix a second chart 
\[
\widetilde{\phi}\colon \widetilde{\mathcal{X}}_e \overset{\cong}{\longrightarrow} \widetilde{\mathbb{V}}_{e}= \{(\widetilde u, \widetilde v,b) \in B_{\epsilon'}(0)^2 \times B :  \widetilde{u} \widetilde{v} =  f_e(b)\}
\]
such that the differential forms $\widetilde{\eta}_{b,[\sigma]}$ are given in these coordinates by 
\[
\widetilde{\eta}^+_{b,[\sigma]} = (\widetilde u^{{\kappa}_e} +  f_e^{\widetilde{\kappa}_e} r) \frac{d\widetilde u}{\widetilde u} \text{ and }\widetilde{\eta}^-_{b,[\sigma]} = - (\widetilde v^{{\kappa}_e} + r) \frac{d\widetilde v}{\widetilde v }.
\]
THen  the prong matching $\widetilde{\sigma}_e$ is given by 
\[\widetilde{\sigma}_e=\widetilde{\phi}^*(d\widetilde u\otimes d \widetilde v).\]
Note that in the definition of $\widetilde{\mathbb{V}}_{e}$ we can choose the \emph{same} smoothing parameter $f_e(b)$, since the geometry of the node has not changed.
To use these coordinates we consider the holomorphic function 
\[
(u,v,b)\mapsto \frac{\widetilde{u}}{u} = \frac{p_{\widetilde{u}}(\widetilde{\phi}\circ\phi^{-1}(u,v,b))}{u}
\]
(where $p_{\widetilde u}$ stands for the projection on the first variable). This function is defined outside the locus $\{u=0\}$. 
However, we have 
$$
\frac{\widetilde{v}}{v}\frac{\widetilde{u}}{u}=1,
$$ 
if we define $\widetilde{v}/v$ similarly  outside $\{v=0\}$. As these two functions do not vanish, both extend first along the divisor $\{uv=0\}$ outside the origin. Finally, the two functions extend to the origin by Hartog's theorem, since $\mathcal{X}_e$ is normal (it is a surface with an $A_N$-singularity at the origin for some $N$). 
Solving the equation above we get
\[
\widetilde{u} = \frac{v}{\widetilde v} u \implies d \widetilde{u} = \left( \frac{v}{\widetilde v}\right) du + u \cdot d\left( \frac{v}{\widetilde v}\right)
\]
and similarly
\[
d \widetilde{v} = \left( \frac{u}{\widetilde u}\right) dv + v \cdot d\left( \frac{u}{\widetilde u}\right)\,.
\]
Evaluating at $(u,v,b)=(0,0,0)$ we see that the terms $u d(v/\widetilde v)$ and $v d(u/\widetilde u)$ vanish and taking the tensor product we have
\[
d \widetilde u \otimes d \widetilde v|_{(0,0,0)} = \underbrace{\left( \frac{v}{\widetilde v} \cdot \frac{u}{\widetilde u} \right)}_{=1} du \otimes dv|_{(0,0,0)} \in (T_{(0,0,0)}^* \mathbb{V}_e)^{\otimes 2}\,.
\]
Pulling this equality back under $\phi$ shows the desired identity $\phi^* (du\otimes dv)=\widetilde \phi^* (d\widetilde{u}\otimes d\widetilde{v})=\widetilde{\sigma}_e$ at the origin.
\end{proof}

\subsection{Splitting formulas with spin parity}

Using the result of the previous section, we prove a splitting formula for $\psi$-classes that takes into account the spin parity. Recall that ${\rm LG}_1^2(g,a)$ is the set of level graphs with exactly 2 vertices $v_0$ of level 0 and $v_{-1}$ of level $-1$, and no horizontal edges. 

\begin{lemma} \label{Lem:kdiffspinsplitting}
We assume that $k$ and the $a_i$'s are all odd, and we chose $1\leq s<t\leq n$ such that $k$ does not divide $s$ or $t$. Then, we have the following relation:
\begin{align*}
    &(a_s \psi_s - a_t \psi_t)\cdot [\oM_g(a)]^\sp= \sum_{(\Gamma,I)\in {\rm LG}_1^2(g,a)}\!\!\!\! \beta(s,t,\Gamma,I),
\end{align*}
where 
\begin{eqnarray*}
     \beta(s,t,\Gamma,I)&=& f_{s,t}(\Gamma,I) \frac{m(\Gamma,I)}{|{\rm Aut}(\Gamma,I)|}   \zeta_{\Gamma *}\left([\oM_{g_0}(I(v_0))]^\sp \otimes [\oM_{g_{-1}}(I(v_{-1}))]^\sp\right) 
\end{eqnarray*}
if $I$ has only odd entries, and $\beta(s,t,\Gamma,I)=0$ otherwise.
\end{lemma}

\begin{proof}
We consider the components $\PP\Xi_g(a)^\odd$ or $\PP\Xi_g(a)^{\rm even}$  of the moduli space of multi-scale differentials. Using the same arguments as in the non-spin case of Lemma \ref{Lem:kdiffsplitting}, we can use Proposition \ref{prop:Adrienrel} in order to  write  $(\eta+a_s\psi_s)$ and $(\eta+a_t\psi_t)$ on $\PP\Xi_g(a)^\odd$ or $\PP\Xi_g(a)^{\rm even}$ as a linear combination of components of the boundary divisors indexed by twisted graphs  with 2 levels. As in the non-spin case, the contribution of a graph $(\Gamma,I)$ vanishes if the graph is not in in ${\rm LG}_1^2(g,a)$.

Fix now a divisor $(\Gamma,I)\in {\rm LG}_1^2(g,a)$.  Thanks to Proposition~\ref{prop:multibanana}, if all entries of $I$ are odd, then the parity of a connected component of this divisor is equal to the sum of the parities of the projections of this connected component in $\M_{g_0}(a_0,I_0)$ and $ \M_{g_{-1}}(a_{-1},I_{-1})$. Thus the contribution of $(\Gamma,I)$ to $(a_s\psi_s-a_t\psi_t)[\oM_{g}(a)]^\sp$ is given by:
$$
f_{s,t}(\Gamma,I) \frac{m(\Gamma,I)}{|{\rm Aut}(\Gamma,I)|} \cdot  \zeta_{\Gamma *}\left([\oM_{g_0}(a_0,I_0)]^\sp \otimes [\oM_{g_{-1}}(a_{-1},I_{-1})]^\sp\right).
$$

If there is at least an even entry of $I$, then by Corollary ~\ref{Cor:prong_balance} half of the prong-matchings equivalence classes give a nearby even or odd differential. 
Since each choice of prong-matchings equivalence classes contributes with the same amount to the difference of the psi-classes on each component, the contribution of $(\Gamma,I)$ to $(a_s\psi_s-a_t\psi_t)[\oM_{g}(a)]^\sp$ is trivial as the even and odd contributions cancel out.
\end{proof}

 Using Assumption~\ref{assumption} we denote by $\cA_g^\sp$ the unique polynomial extending $\cA_g^\sp$. We will prove the following refinement of the identities of Lemma~\ref{lem:identities} taking into account the spin parity. 

\begin{lemma}\label{lem:identitiesspin} For all even $k$, and all vectors $a$ with even entries, we have:
\begin{eqnarray}
\cA_{g}^\sp(a,k) &=&  \cA_{g}^\sp(a)\quad  \text{ for }k=\frac{a_1+\ldots + a_n}{2g-2+n},\label{eqn:Agspinid1}\\
a_1\cdot \cA_{g}^\sp(a,0)\,\, + && \!\!\!\!\!\!\!\!\!\!\!\!\!  \sum_{i>1} (a_i-k) \cA_{g}^\sp(\ldots,a_i-k,\ldots) \label{eqn:Agspinid2}\\
&=& \frac{1}{2}\sum_{\begin{smallmatrix}j \,\,{\rm odd},\\ \nonumber 0<j<k\end{smallmatrix}} j(k-j) \cA_{g-1}^{\rm spin}(a,-j,j-k)\\\nonumber
 \quad  \text{ for }k=\frac{a_1+\ldots + a_n}{2g-2+n+1},&& \nonumber
\end{eqnarray}
\begin{eqnarray}
&& (2g-2+n)a_1\cA_{g}^\sp(a)_{|k=0}   \label{eqn:Agspinid3} \\ \nonumber = &&  - \sum_{j>i>1} (a_i+a_j) \cA_g^\sp(\ldots,\widehat{a_i},\ldots,\widehat{a_j},\ldots,a_i+a_j) \\
  && -\frac{1}{2} \!\! 
  \sum_{1 < i \leq n}
  \sum_{\begin{smallmatrix}j,\ell \,\,{\rm odd},\\ \nonumber j\cdot \ell>0, j+\ell=a_i \end{smallmatrix}} \!\!\!\! {\rm sign}(a_i) \cdot j\ell \cdot \cA_{g-1}^{\rm spin}(\ldots, \widehat{a_i}, \ldots, a_n,j,\ell) + R(a)\,,
\end{eqnarray}
where $R(a)$ is the evaluation of the polynomial expression 
\[
\frac{1}{6}  \sum_{\begin{smallmatrix}j,\ell,m \,\,{\rm odd}, \\  \mathrm{sign}(j)=\mathrm{sign}(\ell)=\mathrm{sign}(m),\\  j+\ell+m=b \end{smallmatrix}} \!\!\!\! j\ell m \cdot \cA_{g-2}^{\rm spin}(a_1, \ldots,a_n,j,\ell,m)
\]
in $b$ at $b=0$.
\end{lemma}
Concering equation \eqref{eqn:Agspinid3}, it will turn out a posteriori from our formula that the term $R(a)$ vanishes. However, this is not obvious from the basic properties of $\cA_g$, so for now we state the lemma in the slightly complicated version above.
\begin{proof}
First, we note that Lemma~\ref{Lem:kdiffspinsplitting} is only valid for odd values of $a$ and $k$. Moreover, by definition
$$\cA_{g}^\sp(a)= \rcomment{\int_{{\oM}_{g,n}}}\psi_1^{2g-3+n}\cdot [\oM_g(a)]^\sp$$
if $a_1$ is either negative or not divisible by $k$.
For all three identities above, the strategy is to prove the corresponding identity with these constraints and write
their polynomial extension for even values of $a$.

Identity~\eqref{eqn:Agspinid1} follows immediately from the property $\pi^* \DR_g^\sp(a) = \DR_g^\sp(a,k)$ from Assumption \ref{assumption}.\footnote{Alternatively, it can be proven using only the polynomiality of $\cA_g^\mathrm{spin}$ and property (3) from Assumption \ref{assumption}, since for $a,k$ odd, the cycle $[\oM_{g}(a,k)]^\sp$ is the pull-back of $[\oM_{g}(a)]^\sp$. Note however that for a later part of the argument we need the full strength of part (2) of Assumption \ref{assumption}, and thus we cannot omit it.}

To prove identity~\eqref{eqn:Agspinid3}, first note that the identity is trivial in case $n=1$: the restriction $k=0$ forces $a=(0)$, so that the left-hand side vanishes due to the factor $a_1=0$ and the right-hand side vanishes since the index-sets of the sum are empty. Thus we may assume $n \geq 2$. 

To prove \eqref{eqn:Agspinid3}, we will first perform a calculation on a moduli space with one additional marking. Consider vectors $a' \in \mathbb{Z}^{n+1}$ with $|a'|=k(2g-1+n)$ so that both $k$ and all entries of $a'$ are odd and such that both $a_1'$ and $a_{n+1}'$ are not divisible by $k$. 
Then we can use Lemma~\ref{Lem:kdiffspinsplitting} above with $s=1$, and $t=n+1$ to get the following expression:
\begin{equation} \label{eqn:spinhelp1}
(a'_1\psi_1  - a_{n+1}' \psi_{n+1})  [\oM_{g}(a')]^\sp= \sum_{(\Gamma,I)} \beta(1,n+1,\Gamma,I).
\end{equation}
When taking the intersection number of \eqref{eqn:spinhelp1} with $\psi_{n+1}\psi_1^{2g-4+n}$, we note that on the left-hand side we can replace the cycle $[\oM_{g}(a')]^\sp$ by $\DR_g^\sp(a')$ by Assumption \ref{assumption}. On the other hand, for the right hand side
almost all terms $\beta(1,n+1,\Gamma,I)$
contribute zero for dimension reasons. Going through all possible pairs $(\Gamma, I)$, one sees that only four types of graphs allow nonzero intersection numbers:
\begin{itemize}
    \item two vertices, connected by a single edge such that one vertex has genus zero and carries markings $i,j,n+1$ (for $1<i<j<n+1$),
    \item two vertices, connected by two edges such that one vertex has genus zero and carries markings $i, n+1$ (for $1 < i \leq n$),
    \item two vertices, connected by three edges such that one vertex has genus zero and carries marking $n+1$,
    \item two vertices, connected by a single edge such that one vertex has genus one and carries marking $n+1$.
\end{itemize}
Using the explicit description of $ \beta(1,n+1,\Gamma,I)$ given in Lemma~\ref{Lem:kdiffspinsplitting}, by pairing  \eqref{eqn:spinhelp1} with $\psi_{n+1}\psi_1^{2g-4+n}$ we get the following expression:
\begin{eqnarray}\label{eqn:Rap}
&&  \int_{\oM_{g,n+1}} \psi_{n+1}\psi_1^{2g-4+n}(a_1' \psi_1  - a_{n+1}' \psi_{n+1})  \DR_g^\sp(a') \label{eqn:spinhelp2} \\  
= && \!\!  - \!\! \sum_{j>i>1} (a_i'+a_j'+a_{n+1}'-2k) \cA_g^\sp(\ldots,\widehat{a_i'},\ldots,\widehat{a_j'},\ldots,a_i'+a_j'+a_{n+1}'-2k)\nonumber \\
&&  - \frac{1}{2}  \sum_{1<i\leq n} \!\!\!\! \sum_{\begin{smallmatrix}j,\ell \,\,{\rm odd}, \nonumber\\ \nonumber j\cdot \ell>0, j+\ell=a_i'+a_{n+1}'-2k \end{smallmatrix}} \!\!\!\! {\rm sign}(a_i'+a_{n+1}'-2k) \cdot j\ell \cdot \cA_{g-1}^{\rm spin}(\ldots,\widehat{a_i'},\ldots,j,\ell)
\\
&& \nonumber - (a_{n+1}'-2k) \cdot \cA_{1}^{\rm spin}(a_{n+1}', 2k-a_{n+1}') \cA_{g-1}^{\rm spin}(a_1', \ldots, a_n',-2k+a_{n+1}') \nonumber
\\
&& + R(a')\,, \nonumber
\end{eqnarray}
where
\begin{equation} \label{eqn:Raprime}
    R(a') = \frac{1}{6}  \sum_{\begin{smallmatrix}j,\ell,m \,\,{\rm odd}, \\  \mathrm{sign}(j)=\mathrm{sign}(\ell)=\mathrm{sign}(m),\\  j+\ell+m=a_{n+1}'-2k \end{smallmatrix}} \!\!\!\! j\ell m \cdot \cA_{g-2}^{\rm spin}(a_1', \ldots,a_n',j,\ell,m)\,.
\end{equation}
\rcomment{The sign conditions for the indexes of the sums in the RHS of~\eqref{eqn:Rap} and \eqref{eqn:Raprime} follow from the definition of 2-twisted graphs, while the sign in front of the sums are determined by the sign function $f_{s,t}$ appearing  in Lemma~\ref{Lem:kdiffspinsplitting} and defined in Proposition~\ref{prop:psiDRformula} (these signs are the same as the non-spin case by~\cite{BSSZ} Corollary~2.1)}
From Assumption \ref{assumption} and Lemma~\ref{lem:polynomialodd} it follows that 
all terms in the equality \eqref{eqn:spinhelp2} except possibly $R(a')$ are polynomial on vectors $a'$ such that all numbers $a_i'+a_{n+1}'-2k$ are even, and of fixed signs, whereas $R(a')$ is only polynomial for vectors $a'$ satisfying the slightly stronger condition that the entire vector $a'$ is odd.

Thus, fixing a collection $\eta \in \{\pm 1\}^{n-1}$ of signs for the numbers $a_i'+a_{n+1}'-2k$ cuts out a polyhedron in the space $\mathbb{R}^{n+1}$ of all vectors $a'$ and we see that on the intersection $\Lambda_\eta$ of that polyhedron with the lattice  of vectors $a'$ such that the numbers $a_i'+a_{n+1}'-2k$ are even, the two sides of \eqref{eqn:spinhelp2} are given by a polynomial in $a'$, where $R(a')$ now really means the polynomial extending \eqref{eqn:Raprime}.
Moreover, the vectors $a'$ we considered above ($k,a_i'$ odd and $a_1', a_{n+1}'$ not divisible by $k$) are Zariski dense in each such domain, so that we have \eqref{eqn:spinhelp2} as an equality on each $\Lambda_\eta$.

Now let us return to equality \eqref{eqn:Agspinid3}.
Consider an even vector $a \in \mathbb{Z}^n$ with $|a|=0$ and so that none of the $a_i$ vanish. Again, since these are Zariski dense, it suffices to prove the statement for such vectors. Then we can consider \eqref{eqn:spinhelp2} at the point $a' = (a_1, \ldots, a_n,0)$, which satisfies $k=0$ and lies in $\Lambda_\eta$ for $\eta = (\mathrm{sign}(a_i))_{i=2}^n$. 
Then the right hand side of \eqref{eqn:spinhelp2} simplifies to the right hand side of \eqref{eqn:Agspinid3}.


On the other hand, for the left-hand side we have
\begin{align*}
& a_1 \int_{\oM_{g,n+1}} \psi_{n+1} \psi_1^{2g-3+n} \DR_g^\sp(a_1, \ldots, a_n,0)\\
= & a_1 (2g-2+n) \int_{\oM_{g,n}} \psi_1^{2g-3+n} \DR_g^\sp(a_1, \ldots, a_n)\\
=& a_1 (2g-2+n) \cA^\sp_g(a)\,,
\end{align*}
using the property $\pi^* \DR_g^\sp(a) = \DR_g^\sp(a,0)$ for the forgetful morphism $\pi: \oM_{g,n+1} \to \oM_{g,n}$ from Assumption \ref{assumption}. This concludes the proof of equality \eqref{eqn:Agspinid3}.

The proof of identity~\eqref{eqn:Agspinid2} uses a similar strategy, now pairing equation \eqref{eqn:spinhelp1} with $\psi_1^{2g-3+n}$, again for vectors $a'$ such that both $k$ and all $a_i'$ are odd, the numbers $a_1', a_{n+1}'$ are not divisible by $k$ and additionally such that $a_1<0$.\footnote{The condition on $a_1$ being negative ensures that certain graphs with nontrivial residue conditions cannot appear below.} The only two kinds of graphs that can contribute are
\begin{itemize}
    \item two vertices, connected by a single edge such that one vertex has genus zero and carries markings $i,n+1$ (for $1<i<j<n+1$),
    \item two vertices, connected by two edges such that one vertex has genus zero and carries marking $n+1$.
\end{itemize}
We then obtain the equality
\begin{eqnarray*}
&& \!\!\!\!\!\!\!\!\!\!\!\!\!\!\!\!\!\!\!\!\!\!\!\!\!\!\!\! a'_1 \cA_g^\sp(a')- \rcomment{\int_{\oM_{g,n+1}}} a'_{n+1} \psi_{n+1}\psi_1^{2g-3+n}  [\oM_{g}(a')]^\sp \\ 
= && \!\!  - \!\! \sum_{1<i\leq n} (a'_i+a'_{n+1}-k) \cA_g^\sp(\ldots,a'_i+a'_{n+1}-k,\ldots)\\
&&  - \frac{1}{2}  \sum_{\begin{smallmatrix}j \text{ odd},\\ 0<j<k-a_{n+1}' \end{smallmatrix}} \!\!\!\!  j(k-a_{n+1}'-j) \cdot \cA_{g-1}^{\rm spin}(a_1',\ldots, a_n',-j,j+a_{n+1}'-k)\,.
\end{eqnarray*}
Combining Assumption \ref{assumption} with Lemma \ref{lem:polynomialodd} we see that both sides of this equality are polynomial in vectors $a'$ for which $k-a_{n+1}'$ is even and that the equality holds on the level of these polynomials. Specializing to even $k$ and even vectors $a'=(a_1, \ldots, a_n, 0)$ produces equality \eqref{eqn:Agspinid2}.
\end{proof}

\begin{proposition}\label{prop:constraintspin}
The functions $\cA_{g}^\sp$ are determined by $\cA_0^\sp$, $\cA_1^\sp$, and the identities of Lemma~\ref{lem:identitiesspin}.
\end{proposition}

\begin{proof}
The arguments used to prove Proposition~\ref{prop:constraint} may be directly adapted to prove this Proposition. 
\end{proof}

As for Theorem~\ref{th:main}, the proof of Theorem~\ref{th:spin} is obtained by a combination of Lemma~\ref{lem:identitiesspin} and Proposition~\ref{prop:constraintspin} .
\begin{proof}[Proof of  Theorem~\ref{th:spin}]
We denote by 
\begin{eqnarray*}
B_{g}^\sp:\ZZ^n&\to& \QQ\\
a&\mapsto& 2^{-g} [z^{2g}] {\rm exp}\left(\frac{\rcomment{a_1}z\cdot \cS'(kz)}{\cS(kz)} \right) \frac{\ch(z/2)}{{\cS}(z)} \frac{\prod_{i>1} \cS(a_iz)}{\cS(kz)^{2g-1+n}}
\end{eqnarray*}
the polynomial of degree $2g$ defined by the RHS of formula~\eqref{eqn:mainspin}.
By direct computation we have $\cA_0^\sp=B_0^\sp$ and $\cA_1^\sp=B_1^\sp$, using the identity $[\oM_0(a)]^\sp = [\oM_{0,n}]$ and the formula \eqref{eqn:M1aspin} for $[\oM_1(a)]^\sp$ discussed in Section \ref{Sect:spinDR} \rcomment{which provides the numerical identity
$$
\cA_1^\sp= \cA_1(a)-2 \cA_1\left(\frac{a_1-k}{2},\ldots,\frac{a_n-k}{2}\right)\,.
$$}
\rcomment{The general equality $\cA_1^\sp=B_1^\sp$ then may either be proven by using hyperbolic identities or simply by checking that it holds for $n=3$ as in the non-spin case.}

Thus we will finish the proof by showing that the function $B_g^\sp$  satisfies the three relations of Lemma~\ref{lem:identitiesspin}. By Proposition~\ref{prop:constraintspin}, this implies that $B_{g}^\sp=\cA_{g}^\sp$ for all $g$. The fact that $B_g^{\rm spin}$ satisfies identity~\eqref{eqn:Agspinid1} is straightforward. 

We first show that $B_g^{\rm spin}$ satisfy identity~\eqref{eqn:Agspinid3}. 
\rcomment{
First, we claim that the term $R(a)|_{k=0}$ in \eqref{eqn:Agspinid3} vanishes. Recall that $R(a)$ is given as the evaluation of
\[
\frac{1}{6}  \sum_{\begin{smallmatrix}j,\ell,m \,\,{\rm odd}, \\  \mathrm{sign}(j)=\mathrm{sign}(\ell)=\mathrm{sign}(m),\\  j+\ell+m=b \end{smallmatrix}} \!\!\!\! {j\ell m \cdot B_{g-2}^{\rm spin}(a_1, \ldots,a_n,j,\ell,m)}
\]
at $b=0$. To argue that this expression vanishes after specializing to $k=|a|/(2g-2+n)=0$ and $b=0$, we first note that $B_{g-2}^{\rm spin}(a_1, \ldots, a_n, j, \ell, m)$ is given by
\[
2^{-(g-2)} [z^{2(g-2)}] {\rm exp}\left(\frac{\rcomment{a_1}z\cdot \cS'(k'z)}{\cS(k'z)} \right) \frac{\ch(z/2)}{{\cS}(z)} \frac{\prod_{i>1} \cS(a_iz) \cdot \cS(jz) \cS(\ell z) \cS(mz)}{\cS(k'z)^{2(g-2)-1+n}}
\]
with
\[
k' = \frac{|a|+j+\ell+m}{2(g-2)-2+(n+3)} = \frac{2g-2+n}{2(g-2)-2+(n+3)} \cdot k + \frac{1}{2(g-2)-2+(n+3)} b\,,
\]
where again $b=j + \ell + m$.
Setting $k=b=0$ commutes with the summation in the definition of $R(a)$.\footnote{\rcomment{To be more precise, when taking a summation over $j+\ell+m=b$ and setting $b=0$ as in the definition of $R(a)$, any modification of the summand by a polynomial divisible by $b$ does not change the result. In practice, this allows us to set $b=0$ in any part of the formula of the summand which only depends on $b$.}} Since this substitution implies $k'=0$, the formula simplifies substantially and we get
\[
B_{g-2}^{\rm spin}(a, j, \ell, m)|_{k=b=0} = 2^{-(g-2)} [z^{2(g-2)}] \frac{\ch(z/2)}{{\cS}(z)} \prod_{i>1} \cS(a_iz) \cdot \cS(jz) \cS(\ell z) \cS(mz)\,.
\]
In this formula, the variables $j, \ell, m$ appear with even degrees, so that in the expression
\[
P(a_1, \ldots, a_n, j, \ell, m) = j\ell m \cdot B_{g-2}^{\rm spin}(a_1, \ldots,a_n,j,\ell,m)|_{k=b=0}
\]
all monomials have odd degree in $j, \ell, m$. This implies that
\[
Q(a,b) = \frac{1}{6}  \sum_{\begin{smallmatrix}j,\ell,m \,\,{\rm odd}, \\  \mathrm{sign}(j)=\mathrm{sign}(\ell)=\mathrm{sign}(m),\\  j+\ell+m=b \end{smallmatrix}} \!\!\!\! P(a_1, \ldots, a_n, j, \ell, m)
\]
is divisible by $b$ using the last statement in Lemma \ref{lem:polynomialodd}. Then $$R(a)|_{k=0} = Q(a,b)|_{k=b=0} = 0$$ as desired.
}


For the first term on the RHS of identity~\eqref{eqn:Agspinid3}, we use addition formulas of the hyperbolic sine and the form of $B_g^\sp$ for $k=0$ to obtain: 
\begin{eqnarray*} &&
 \sum_{j>i>1} (a_i+a_j)  B_g^\sp(\ldots,\widehat{a_i},\ldots,\widehat{a_j},\ldots,a_i+a_j) \\
 =&& -2^{-g} [z^{2g}] \frac{{\rm cosh}\left(\frac{z}{2}\right)}{\mathcal{S}(z)}  \sum_{i>1} (a_i+a_1) {\rm cosh}\left(\frac{a_iz}{2}\right)  \prod_{j\neq i,1} \cS(a_jz) \,.
\end{eqnarray*}
For the second term, we use the identity
\begin{equation} \label{eqn:sinhsumformula2}
\mathrm{sign}(m) \sum_{\begin{smallmatrix}j,\ell \,\,{\rm odd},\\ j\cdot \ell>0, j+\ell=m \end{smallmatrix}} \frac{j\ell}{2} \cS(jz) \cS(\ell z) = \frac{1}{z^2} \left(\frac{m}{2} \cosh(mz/2)  - \frac{\sinh(mz/2)}{\sinh(z)}\right)\,,
\end{equation}
which can be shown in a way analogous to the proof of identity \eqref{eqn:sinhsumformula}. Then we have
\begin{eqnarray*}
 &&
 \frac{1}{2}  \sum_{i>1}  \sum_{\begin{smallmatrix}j,\ell \,\,{\rm odd},\\ \nonumber j\cdot \ell>0, j+\ell=a_i \end{smallmatrix}} \!\!\!\! {\rm sign}(a_i) \cdot j\ell \cdot B_{g-1}^{\rm spin}(\ldots, \widehat{a_i}, \ldots, j,\ell)\\
 =&& 2^{-g+1} [z^{2g}] \frac{{\rm cosh}\left(\frac{z}{2}\right)}{\mathcal{S}(z)} \sum_{i>1} \left(\frac{a_i}{2} {\rm cosh}\left(\frac{a_iz}{2}\right) - \frac{{\rm sinh}\left(\frac{a_iz}{2}\right) }{{\rm sinh}(z) }\right) \prod_{j\neq 1,i} \cS(a_jz)
\end{eqnarray*}
The sum of these two terms gives: 
\begin{eqnarray*} 
 &&2^{-g} [z^{2g}]  \sum_{i>1} \left( - \frac{ {\rm sinh}\left(\frac{a_iz}{2}\right) }{{ \rm sinh}\left(\frac{z}{2}\right){ \rm cosh}\left(\frac{z}{2}\right) } -a_1{\rm cosh}\left(\frac{a_iz}{2}\right) \right) \frac{{\rm cosh}\left(\frac{z}{2}\right)}{\mathcal{S}(z)} \prod_{j\neq 1,i} \cS(a_jz)\\
&=&\rcomment{2^{-g} [z^{2g}]  \sum_{i>1} \left(  \frac{ a_i }{{ \rm sinh}\left(\frac{z}{2}\right){ \rm cosh}\left(\frac{z}{2}\right) } - a_1\frac{{\rm cosh}\left(\frac{a_iz}{2}\right)}{\cS(a_iz)}\right) \frac{{\rm cosh}\left(\frac{z}{2}\right)}{\mathcal{S}(z)} \prod_{j\neq 1} \cS(a_jz)}\\
&=&\rcomment{2^{-g}a_1 [z^{2g}]   \left( \frac{ z/2 }{{ \rm sinh}\left(\frac{z}{2}\right){ \rm cosh}\left(\frac{z}{2}\right)} - \sum_{i>1}\left(\frac{{\rm cosh}\left(\frac{a_iz}{2}\right)-\cS(a_iz)}{\cS(a_iz)} +1\right) \right) \frac{{\rm cosh}\left(\frac{z}{2}\right)}{\mathcal{S}(z)} \prod_{j\neq 1} \cS(a_jz)}\\
&=&\rcomment{2^{-g}a_1 [z^{2g}]   \left( 1+ z\frac{\cS'(z)}{\cS(z)} - \frac{z\, {\rm sinh}\left(\frac{z}{2}\right)}{2{ \rm cosh}\left(\frac{z}{2}\right)} - (n-1) - \sum_{i>1} z\frac{a_i\cS'(a_iz)}{\cS(a_iz)}  \right) \frac{{\rm cosh}\left(\frac{z}{2}\right)}{\mathcal{S}(z)} \prod_{j\neq 1} \cS(a_jz)}\\
 &=& - 2^{-g} a_1 [z^{2g}]  \left(z \frac{d}{dz}+n-2\right) \frac{{\rm cosh}\left(\frac{z}{2}\right)}{\mathcal{S}(z)} \prod_{i>1} \cS(a_iz)\\
 &=& - 2^{-g} a_1 (2g-2+n) [z^{2g}]  \frac{{\rm cosh}\left(\frac{z}{2}\right)}{\mathcal{S}(z)} \prod_{i>1} \cS(a_iz)
\end{eqnarray*}
which is the opposite of the LHS of identity~\eqref{eqn:Agspinid3}.  

Finally, to check that the numbers $B_g^\sp(a)$ satisfy the identity~\eqref{eqn:Agspinid2} we proceed as in the non-spin case. For all $a\in \ZZ^n$, we by introduce the formal series:
\begin{eqnarray*}
 B_{g}^\sp (a)(z)
 &=& 2^{-g}  {\rm exp}\left(\frac{\rcomment{a_1}z\cdot \cS'(kz)}{\cS(kz)} \right) \frac{\ch(z/2)}{{\cS}(z)} \frac{\prod_{i>1} \cS(a_iz)}{\cS(kz)^{2g-2+n}}.
\end{eqnarray*}
Then the derivative of this function is similar to the derivative of $B_g(a)(z)$
\begin{eqnarray*}
 \frac{d}{dz}B_{g}^\sp (a)(z)&=& \rcomment{ {B_{g}^\sp(a)(z)} \cdot \Bigg( \frac{a_1}{kz}-\frac{a_1}{k { z} \cS(kz)^2} - (k(2g-2+n)-a_1) \frac{\cS'(kz)}{\cS(kz)}} \\   &&\rcomment{+\frac{1}{z}- \frac{1}{{\rm sinh}(z)} + \sum_{i>1} a_i \frac{\cS'(a_iz)}{\cS(a_iz)}\Bigg).}
\end{eqnarray*}
Then we can check that:
\begin{eqnarray*}
-k [z^{2g}]\!\!\! \!\!\! && \!\!\!\!\!\!z^{2g+1} \frac{d}{dz} z^{-2g} B_g^\sp(a)(z) \\&= & a_1 B_{g}^\sp(a,0) + \sum_{i>1} (a_j-k) B_g^\sp(\ldots,a_j-k,\ldots) \\&& \,\,\,\, -\frac{1}{2}\sum_{\begin{smallmatrix} j \,\,{\rm odd},\\  0<j<k\end{smallmatrix}} j(k-j) B_{g-1}^{\rm spin}(\ldots,-j,j-k).
\end{eqnarray*}
The left hand side vanishes, and the vanishing of the RHS shows that $B_g$ satisfies identity~\ref{eqn:Agspinid2}.
\end{proof}

\section{Euler characteristic of the minimal strata} \label{Sect:Eulerchar}

In this section, we will work only under the assumption  $k=1$. We will show how to obtain the formula of Euler characteristic of minimal strata and of the spin refinement of this Euler characteristic under Assumptions~\ref{assumption}.

\subsection{Top-\texorpdfstring{$\psi$}{psi} for strata with residue conditions}
A central ingredient in our study will be spaces of meromorphic differentials on curves satisfying residue conditions at some of their poles and the intersection numbers of their fundamental classes with powers of $\psi$-classes. These spaces are very natural since they appear in the description of boundary strata of the usual spaces of abelian differentials (see e.g. \cite{BCGGM3}). Moreover, their fundamental classes have many nice properties and explicit descriptions. For instance, as shown in \cite{BRZ} these classes form a partial cohomological field theory when requiring the residues to vanish at all of the poles, and for the case of differentials with precisely two zeroes there exists a connection to the KP hierarchy. 

We start by choosing a vector $a=(a_0,\ldots,a_n)$ such that $|a|=2g-1+n$, $a_0>0$ and $a_i<0$ for $1\leq i\leq n$. For $m \in \{1,\ldots,n-1\}$, we denote by $\M_g^ {\frakR(m)}(a)\subset \M_g(a)$, the subset of the stratum of abelian differentials cut out by the condition that the residues $r_1, \ldots, r_m$ at markings $1, \ldots, n$ vanish. We denote by $\PP\Xi_g^{\frakR(m)}(a)$ its multi-scale differentials compactification as defined in \cite[Prop. 4.2]{CMZ20}. In analogy with the previous sections, we define the intersection number
\[\cA_{g}^{{\frakR(m)}}(a)=\int_{\PP\Xi^{\frakR(m)}(a)}\psi_0^{2g+n-3-m}.\]

We use Proposition \ref{prop:AdrienR} which translates a stratum cut out by a residue condition as a tautological in class in the ambient stratum. We obtain a recursive formula in $m$, describing the previously defined intersection numbers.
\begin{lemma}\label{lem:residuecond}
Let $a=(a_0,a_1,\dots,a_n)$ be a vector such that $a_0>0$ and $a_i<0$ for $i>1$. For any $1\leq m\leq n-1$, we have 
\begin{align*}
	&\cA_{g}^{\frakR(m)}(a)=-a_0	\cA_{g}^{{\frakR(m-1)}}(a)+\!\!\sum_{j=m+1}^n\sum_{I\subseteq\{1,\dots,m-1\}}\!\!\! m_{I,j}\cdot \cA_{0}^{\frakR(|I|)}(a_{I,j})\cdot \cA_{g}^{\frakR(|I^\complement|)}(\hat{a}_{I,j}),
\end{align*}
where $m_{I,j}=-a_{m}-\sum_{i\in I\cup \{j\}}(a_i-1),$ and: 
\begin{align*}
    & a_{I,j}=(m_{I,j},\{a_i\}_{i\in I},a_{m},a_j)\,,\\
    &\hat{a}_{I,j}=(a \setminus \{a_i\}_{i\in I\cup \{m,j\}},-m_{I,j})\,.
\end{align*}
\end{lemma}
We illustrate the lemma in equation \eqref{eqn:residuelemmaillustration} below. Here, a stable graph in brackets should be replaced by the product over the vertices $v$ of the graph of values $\cA_{g(v)}^{\frakR(m(v))}(a(v))$, where $a(v)$ are the orders of zeros and poles at half-edges incident to $v$ and $m(v)$ is the number of such half-edges with imposed residue conditions (which are drawn in red). We also write the order $m_{I,j}$ of the unique zero on the genus zero vertex in blue.
\begin{align} \label{eqn:residuelemmaillustration}
\left[
{\begin{tikzpicture}[scale=0.5, vert/.style={circle,draw,font=\Large,scale=0.7, outer sep=0}, cvert/.style={circle,draw,font=\Large,scale=2, outer sep=0}, thick, baseline=-0.15cm]
\node [vert] (A) at (0,0) {$g$};
\draw[-] (A) -- ++(40:1) node[right]{$m+1$};
\draw (A) ++(0:1) node[right]{$\ldots$};
\draw[-] (A) -- ++(-40:1) node[right]{$n$};
\draw[-, red] (A) -- ++(140:1) node[left]{$m-1$};
\draw[-, red] (A) -- ++(180:1) node[left]{$\ldots$};
\draw[-, red] (A) -- ++(220:1) node[left]{$1$};
\draw[-, red] (A) -- ++(90:1) node[above]{$m$};
\draw[-] (A) -- ++(-90:1) node[below]{$\psi_0^{\text{top}}$};
\end{tikzpicture}}     
\right] =& - a_0 \cdot 
\left[
{\begin{tikzpicture}[scale=0.5, vert/.style={circle,draw,font=\Large,scale=0.7, outer sep=0}, cvert/.style={circle,draw,font=\Large,scale=2, outer sep=0}, thick, baseline=-0.15cm]
\node [vert] (A) at (0,0) {$g$};
\draw[-] (A) -- ++(40:1) node[right]{$m+1$};
\draw (A) ++(0:1) node[right]{$\ldots$};
\draw[-] (A) -- ++(-40:1) node[right]{$n$};
\draw[-, red] (A) -- ++(140:1) node[left]{$m-1$};
\draw[-, red] (A) -- ++(180:1) node[left]{$\ldots$};
\draw[-, red] (A) -- ++(220:1) node[left]{$1$};
\draw[-] (A) -- ++(90:1) node[above]{$m$};
\draw[-] (A) -- ++(-90:1) node[below]{$\psi_0^{\text{top}}$};
\end{tikzpicture} }    
\right]\\
&+\sum_{j=m+1}^n \sum_{I \subseteq \{1, \ldots, -m-1\}} m_{I,j} \cdot  \left[
{\begin{tikzpicture}[scale=0.5, vert/.style={circle,draw,font=\Large,scale=0.7, outer sep=0}, cvert/.style={circle,draw,font=\Large,scale=2, outer sep=0}, thick, baseline=0.65cm]
\node [vert] (A) at (0,0) {$g$};
\node [vert] (B) at (0,3) {$0$};
\draw [-] (A) to node[near end, right, blue]{$m_{I,j}$} node[near end, left]{$\psi^{\text{top}}$} (B);
\draw[-] (A) -- ++(30:1);
\draw[-, red] (A) -- ++(160:1);
\draw[-, red] (A) -- ++(180:1) node[left]{$I^\complement$};
\draw[-] (A) -- ++(-90:1) node[below]{$\psi_0^{\text{top}}$};
\draw[-, red] (B) -- ++(180:1) node[left]{$I$};
\draw[-] (B) -- ++(70:1) node[above]{$m$};
\draw[-] (B) -- ++(30:1) node[above right]{$j$};
\end{tikzpicture}  }   
\right] \nonumber
\end{align}

 \begin{proof}  Let $1\leq m\leq n-1$. We can use Propositions~\ref{prop:Adrienrel} and~\ref{prop:AdrienR} in order to express the locus cut out by one residue condition $r_m=0$ as:
 \[-\eta\cdot [\PP\Xi^{\frakR(m-1)}_{g}(a)]=[\PP\Xi^{\frakR(m)}_{g}(a)] + \delta(m) \in A^1(\PP\Xi^{\frakR(m-1)}_{g}(a)),\]
 such that $\delta(m)$ is a linear combination of the boundary divisors of  $\PP\Xi^{\frakR(m-1)}_{g}(a)$, along which the residue $r_m$ vanishes identically. Besides, we also have:
 \[(-\eta+a_0\psi_0)\cdot [\PP\Xi^{\frakR(m-1)}_{g}(a)]=\delta \in A^1(\PP\Xi^{\frakR(m-1)}_{g}(a)),\]
 where $\delta$ is the linear combination of all the boundary divisors $\PP\Xi^{\frakR(m-1)}_{g}(a)$ as $a_0$ is the only entry of $a$ that is positive. Taking the difference of the two expressions implies: 
 $$[\PP\Xi^{\frakR(m)}_{g}(a)]=-a_0\psi_0\cdot [\PP\Xi^{\frakR(m-1)}_{g}(a)] - \delta(m)+\delta.$$
Thus, the only divsors that are involved in $\delta(m)-\delta$ are the ones defined by level graphs in ${\rm LG}_1(g,a)$ such that the $m$th marking is supported on a vertex of level 0, and such that at least one of the markings in $\{m+1,\ldots, n\}$ is attached to the same vertex. 

Let $\oGamma$ be such a level graph. 
Note that since marking $0$ is the only zero of the differential, the graph has a unique vertex on level $-1$ and this vertex carries the leg associated to this marking.
The intersection of the divisor associated to this graph with $\psi_0^{2g-3+n-m}$  vanishes, unless there is exactly one edge between one vertex of level 0 of genus 0 and  one vertex of level $-1$ of genus $g$, and only one leg in $\{m+1,\ldots,n\}$ is adjacent to the vertex of level 0. The coefficients $m_{i,j}$ are given by the conversion factor \eqref{eq:convfactor}, in particular noting that these graphs have only the trivial automorphism.
 \end{proof}
Note that apart from the theoretical argument given above, we were also able to check the formula from Lemma~\ref{lem:residuecond} in many non-trivial cases using the software package \texttt{diffstrata} \cite{diffstrata}, which can compute formulas for the pushforwards of $\PP\Xi_g^{\frakR(m)}(a)$ to $\oM_{g,n+1}$ in terms of tautological classes.
 
We denote by ${\rm TR}(g,a)$ the set of genus $g$ twisted graphs with legs $a$ 
such that the vertex $v_0$  carrying the marking 0 has genus $g$.
In particular, this forces the graph to be a tree, with all vertices apart from $v_0$ having genus $0$.
Such twisted graphs are uniquely determined by the  underlying stable graph. Moreover there are no automorphisms of such objects as the poles are marked. 

By a straightforward analysis, each vertex $v$ has exactly one half-edge adjacent to it with a positive twist. We denote this number by $I(v)^+$. For such a twisted graph $(\Gamma,I)$, we define
\begin{equation}\label{eq:tildeell}
  F(\Gamma,I)=  \prod_{v\in V(\Gamma)}-\left(-I(v)^+\right)^{n(v)-2},
\end{equation}
where $n(v)$ is the number of half edges attached to $v$. 
\begin{proposition}\label{prop:fullresidue}
Let $a=(a_0,a_1,\dots,a_n)$ with $a_0>0$ and $a_i<0,$ for $i=1,\dots,n$. Then, we have:
    \begin{align*}
\cA_g^{\frakR(n-1)}(a)=\sum_{(\Gamma,I)\in {\rm TR}(g,a) }-\cA_g(I(v_0)) \cdot F(\Gamma,I).
\end{align*}
\end{proposition}
\begin{proof}
The expression is obtained by applying Lemma \ref{lem:residuecond} iteratively $n-1$ times to impose the vanishing of the residues at the $(n-1)$ first poles (and thus at all poles). 
Looking at equation \eqref{eqn:residuelemmaillustration}, we see that when removing the residue condition at marking $h$ incident to a vertex $v$, we obtain a sum of terms, corresponding to either
\begin{itemize}
    \item leaving the graph unchanged, dropping the residue condition at $h$ and obtaining a factor equal to minus the order of the unique zero on $v$, or
    \item sprouting off a genus zero vertex from $v$, carrying some markings $I$ with residue conditions, the marking $h$ and precisely one additional marking without residue condition. In this case, the multiplicity is given by the order of the unique zero on that new vertex. 
\end{itemize}
Continuing to remove one residue condition at a time, we see terms with more complicated underlying stable graphs appearing, all of them trees with a unique genus $g$ vertex, and thus contained in ${\rm TR}(g,a)$. 

The terms appearing after $n-1$ steps are given by intersections numbers on strata of  genus 0 differentials with residue constraints. For each of these strata we apply Lemma \ref{lem:residuecond} until only the functions $\cA_{g}$ and $\cA_0$ are needed. 

After $n-1$ steps, it can be shown that all twisted graphs in ${\rm TR}(g,a)$ appear once in the resulting formula with coefficient given by $-\cA_g(I(v_0)) \cdot F(\Gamma,I)$.
This equality uses that after having removed all residue conditions, the contribution from the genus zero vertices is given by
\[
\int_{\oM_{0,n}} \psi_1^{n-3} = 1 \qedhere
\]
\end{proof}

The following proposition will not be used in the sequel, but it is interesting to note that in the genus 0 case, we can provide a closed formula.
 
 \begin{proposition}\label{prop:residuecondgenus0}
 We assume here that $g=0$. Let $a=(a_0,a_1,\dots,a_n)$ with $a_0>0$ and $a_i<0$ for $i=1,\dots,n$.
 For $n\geq 2$ and $0\leq m\leq n-2$, we have:
 \[\cA_{0}^{\frakR(m)}(a_0,\dots,a_n)=\frac{(n-2)!}{(n-m-2)!}\prod_{i=1}^m (-a_i)\]
 \end{proposition}
 \begin{proof} We will show the stronger statement
\begin{equation}\label{eqn:resgenus0}
    [\PP\Xi^{\frakR(m)}_{0}(a)]=(-1)^m\prod_{i=1}^m a_i\psi_i.
\end{equation}
which implies the statement above by the dilaton equation.
 The proof is similar to the proof of~\ref{lem:residuecond} above. We write:
\begin{eqnarray*}
-\eta\cdot [\PP\Xi^{\frakR(m-1)}_{0}(a)] &=&[\PP\Xi^{\frakR(m)}_{0}(a)]+ \delta(m) \\
(-\eta+a_m\psi_m)\cdot [\PP\Xi^{\frakR(m-1)}_{0}(a)]&=&\delta(m).
\end{eqnarray*}
Indeed, in genus 0, the only boundary components along which the residue $r_m$ vanishes identically are the ones such that the $m$-th marking belong to the level $-1$. These boundary components are also the only ones contributing to $(-\eta+a_m\psi_m)$ (and with the same coefficients). Thus 
$$
[\PP\Xi^{\frakR(m)}_{0}(a)] = -a_m\psi_m \cdot [\PP\Xi^{\frakR(m-1)}_{0}(a)],
$$
and by induction we get the desired identity~\eqref{eqn:resgenus0}
 \end{proof}
 
\subsection{Euler characteristic and minimal strata}

As before we consider $a=(a_0,a_1,\ldots,a_n)$ with $a_0>0$ and $a_i<0$ for $i<0$, and we fix $0\leq m\leq n$.  We will first work on the stratum $\PP\Xi_{g}^{\frakR(m)}(a)$ and denote by $d^m_g(a)$ its dimension. These strata are the ones that will appear as level -1 strata of boundary divisors of minimal holomorphic strata.

Let  $\overline{\Gamma}$ be a level graph of depth $L$, and $-L\leq i\leq 0$. We denote by  $d_{\overline{\Gamma}}^{[i]}$ the dimension of $\PP\Xi(\overline{\Gamma})^{[i]}$ (the space of multi-scale differentials parametrizing the $i$-th level of this graph), and by $\eta_{\overline{\Gamma}}^{[i]}$ the first Chern class of its tautological line bundle. We will also denote by $\LG_L^{\frakR(m)}(a)$ the set of level strata parametrizing boundary components of $\PP\Xi_g^{\frakR(m)}(a)$ (see \cite[Prop. 4.2]{CMZ20}). Note that, for $\oGamma\in\LG_L^{\frakR(m)}(a)$, the residue condition defining the level strata $\PP\Xi(\overline{\Gamma})^{[i]}$ is a combination of the GRC of the graph and the condition $\frakR(m)$. 
\begin{proposition}\label{prop:toppsiformula}
We assume that $a_0>0$, and $a_i<0$ for all $i\geq 1$. Then:
 \[\int_{\PP\Xi_{g}^{\frakR(m)}(a)}(a_0\psi_0)^{d^m_g(a)}= \,\sum_{L=0}^{d^m_g(a)} \,\sum_{\overline{\Gamma} \in \LG_L^{\frakR(m)}(g,a)} \frac{m(\overline{\Gamma})}{|\Aut(\oGamma)|}
 \prod_{i=0}^{-L}
 \int_{\PP\Xi(\overline{\Gamma})^{[i]}}(\eta_{\oGamma}^{[i]})^{d_{\oGamma}^{[i]}}\]
 where $m(\overline{\Gamma})$ is the product of the twists at the edges of $\oGamma$.
\end{proposition}
\begin{proof}
Since there is only one zero, by Proposition~\ref{prop:Adrienrel} and the conversion factor \eqref{eq:convfactor}, we have 
\[\eta=a_0\psi_0-\sum_{\overline{\Gamma} \in \LG_1^{\frakR(m)}(g,a)} \ell(\overline{\Gamma})\cdot [\PP\Xi(\oGamma)] \in A^1(\PP\Xi_{g}^{\frakR(m)}(a)).\]
Consider now the chain of equalities 
\begin{align}
&\nonumber\int_{\PP\Xi_{g}^{\frakR(m)}(a)}\eta^{d_g^m(a)}=\int_{\PP\Xi_{g}^{\frakR(m)}(a)}\eta^{d_g^m(a)-1}\left(a_0\psi_0-\sum_{\oGamma \in \LG_1^{\frakR(m)}(g,a)} \ell(\overline{\Gamma})[\PP\Xi(\oGamma)]\right)\\
\nonumber &=\int_{\PP\Xi_{g}^{\frakR(m)}(a)}a_0\psi_0\cdot \eta^{d_g^m(a)-1}\\
&\nonumber \quad \quad -\sum_{\oGamma \in \LG_1^{\frakR(m)}(g,a)} \frac{m(\oGamma)}{|\Aut(\oGamma)|} \int_{\PP\Xi(\overline{\Gamma})^{[0]}}(\eta_{\oGamma}^{[0]})^{d_g^m(a)-1}\int_{\PP\Xi(\overline{\Gamma})^{[-1]}}(a_0\psi_0)^{0}\\
\label{eqn:etapsiintegrals} &=\int_{\PP\Xi_{g}^{\frakR(m)}(a)}(a_0\psi_0)^{d_g^m(a)}\\
&\nonumber \quad \quad -\sum_{\oGamma \in \LG^{\frakR(m)}_1(g,a)}\frac{m(\oGamma)}{|\Aut(\oGamma)|} \int_{\PP\Xi(\overline{\Gamma})^{[0]}}(\eta_{\oGamma}^{[0]})^{d_\oGamma^{[0]}}\int_{\PP\Xi(\overline{\Gamma})^{[-1]}}(a_0\psi_0)^{d_\oGamma^{[-1]}},
\end{align}
where we used  the conversion factor \eqref{eq:convfactor} to pass from class of boundary divisors to the integrals on their level strata and where the last equality follows by repeating  the replacement procedure of $\eta$ by $\psi_0$ a total number of $d_g^m(a)$ times. The crucial observation is that for a given graph $\oGamma \in \LG^{\frakR(m)}_1(a)$ the product of the two integrals in the sum vanishes for dimension reasons unless we are in the step number $d_\oGamma^{[-1]}$ of the replacement procedure. Note that for this argument it is important that $\psi_0$ is always supported on bottom level. 

Now the bottom levels $\PP\Xi(\overline{\Gamma})^{[-1]}$ of every graph of a stratum with only one positive entry in its signature is again a stratum of the same type (in particular the level $-1$ of $\oGamma$ consists of a single vertex), so we can re-apply the previous relation to the last term $\int_{\PP\Xi(\overline{\Gamma})^{[-1]}}(a_0\psi_0)^{d_\oGamma^{[-1]}}$. We then obtain:
 \begin{align*}
 \int_{\PP\Xi_{g}^{\frakR(m)}(a)}\!\!\!\!\!\!(a_0\psi_0)^{d_g^m(a)}
 &=\int_{\PP\Xi_{g}^{\frakR(m)}(a)}\eta^{d_g^m(a)}  \\ & \!\!\!\!\!\!\!\!\! \!\!\!\!\!\!\!\!\! \!\!\!\!\!\!\!\!\! \!\!\!\!\!\!\!\!\!  + \!\!\!\!\!\!\!\!\!\!\sum_{\oGamma \in \LG^{\frakR(m)}_1(g,a)}\frac{m(\oGamma)}{|\Aut(\oGamma)|} \int_{\PP\Xi(\overline{\Gamma})^{[0]}}(\eta_{\oGamma}^{[0]})^{d_\oGamma^{[0]}}\int_{\PP\Xi(\overline{\Gamma})^{[-1]}}(\eta_{\oGamma}^{[-1]})^{d_\oGamma^{[-1]}}\\
 &\!\!\!\!\!\!\!\!\! \!\!\!\!\!\!\!\!\! \!\!\!\!\!\!\!\!\! \!\!\!\!\!\!\!\!\!  +\!\!\!\!\! \!\!\!\!\! \sum_{\oGamma \in \LG^{\frakR(m)}_2(g,a)}\frac{m(\oGamma)}{|\Aut(\oGamma)|} \int_{\PP\Xi(\overline{\Gamma})^{[0]}}\!\!\! (\eta_{\oGamma}^{[0]})^{d_\oGamma^{[0]}}\int_{\PP\Xi(\overline{\Gamma})^{[-1]}}\!\!\! (\eta_{\oGamma}^{[-1]})^{d_\oGamma^{[-1]}}\int_{\PP\Xi(\overline{\Gamma})^{[-2]}}\!\!\! (a_0\psi_0)^{d_\oGamma^{[-2]}}\!\!\!.
  \end{align*}
Repeating this procedure ${d_g^m(a)}$ times, we get the claim of the proposition.
For this, a crucial observation is that for any $L \geq 0$, the level graphs obtained by gluing an arbitrary graph with $1$ level passage into the bottom vertex of a graph in $\LG^{\frakR(m)}_L(a)$ precisely form the set $\LG^{\frakR(m)}_{L+1}(a)$ and each graph in $\LG^{\frakR(m)}_{L+1}(a)$ is obtained like this in a unique way.
\end{proof}
We apply  the previous proposition in order to simplify the expression for the Euler characteristic shown in \cite{CMZ20} specialized to the case of  strata of abelian differentials with only one positive entry in their signature. The original formula expresses the Euler characteristic as a sum over all possible level graphs of some intersection numbers, while here we simplify the problem by expressing the Euler characteristic as a sum indexed by level graph with $L=1$.
\begin{corollary}\label{cor:eclg1}
	 If $(a_0,a_1,\dots,a_n)$ is a vector such that  $a_0>0$ and $a_i\leq 0$ for all $i=1,\dots,n$, and $1\leq m\leq n$, then we have:
	 \begin{align*}
	 	(-1)^{d_g^m(a)}\chi(\M_g^{\frakR(m)}(a))&=(d_g^m(a)+1)\int_{\PP\Xi_g^{\frakR(m)}(a)}(a_0\psi_0)^{d_g^m(a)}\\
	 	&\!\!\!\!\!\!\!\!\! \!\!\!\!\!\!\!\!\! \!\!\!\!\!\!\!\!\! -\!\!\!\!  \sum_{\oGamma\in \LG_1^{\frakR(m)}(g,a)}\!\!\!\!  (d_\oGamma^{[-1]}+1)\frac{m(\oGamma)}{|\Aut(\oGamma)|}\int_{\PP\Xi(\overline{\Gamma})^{[0]}}(\eta_{\oGamma^{[0]}})^{d_\Gamma^{[0]}}\int_{\PP\Xi(\overline{\Gamma})^{[-1]}}(a_0\psi_{0})^{d_\oGamma^{[-1]}}
	 \end{align*}
\end{corollary}
\begin{proof}
To start, let us recall the formula for the Euler characteristic of $\M_g^{\frakR(m)}(a)$ proven in \cite[Theorem 1.3]{CMZ20}:
\begin{equation}\label{eqn:CMZeulercharformula}(-1)^{d_g^m(a)}\chi(\M_g^{\frakR(m)}(a))=\sum_{L=0}^{d_B} \,\sum_{{\oGamma'} \in \LG_L^{\frakR(m)}(g,a)} (d_{\oGamma'}^{[0]}+1)\frac{m({\oGamma'})}{|\Aut({\oGamma'})|}
 \prod_{i=0}^{-L}
 \int_{\PP\Xi({\oGamma'})^{[i]}}
(\eta_{{\oGamma'}}^{[i]})^{d_{{\oGamma'}}^{[i]}}\end{equation}
Given $\oGamma'\in \LG_L^{\frakR(m)}(g,a)$ for $L \geq 1$ as in the formula above, we write $\delta_1(\oGamma')=\oGamma$ for the first undegeneration of $\oGamma'$, that is the $2$-level graph $\oGamma$ obtained from $\oGamma'$ by contracting all level passages of $\oGamma'$ apart from the first one.
To prove the formula from the corollary, we group all summands of \eqref{eqn:CMZeulercharformula} associated to non-trivial graphs $\oGamma'$ according to their undegeneration $\oGamma=\delta_1(\oGamma')$.

To simplify the formula from there, observe that given any $\oGamma\in \LG_1^{\frakR(m)}(g,a)$, we obtain the equality
\begin{align*}
&\frac{m(\oGamma)}{|\Aut(\oGamma)|}\cdot \int_{\PP\Xi(\overline{\Gamma})^{[0]}}(\eta_{\oGamma^{[0]}})^{d_\Gamma^{[0]}}\cdot \int_{\PP\Xi(\overline{\Gamma})^{[-1]}}(a_0\psi_{0})^{d_\oGamma^{[-1]}}\\
&\quad =\sum_{L=1}^{d_g^m(a)} \sum_{\oGamma'\in \LG_L^{\frakR(m)}(g,a) \atop \delta_1(\oGamma')=\oGamma}  \frac{m(\oGamma')}{|\Aut(\oGamma')|}
 \prod_{i=0}^{-L}
 \int_{\PP\Xi(\oGamma')^{[i]}}
(\eta_{\oGamma'}^{[i]})^{d_{\oGamma'}^{[i]}}.
\end{align*}
by applying Proposition \ref{prop:toppsiformula} to the bottom level stratum $\PP\Xi(\overline{\Gamma})^{[-1]}$. 
Multiplying this equality by $(d_\oGamma^{[0]}+1)$ and summing over all choices of $\oGamma$, we obtain all summands of the right-hand side of \eqref{eqn:CMZeulercharformula} associated to non-trivial graphs $\oGamma'$. In particular, writing the term of the trivial graph separately, we obtain 
\begin{align*}
	 &	(-1)^{d_g^m(a)}\chi(\M_g^{\frakR(m)}(a))=(d_g^m(a)+1)\int_{\PP\Xi_g^{\frakR(m)}(a)}\eta^{d_g^m(a)}+\\
	 	&\quad \quad +\sum_{\oGamma \in \LG_1(g,a)} (d_\oGamma^{[0]}+1)\frac{m(\oGamma)}{|\Aut(\oGamma)|}
 \int_{\PP\Xi(\oGamma)^{[0]}}
(\eta_{\oGamma}^{[0]})^{d_{\oGamma}^{[0]}}\int_{\PP\Xi(\oGamma)^{[-1]}}
(a_0\psi_0)^{d_{\oGamma}^{[-1]}}
\end{align*}
In order to show the main statement, it is enough to rewrite $\int_{\PP\Xi_g^{\frakR(m)}(a)}\eta^{d_g^m(a)}$ using the relation \eqref{eqn:etapsiintegrals} and using that $d_g^m(a)-d_\oGamma^{[0]}=d_\oGamma^{[-1]}+1$.
\end{proof}

\begin{remark}
Following a similar approach, one can show that the  Euler characteristic of a general stratum  (not necessarily with only one zero), may be written as a sum only on two level graphs (even more, only the ones with only one vertex on top level). However this sum would contain integrals of powers of $\eta$ on meromorphic strata which are a priori hard to compute.
\end{remark}
Now, we apply Corollary \ref{cor:eclg1} to the case of the minimal stratum $\cM_{g}(2g-1)$. In this case, the top level of a 2-level graph is basically given by a product of holomorphic strata. Then, as we explain below, we can use the fact that top-$\eta$ powers of non-minimal holomorphic strata vanish thanks to \cite{Sau:volumeminimal}, to simplify further the expression of the Euler characteristic.
\begin{proposition}\label{prop:ECmin}
	For $g\geq 1$, we denote $d_g=2g-1$. Then for all $g\geq 1$, we have:
	\begin{align*}
\chi(\cM_{g}(d_g))&=2g\cdot d_g^{d_g-1}\cdot \cA_g(d_g)\\
&\!\!\!\!\!\!\!\!+\sum_{g^\bot=1}^{g-1} 2g^\bot \sum_{\underline{g}=(g_1, \dots , g_n)\atop |\underline{g}|=g-g^\bot}	\frac{ (-1)^{n+1} }{n!}d_g^{d_{g^\bot}}\cA_{g^\bot}^{\frakR(n-1)}(d_g,-d_{\underline{g}})\prod_{i=1}^n d_{g_i} \int_{\PP\Xi_{g_i}(d_{g_i})}\eta^{d_{g_i}}\\
	\end{align*}
	where the second sum is on all vectors \rcomment{with positive entries} of size size $g-g^{\bot}$, and  $d_{\underline{g}}=(d_{g_1},\dots, d_{g_n})$. 
\end{proposition}
\begin{proof}
We want to describe the 2-level graphs which give a non-trivial contribution in Corollary \ref{cor:eclg1}. In order to do that, we want to investigate top $\eta$-powers on top level strata of 2-level graphs. 

Since these level strata are in general disconnected, we first reduce the computation of powers of $\eta$ on disconnected strata in terms of powers of $\eta$ on connected ones. For a graph $\oGamma\in \LG_1(g,(d_g))$,  the total space of the top level stratum $\PP\Xi(\oGamma)^{[0]}$ is isomorphic to the total space of the projective bundle $\pi: \PP(\cE_\oGamma) \to Y_\oGamma$ where 
$$Y_\oGamma = \prod_{v \in V(\oGamma):\ l(v)=0} \PP\Xi_{g(v)}(I(v))$$
and $$\cE_\oGamma = \bigoplus_{v \in V(\oGamma):\ l(v)=0}
p_v^* \cO_{\PP\Xi_{g(v)}(I(v))}(-1),$$ with $p_v:Y_\oGamma\to \PP\Xi_{g(v)}(I(v))$ being the projection map. Then, we have:
\begin{align}\label{eq:topetadisconnected}
\int_{\PP\Xi(\oGamma)^{[0]}} \eta^{d_\oGamma^{[0]}} &=(-1)^{d_\oGamma^{[0]}} \int_{Y_\oGamma} s_{d_\oGamma^{[0]}-v_\oGamma^{[0]}+1}(\cE_{\oGamma})\\
&\nonumber =(-1)^{v_\oGamma^{[0]}+1} \!\!\!\!\! \prod_{v\in V(\oGamma):\ l(v)=0} \int_{\PP\Xi_{g(v)}(I(v))}\eta^{d_{g(v)}+|I(v)|-1},
\end{align}
where $v_\oGamma^{[0]}$ is the number of vertices of level $0$. 

By \cite[Prop. 3.3]{Sau:volumeminimal}, the top $\eta$-powers on non-minimal holomorphic strata vanish. Hence, thanks to \eqref{eq:topetadisconnected}, the only non trivial contribution from a 2-level graph on the right hand side of the statement of Corollary \ref{cor:eclg1} is given by star-shaped graphs, i.e. 2-level graphs with only one vertex on bottom level (this is the only possiblity in a stratum with only one positive entry in its signature) and with some vertices on top level, each top vertex connected to the bottom vertex via only one edge. 

We can parametrize such 2-level graphs by choosing a partition $(g_1,\dots,g_n)$ of the possible total genus $g^\top$ on top level.\footnote{We remark that a $2$-level graph only determines an unordered partition $g^\top = g_1 + \ldots + g_n$. By a standard combinatorial argument using the Orbit-Stabilizer theorem, summing over such partitions and weighting by $1/|\Aut(\oGamma)|$ as in Corollary \ref{cor:eclg1} is equivalent to summing over vectors $(g_1, \ldots, g_n)$ and dividing by $n!$ as in Proposition \ref{prop:ECmin}.}  Note that in the formula the number $g^\top$ was rewritten as the difference $g-g^\bot$. The bottom level stratum of a 2-level graph defined by a partition $(g_1,\dots,g_n)$ is given by $\PP\Xi^{\frakR(n-1)}_{g^\bot,n+1}(2g-1,-2g_1+1,\dots,-2g_n+1)$, since the global residue condition imposes zero residue as every pole. We conclude by using \eqref{eq:topetadisconnected} and rewriting the top $\psi_0$ contributions using the symbols previously introduced. Note that in the contribution from the trivial graph, the term  $\int_{\oM_g(d_g)}\psi_{1}^{d_g}$ has to be replaced by $-\cA_g(d_g)/d_g$, since $\M_g(2g-1)$ is a holomorphic stratum (see the relation~\eqref{eq:drstrata}).
\end{proof}

We may now state and prove the following theorem, which is equivalent to Theorem~\ref{thm:eulerseries} stated in the introduction after using coefficient extraction formulas.

\begin{theorem}\label{thm:ECminformula}
For all $g\geq 1,$ we have
    \begin{align*}
&\chi(\M_g(d_g))=\sum_{1\leq g^\bot\leq g \atop n\geq 0}  2 g^\bot \sum_{\underline{g}=(g_1,\dots,g_n)\atop \sum g_i=g-g^\bot}	\frac{(-1)^n }{n!}d_g^{d_{g^\bot}+n-1}\cA_{g^\bot}(d_g,-d_{\underline{g}})\prod_{i=1}^n (d_{g_i})!\cdot b_{g_i}\\
\end{align*}
where \rcomment{the vector $\underline{g}$ has positive entries,} $d_{\underline{g}}=(d_{g_1},\dots, d_{g_n})$ and where the numbers $b_{g_i}$ are determined by 
$$\frac{z/2}{{\rm sinh}(z/2)}=1+\sum_{g\geq 0} b_g z^{2g}\,.$$
\end{theorem}

Before giving the proof of this theorem we introduce a set of twisted graphs that will be used only in this proof and in the proof of Lemma~\ref{lem:regrouping} below. We fix $g\geq 1$ and  $0\leq g^\bot\leq g$. The set $\widetilde{\rm TR}(g,g^\bot)$ is the set of graphs of compact type (trees) of genus $g$ with 1 leg satisfying the following properties:
\begin{itemize}
    \item the root $v^\bot$ of the tree has genus $g^\bot$ and carries the unique leg;
    \item the leaves of the tree have positive genus (we denote by ${\rm Leaf}(\Gamma)$ the set of leaves);
    \item internal nodes (vertices which are neither the root nor leaves) have genus 0 (this set will be denoted by ${\rm Int}(\Gamma)$).
\end{itemize}
Each such graph carries a unique twist such that the twist at the unique leg is $2g-1$. Moreover, we define $\widetilde{\rm TR}(g)$ to be the set obtained from $\widetilde{\rm TR}(g,0)$ by adding the trivial graph in genus $g$ with one leg. 

\begin{proof} We start from Proposition \ref{prop:ECmin}. First we will use Proposition \ref{prop:fullresidue} to replace integrals on strata with residue conditions by integrals on classical strata:
 \begin{align} \label{eqn:AgRneqn}
\cA_{g^\bot}^{\frakR(n-1)}(d_g,-d_{\underline{g}})=-\sum_{(\Gamma,I)\in {\rm TR}(g^\bot,(d_g,-d_{\underline{g}})) }\cA_{g^\bot}(I(v_0)) \cdot F(\Gamma,I).
\end{align}
Note that if we define $\cA_g^{\frakR(-1)}(d_g) = \cA_g(d_g)/d_g$, then it is immediate from the definition of $F(\Gamma, I)$ that equation \eqref{eqn:AgRneqn} is still satisfied in this case. Moreover
the first term in Proposition \ref{prop:ECmin} can be written as
\[2g \cdot d_g^{d_g-1} \cdot \cA_g(d_g) = 2g \cdot d_g^{d_g} \cdot \cA_g^{\frakR(-1)}(d_g)\]
and can then be integrated as the $g^\bot=g, n=0$ term of the remaining sum.

By inserting equality \eqref{eqn:AgRneqn} into the formula of Proposition \ref{prop:ECmin}, we can replace the expression of $\chi(\M_g(d_g))$ as a sum over star-shaped graphs where one central vertex contributes with a function $\cA$ with residue condition, by a sum  indexed by the sets $\widetilde{\rm TR}(g,g^\bot)$ where the central vertex contributes with the function $\cA$ (without residues): 
 \begin{align}
\label{eq:chiI1}&\chi(\M_g(d_g))=\sum_{1\leq g^\bot\leq g}  2 g^\bot \sum_{\Gamma \in \widetilde{\rm TR}(g,g^\bot)}	\frac{(-1)^{n(v^\bot)+1} }{|{\rm Aut}(\Gamma)|} d_g^{d_{g^\bot}+n(v^\bot)-2}\cA_{g^\bot}(I(v^\bot)) \\ 
\nonumber&\ \hspace{80pt} \times\left(\prod_{v \in {\rm Int}(\Gamma)}  -(-I(v)^+)^{n(v)-2}\right) \times \left(\prod_{v \in {\rm Leaf}(\Gamma)} d_{g(v)}a_{g(v)}\right),
\end{align}
where, as above we used the notation $I(v)^+$ for the unique positive twist at an internal vertex, and we used the notation: $$a_g=-\int_{\PP\Xi_g(2g-1)}\eta^{2g-1},$$
for all $g\geq 1$.

We will regroup this sum as follows. In order to construct a tree in $\widetilde{\rm TR}(g,g^\bot)$, one may equivalently chose: (i) the number $n = n(v^\bot)-1 \geq 0$ of descendants of the root; (ii) the genera $g_1,g_2,\ldots,g_n$ of the descendants (i.e. the $n$ trees obtained by removing the root); (iii) a tree in $\widetilde{\rm TR}(g_i)$ for all $1\leq i\leq n$. An autormorphism of this new data is given by a permutation $\sigma$ of $\{1,\ldots,n\}$ and an isomorphism of $\Gamma_i$ and $\Gamma_{\sigma(i)}$ for all $1\leq i\leq n$.

Then formula~\eqref{eq:chiI1} becomes:
 \begin{align*}
&\chi(\M_g(d_g))=\sum_{1\leq g^\bot\leq g \atop n\geq 0}  2 g^\bot \sum_{\underline{g}=(g_1,\dots,g_n)\atop \sum g_i=g-g^\bot}	\frac{(-1)^n }{n!}d_g^{d_{g^\bot}+n-1}\cA_{g^\bot}(d_g,-d_{\underline{g}})
\\ &\ \hspace{2pt} \times \prod_{i=1}^n \sum_{\Gamma_i \in \widetilde{\rm TR}(g_i)} \!\!\! \frac{1}{|{\rm Aut}(\Gamma_i)|}
\left(\prod_{v \notin {\rm Leaf}(\Gamma_i)}  \!\!\!\!\!\!\!\!\!-(-I(v)^+)^{n(v)-2}\right) \times\left(\prod_{v \in {\rm Leaf}(\Gamma_i)} \!\!\!\!\!\!\!\!\! d_{g(v)} a_{g(v)}\right).
\end{align*}
 Thus, the theorem follows from Lemma~\ref{lem:regrouping} below. 
\end{proof}

\begin{lemma}\label{lem:regrouping}
For all $g\geq 1$, we have: 
$$(2g-1)!b_g=\sum_{\Gamma \in \widetilde{\rm TR}(g)} \frac{1}{|{\rm Aut}(\Gamma)|}
\left(\prod_{v \notin {\rm Leaf}(\Gamma)}  \!\!\!\!\!\!-(-I(v)^+)^{n(v)-2}\right) \times\left(\prod_{v \in {\rm Leaf}(\Gamma)}\!\!\!\!\!\! d_{g(v)} a_{g(v)}\right).
$$
\end{lemma}

\begin{proof}We first recall from~\cite[Theorem 1.6]{Sau:volumeminimal} that the series $b_g$ and $a_g$ satisfy the following identity for all $g\geq 1$
\begin{equation}[z^{2g}] \label{eq:defag} \mathcal{F}(z)^{2g}= (2g)! b_g,
\end{equation}
where $\mathcal{F}(z)=1+\sum_{g\geq 1} (2g-1) a_g z^{2g}$.\footnote{Note that the paper \cite{Sau:volumeminimal} works with the class $\xi=-\eta$ in the definition of its numbers $a_g$ and the function $(t/2)/\sin(t/2)$ for its analogue of the numbers $b_g$, both resulting in a factor $(-1)^g$ compared to our definition.}

We will prove the lemma by induction on $g\geq 1$. First, we note that the initial step is trivial as $b_1=a_1$ and there is only one (trivial) graph in $\widetilde{\rm TR}(1)$. 

Let $g\geq 2$. As in the previous proof, we may decompose a sum over non-trivial graphs in $\widetilde{\rm TR}(g)$ as a sum over: (i) choice of $n\geq 2$; (ii) a vector $(g_1,\ldots,g_n)$ of size $g$; (iii) choices of $n$ trees in $\widetilde{\rm TR}(g_i)$. Thus we get: 
\begin{eqnarray*}
&&\sum_{\Gamma \in \widetilde{\rm TR}(g)} \frac{1}{|{\rm Aut}(\Gamma)|}
\left(\prod_{v \notin {\rm Leaf}(\Gamma)}  -(-I(v)^+)^{n(v)-2}\right) \times\left(\prod_{v \in {\rm Leaf}(\Gamma)} d_{g(v)}a_{g(v)}\right)  
\\ &=& d_g a_g- \sum_{n\geq 2\atop g_1+\ldots+g_n=g} \frac{(-2g+1)^{n-1}}{n!} \prod_{i=1}^n   \sum_{\Gamma_i \in \widetilde{\rm TR}(g_i)} \frac{1}{|{\rm Aut}(\Gamma_i)|}
\\
&& \hspace{30pt} \left(\prod_{v \notin {\rm Leaf}(\Gamma_i)}  -(-I(v)^+)^{n(v)-2}\right) \times\left(\prod_{v \in {\rm Leaf}(\Gamma_i)} d_{g(v)} a_{g(v)}\right)\\
&=& d_ga_g - \sum_{n\geq 2\atop g_1+\ldots+g_n=g} \frac{(-2g+1)^{n-1}}{n!} \prod_{i=1}^n (2g_i-1)! b_{g_i}.
\end{eqnarray*}
The last line was obtained by applying the induction hypothesis. Thus we need to show that 
\begin{equation} \label{eqn:lemregrouping}
(2g-1) a_g=\sum_{n\geq 1\atop g_1+\ldots+g_n=g}\frac{(-2g+1)^{n-1}}{n!} \prod_{i=1}^n (2g_i-1)! b_{g_i}
\end{equation}
for all $g\geq 2$. This last identity is equivalent to formula~\eqref{eq:defag} by~\cite[Corollary 2.(ii)]{Combinatorics}, which we apply to $x_n=(2n-1)a_n, y_n=(2n-1)!b_n$, and $(a,b,c,d)=(-2,0,0,1)$  in the notation of this paper, thus finishing the proof.
\end{proof}

Finally, we are able to prove the version of the formula for the Euler characteristic that is claimed in the introduction.
\begin{proof}[Proof of Theorem \ref{thm:eulerseries}] We use the expression of $\chi(\M_g(2g-1))$ given by Theorem~\ref{thm:ECminformula} to obtain the coefficient extraction formula of Theorem \ref{thm:eulerseries}. We begin with a slight rearrangement of terms:
\begin{eqnarray*}
&&\sum_{1\leq g^\bot\leq g \atop n\geq 0}  2 g^\bot \sum_{\underline{g}=(g_1,\dots,g_n)\atop \sum g_i=g-g^\bot}	\frac{(-1)^n }{n!}d_g^{d_{g^\bot}+n-1}\cA_{g^\bot}(d_g,-d_{\underline{g}})\prod_{i=1}^n (d_{g_i})!\cdot b_{g_i}\\
&=& d_g^{2g-2} \sum_{1\leq g^\bot\leq g \atop n\geq 0}  \frac{2 g^\bot}{n!}   \sum_{\underline{g}=(g_1,\dots,g_n)\atop \sum g_i=g-g^\bot} \cA_{g^\bot}(d_g,-d_{\underline{g}})\prod_{i=1}^n -d_g^{1-2g_i} (d_{g_i})!\cdot b_{g_i}.
\end{eqnarray*}
Now, we use Theorem~\ref{th:main} to write the function $\cA_{g^\bot}$ with a coefficient extraction formula. Then the $\chi(\M_g(2g-1))$ takes the following shape: 
\begin{eqnarray*}
&& d_g^{2g-2} \sum_{1\leq g^\bot\leq g \atop n\geq 0}  \frac{2 g^\bot}{n!}  [z^{2g^\bot}] \frac{{\rm exp}(d_gz\cS'(z)/\cS(z))}{\cS(z)^{2g+1}} \\
&& \hspace{20pt} \sum_{\underline{g}=(g_1,\dots,g_n)\atop \sum g_i=g-g^\bot}  \prod_{i=1}^n -\cS(d_{g_i}z)(d_g/\cS(z))^{1-2g_i} (d_{g_i})!\cdot b_{g_i} \\
&=& d_g^{2g-2} [z^{2g}] \frac{{\rm exp}(d_gz\cS'(z)/\cS(z))}{\cS(z)^{2g+1}} \\
&& \hspace{13pt} \sum_{1\leq g^\bot\leq g }  {2 g^\bot}  [y^{-2(g-g^\bot)}] {\rm exp}\left( \sum_{\widetilde{g}>0} -\left(\frac{z}{y}\right)^{2\widetilde{g}}\cS(d_{\widetilde{g}}z)(d_{{g}}/\cS(z))^{1-2\widetilde{g}} (d_{\widetilde{g}})!\cdot b_{\widetilde{g}} \right)\\
&=& d_g^{2g-2} [z^{2g}] \frac{{\rm exp}(d_gz\cS'(z)/\cS(z))}{\cS(z)^{2g+1}} \\
&& \hspace{13pt} \sum_{1\leq g^\bot\leq g }  {2 g^\bot}  [y^{-2(g-g^\bot)}] {\rm exp}\left(- d_g \mathcal{H}(d_g y, z)\right)\\
&=& d_g^{2g-2} [z^{2g}] \frac{{\rm exp}(d_gz\cS'(z)/\cS(z))}{\cS(z)^{2g+1}} \left(\left(y+1+y\frac{\partial}{\partial y}\right){\rm exp}\left(-d_g\mathcal{H}\right)\right)\big|_{y=2g-1}\\
&=& d_g^{2g-2} [z^{2g}]  \left(\frac{y+1-y^2\frac{\partial \cH}{\partial y}}{\cS(z)^{2}} {\rm exp}\left(y\left(\frac{z\cS'(z)}{\cS(z)}-{\rm ln}(\cS(z))-\mathcal{H}\right)\right)\right)\big|_{y=2g-1}.
\end{eqnarray*}
Note that in the third equality we use the fact that for a Laurent series $G(y,z)$ such that all Laurent monomials $y^e z^f$ satisfy $f \geq -e \rcomment{\geq} 0$ \rcomment{and $e$ even}, which we later specialize to $G(y,z) = \exp(-d_g \mathcal{H}(y,z))$, we have that
\begin{equation} \label{eqn:crazyequationwithLaurentseries}
[z^{2g}] \sum_{1 \leq g^\bot \leq g} 2g^\bot [y^{-2(g-g^\bot)}] G(d_g y, z) = [z^{2g}] (y+1+y  \frac{\partial}{\partial y}) G(y,z) \big|_{y=2g-1}\,.
\end{equation}
The equality \eqref{eqn:crazyequationwithLaurentseries} can be checked on Laurent monomials.
\end{proof}

\subsection{Euler characteristic of connected components of minimal strata}

We want to refine the previous Euler characteristic computation and compute the Euler characteristic of the connected components of minimal strata. In order to do so, we show first a spin version of Theorem \ref{thm:ECminformula}. Recall that we set  
\[\chi(\M_g(2g-1))^\sp:= \chi(\M_g(2g-1)^\even)-\chi(\M_g(2g-1)^\odd)\]
to be difference of the orbifold Euler characteristics of the even and odd components of the minimal stratum $\M_g(2g-1)$.

\begin{theorem}\label{thm:ECminformulaspin}
Assume part (4) of Assumption \ref{assumption} holds, then for all $g\geq 1,$ we have
    \begin{align*}
&\chi(\M_g(d_g))^\sp= \!\!\sum_{1\leq g^\bot\leq g \atop n\geq 0} \!2 g^\bot \!\!\!\! \sum_{\underline{g}=(g_1,\dots,g_n)\atop \sum g_i=g-g^\bot}	\frac{(-1)^n }{n!}d_g^{d_{g^\bot}+n-1}\cA^\sp_{g^\bot}(d_g,-d_{\underline{g}})\prod_{i=1}^n (d_{g_i})!\cdot \widetilde{b}_{g_i}\\
	\end{align*}
	where $d_g = 2g-1$, $\widetilde{b}_g=\frac{-2^{g-1}}{2^{2g-1}-1}b_g$ and $d_{\underline{g}}=(d_{g_1},\dots, d_{g_n})$.
\end{theorem}
\begin{proof}
First of all recall that by a generalized version of the Gauss-Bonnet theorem (see for example \cite[Sec. 2]{CMZ20}), we have
\[(-1)^{d_g(a)}\chi(\M_g^\sp(a))=c_{d_g(a)}(\Omega^1_{\PP\Xi_g(a)}(\log(D)))\cdot [\PP\Xi^\sp_g(a)]\]
where $D=\PP\Xi_g(a)\setminus \M_g(a)$. 
The top Chern class of the logarithmic cotangent bundle $\Omega^1_{\PP\Xi_g(a)}(\log(D))$ was computed (as a cohomology class) in \cite[Theorem 9.10]{CMZ20}. It can be written as a sum over all level graphs $\oGamma'$ for $\PP\Xi_g(a)$ as in the right-hand side of \eqref{eqn:CMZeulercharformula}, with integrals replaced by gluing pushforwards.
As in the proof of Corollary \ref{cor:eclg1} we can then regroup the summands according to their first undegeneration $\oGamma$ and use Proposition \ref{prop:toppsiformula} to obtain 
\begin{align} \label{eqn:topChernlogcotangent}
	 	c_{d_g(a)}(\Omega^1_{\PP\Xi_g(a)}(\log(D)))&=(d_g(a)+1)\cdot (a_0\psi_0)^{d_g(a)}\cdot [\PP\Xi_g(a)]\\
	 	&\!\!\!\!\!\!\!\!\! \!\!\!\!\!\!\!\!\! \!\!\!\!\!\!\!\!\! -\!\!\!\!  \sum_{\oGamma\in \LG_1(g,a)}\!\!\!\!  (d_\oGamma^{[-1]}+1)\cdot \ell_\oGamma \cdot \eta^{d_\Gamma^{[0]}}\cdot (a_0\psi_{0})^{d_\oGamma^{[-1]}} \cdot [\PP\Xi(\overline{\Gamma})] \nonumber
 \end{align}
Here we note that while Proposition \ref{prop:toppsiformula} was formulated on the level of intersection numbers, the proof only uses relations in the Chow groups and can thus be lifted to an equality of cohomology classes.


Now we want to argue that, as in the proof of Proposition \ref{prop:ECmin}, all non-star graphs in the sum above have a trivial contribution. To see this, first note that when pairing the class \eqref{eqn:topChernlogcotangent} with $[\oM_g(a)]^\sp$, the fundamental class of each connected component of the loci $\PP\Xi(\overline{\Gamma})$ is simply multiplied by $\pm 1$, depending on the parity of the associated differential. Then, we may carry out the argument in Proposition \ref{prop:ECmin}, which converts the top power of $\eta$ on $\PP\Xi_{g(v)}(I(v))^{[0]}$ into a product over the level zero strata, on each component of $\PP\Xi_{g(v)}(I(v))^{[0]}$ separately. Finally, we observe that the proofs explained in Section 3.1 of \cite{Sau:volumeminimal} show that the top power of $\eta$ vanishes \emph{on each connected component} of every non-minimal stratum of holomorphic differentials. Thus, even in the spin setting, only terms coming from star graphs $\oGamma$ can give non-trivial contributions in \eqref{eqn:topChernlogcotangent}.

Since these terms are all of compact type, we know that the parity under the associated gluing map is simply the sum of parities at all vertices by Proposition \ref{pro:parityct_improved}. 
This implies that we can now repeat the proof of Proposition \ref{prop:ECmin}, replacing all fundamental classes by their spin counterparts, and arrive at the following formula:

 \begin{align*}
&\chi(\cM^\sp_{g}(d_g))=2g\cdot d_g^{d_g-1}\cdot \cA^\sp_g(d_g)\\
&+\sum_{g^\bot=1}^{g-1} 2g^\bot \sum_{\underline{g}=(g_1, \dots , g_n)\atop |\underline{g}|=g-g^\bot}	\frac{ (-1)^{n+1} }{n!}d_g^{d_{g^\bot}}\cA_{g^\bot}^{\sp,\frakR(n-1)}(d_g,-d_{\underline{g}})\prod_{i=1}^n d_{g_i} \int_{\PP\Xi^\sp_{g_i}(d_{g_i})}\eta^{d_{g_i}}.\\
\end{align*}
Note that in order to obtain the first term above, we use part (4) of Assumption \ref{assumption} to replace $\int_{[\oM_g(d_g)]^\sp} \psi_1^{d_g}$ by $-\cA_g^\sp(d_g)/d_g$.
%

Before continuing, we remark that Lemma \ref{lem:residuecond} and its consequence Proposition \ref{prop:fullresidue} are true in the spin case after substituting the terms with their spin counterparts. Indeed the only graphs appearing are of compact type, and so again the parity of the graph is given by the sum of the parities of its vertices. Moreover, when generalizing the proof of Proposition \ref{prop:fullresidue} we use in the end that in genus zero, the spin cycle agrees with the fundamental class $[\oM_{0,n}]$.

Using the analogue of Proposition \ref{prop:fullresidue}, we can now run the same argument as in the proof of Theorem \ref{thm:ECminformula} to reduce to the spin-version of Lemma \ref{lem:regrouping}, where $b_g$ is replaced by $\widetilde{b}_g$ and $a_g$ is replaced by 
$$
a_g^\sp=-\int_{[\PP\Xi(2g-1)]^\sp} \eta^{2g-1}\,.
$$
In the proof of this lemma, the role of equality
\eqref{eqn:lemregrouping}
is played by the corresponding equality
\begin{equation*} 
(2g-1) a_g^\sp=\sum_{n\geq 1\atop g_1+\ldots+g_n=g}\frac{(-2g+1)^{n-1}}{n!} \prod_{i=1}^n (2g_i-1)! \widetilde{b}_{g_i}
\end{equation*}
which is proved in ~\cite[Corollary~6.11]{chmosaza} by means of representation theory. Indeed, from~\cite{CMZ20} and~\cite{Sau:volumeminimal}, the intersection number $a_g^\sp$ is (up to a simple combinatorial coefficient) the difference between the Masur-Veech volumes of the odd and even component of $\cM_g(2g-1)$ computed in \cite{chmosaza}.

The final observation that we need is that when generalizing the proof of Lemma \ref{lem:regrouping}, which was an induction in $g \geq 1$, the initial equation to check here is $\widetilde{b}_1 = a_1^\sp$. On the one hand, from the formula we have $\widetilde{b}_1 = - b_1$. On the other hand, the equality $[\oM_1(1)]^\sp = -[\oM_{1,1}]$ following from \eqref{eqn:M1aspin} implies $a_1^\sp = -a_1$ and so we can conclude from the known equality $b_1 = a_1$.
\end{proof}
Using this theorem, the proof of Theorem~\ref{thm:eulerseriesspin} is analogous to the proof of Theorem~\ref{thm:eulerseries} given in the previous section. Here we use Theorem \ref{th:spin} to explicitly compute the function $\cA_{g^\bot}^\sp$ and thus we need parts (1) to (3) of Assumption \ref{assumption}.

Note that one can then compute the Euler characteristic of all the connected components of minimal strata. Indeed recall from \cite[Cor. 1]{KZ} that for $g=2$ there is only the hyperelliptic component, for $g=3$ there is hyperelliptic and the odd component, and for $g\geq 4$ there are exactly three components given by the hyperelliptic component and the non-hyperelliptic odd and even components.
Moreover by  \cite[Cor. 3]{KZ} we know that the hyperelliptic component has even parity for odd genera and odd parity for even genera. Hence, using the previous information together with the result $\chi(\M_g^{\text{hyp}}(d_g))=\frac{-1}{4g(2g+1)}$ (see \cite[Prop. 10.4]{CMZ20}), we can use the formulas we have shown for $\chi(\M_g(d_g))$ and $\chi(\M_g(2g-1))^\sp$ to compute the Euler characteristic of every connected component of the minimal strata.

Using the statements of the previous paragraph, one can independently compute the Euler characteristic of the spin components in genus $g=2,3$. One can then double-check that the values
given in Table \ref{cap:EulerHolo} using the formula of Theorem \ref{thm:eulerseriesspin} are indeed correct in genus $2$ and $3$.

\begin{appendix}
\section{Polynomiality properties in the splitting formula} \label{Sect:Polyproperties}
The goal of this section is to prove the following result used in the proof of Proposition \ref{prop:psiDRformula}.
\begin{lemma} \label{Lem:DRsplittingpolynomiality}
For $g,n \geq 0$ with $2g-2+n>0$, the right-hand side of equation \eqref{eqn:psiDRformula} from Proposition \ref{prop:psiDRformula} is given by a (cycle-valued) polynomial in $a$.
\end{lemma}
Before we begin, we need two technical preliminaries about sums of polynomials over partitions of given numbers. Here we remark that for the entire section we have the convention that when iterating over sums $b_1 + \ldots + b_e = c$ for $e,c$ fixed, the terms are ordered (so that for $e=2, c=3$ the sums $1+2=2+1=3$ are counted separately).
\begin{lemma} \label{Lem:technicalsumformula}
Let $f \geq 1$, then the function
\[
S_f : \mathbb{Z}_{\geq 0} \to \mathbb{Z}, c \mapsto \sum_{\substack{b_1 +b_2 = c\\b_i \in \mathbb{Z}_{\geq 1}}} (b_1 b_2)^f
\]
is given by a polynomial in $c$ satisfying $S_f(-c)=-S_f(c)$.
\end{lemma}
\begin{proof}
Let $c \geq 1$ be an integer, then we write $b_1=b, b_2 = c-b$ and expand the binomial to compute
\begin{align*}
S_f(c)&=\sum_{b=1}^{c-1} b^f (c-b)^f = \sum_{b=1}^{c-1} b^f \sum_{i=0}^f \binom{f}{i} c^{f-i} (-b)^i \\
&= \sum_{i=0}^f (-1)^i \binom{f}{i} c^{f-i} \sum_{b=1}^{c-1} b^{f+i}\,.
\end{align*}
Using Faulhaber's formula, we can compute the sum of the $b^{f+i}$ in terms of Bernoulli numbers $B_j$, obtaining
\begin{align*}
S_f(c)&=\sum_{i=0}^f (-1)^i  \binom{f}{i} c^{f-i} \frac{1}{f+i+1} \sum_{j=0}^{f+i} \binom{f+i+1}{j} B_j c^{f+i+1-j}\\
&=\sum_{i=0}^f \sum_{j=0}^{f+i} (-1)^i  \binom{f}{i} \binom{f+i+1}{j} \frac{1}{f+i+1}   B_j c^{2f+1-j}\,.
\end{align*}
Since all odd Bernoulli numbers except for $B_1$ vanish, the lemma is proved once we show that the coefficient of $c^{2f+1-1}$ vanishes, since then only odd powers of $c$ appear above. This corresponds to extracting the terms for $j=1$ and so indeed we obtain
\begin{align*}
[c^{2f}] S_f(c) &= B_1 \sum_{i=0}^f (-1)^i  \binom{f}{i} \binom{f+i+1}{1} \frac{1}{f+i+1}  \\
&=B_1 \sum_{i=0}^f (-1)^i  \binom{f}{i} = 0\,.
\end{align*}
\end{proof}

\begin{lemma} \label{Lem:technicalsummationQ}
Let $Q(A,B) = Q(a_1, \ldots, a_n, b_1, \ldots, b_e)$ be a polynomial with rational coefficients in $n+e$ variables (for $n \geq 0, e \geq 1$) satisfying $Q(-A, -B)=Q(A,B)$. Let $c : \mathbb{Q}^n \to \mathbb{Q}$ be a nonzero $\mathbb{Q}$-linear map and denote by $\Lambda_c^+ \subseteq \mathbb{Z}^n$ the set of integer vectors $A$ for which $c(A)$ is integral and non-negative. Then the expression
\begin{equation} \label{eqn:Ptechnical}
P(A) = \sum_{\substack{b_1 + \ldots + b_e = c(A)\\b_i \in \mathbb{Z}_{\geq 1}}} b_1 \cdots b_e \cdot Q(A,B),\ \ A \in \Lambda_c^+
\end{equation}
is given by a polynomial expression in the entries of $A$, and this polynomial is divisible by $c(A)$ and satisfies $P(-A)=-P(A)$.
\end{lemma}
\begin{proof}
We prove this result by induction on $e$, treating the cases $e=1,2$ separately. For $e=1$ the expression \eqref{eqn:Ptechnical} takes the simple shape $P(A) = c(A) \cdot Q(A,c(A))$, which is clearly polynomial in $A$, divisible by $c(A)$ and satisfies
\[
P(-A) = c(-A) \cdot Q(-A,c(-A)) = - c(A) \cdot Q(A, c(A)) = - P(A)\,,
\]
using that $c$ is linear and that $Q$ is even.

For the case $e=2$ we first note that we can assume without loss of generality, that $Q$ is symmetric in the variables $b_1, b_2$. Indeed, since the tuples $(b_1, b_2)$ in the sum over $b_1 + b_2 = c(A)$ are symmetric under exchanging $b_1, b_2$, we can replace $Q$ by the symmetric average $\tilde Q(A,B) = ( Q(A,b_1, b_2) + Q(A,b_2, b_1))/2$ without changing the value of $P$. 

Since the symmetric functions in $b_1, b_2$ are generated as an algebra by the elementary symmetric functions $b_1 + b_2$ and $b_1 b_2$, it suffices to prove the lemma for $Q$ of the form
\[
Q(A,b_1, b_2) = q(A) \cdot (b_1 + b_1)^{f_1} (b_1 b_2)^{f_2}\,,
\]
where $q$ is a polynomial, which must be even for $f_1$ even and odd for $f_1$ odd. Plugging this form into the expression for $P$ we have
\begin{align}
P(A) &= q(A) \sum_{\substack{b_1 + b_2 = c(A)}} b_1 b_2 (b_1 + b_1)^{f_1} (b_1 b_2)^{f_2} \nonumber\\
&= q(A) c(A)^{f_1} \sum_{\substack{b_1 + b_2 = c(A)}} (b_1 b_2)^{f_2+1} = q(A) c(A)^{f_1} S_{f_2+1}(c(A))\,, \label{eqn:Pfore2}
\end{align}
where for the last equality we use the result and notation of Lemma \ref{Lem:technicalsumformula}. From the form \eqref{eqn:Pfore2} we see that $P$ is polynomial. Moreover, for $f_1$ either even or odd, we see that the expression $q(A) c(A)^{f_1}$ is even in $A$. Thus, since $S_{f_2+1}$ is odd by  Lemma \ref{Lem:technicalsumformula}, the overall expression for $P$ is odd and the oddness of $S_{f_2+1}$ also implies that the term $S_{f_2+1}(c(A))$ is divisible by $c(A)$. 

We conclude by proving the result for arbitrary $e \geq 3$, assuming by induction that the result is true for smaller values of $e$. For this, we split the sum over $b_1, \ldots, b_e$ into two parts:
\begin{align}
P(A) &= \sum_{b_1 + \ldots + b_e = c(A)} b_1 \cdots b_e \cdot Q(A,B) \nonumber \\
&= \sum_{b_1 + \ldots + b_{e-2} + \widetilde b = c(A)} b_1 \cdots b_{e-2} \cdot \sum_{b_{e+1}+b_e = \widetilde b} b_{e+1} b_e \cdot Q(A,B)\,. \label{eqn:Pforebigger2}
\end{align}
By applying the proven case $e=2$ of the lemma to the auxiliary functions
\begin{align*}
\widetilde{Q}(\underbrace{A,b_1, \ldots, b_{e-2}, \widetilde b}_{=\widetilde A}, \underbrace{b_{e-1}, b_e}_{= \widetilde B}) = Q(A,B),\ \widetilde c(\widetilde A, \widetilde B) = \widetilde b\,,
\end{align*}
we find that
\[
F(A,b_1, \ldots, b_{e-2}, \widetilde b) = \sum_{b_{e+1}+b_e = \widetilde b} b_{e+1} b_e \cdot Q(A,B)
\]
is an odd polynomial in the entries of $A,B,\widetilde b$, which is divisible by $\widetilde c = \widetilde b$. In particular we can write it as $F = \widetilde b \cdot \overline{F}$ for an even polynomial $\overline{F}$. Plugging this expression back into \eqref{eqn:Pforebigger2} we find
\[
P(A) = \sum_{b_1 + \ldots + b_{e-2} + \widetilde b = c(A)} b_1 \cdots b_{e-2} \cdot \widetilde b \cdot \overline{F}(A,b_1, \ldots, b_{e-2}, \widetilde b)\,.
\]
But this sum is now covered by the proven case of the lemma for $e-1$ and all desired properties of $P$ follow.
\end{proof}

\begin{proof}[Proof of Lemma \ref{Lem:DRsplittingpolynomiality}]
For the proof, we group the summands $(\Gamma, I)$ in \eqref{eqn:psiDRformula} according to the underlying graph $\Gamma$ and show that each partial sum is a polynomial in $a$. Such a graph $\Gamma$ is specified by the data of
\begin{itemize}
\item the number $e \geq 1$ of its edges,
\item a partition $g-e+1 = g' + g''$ of the remaining genus and
\item a partition $J' \sqcup J'' = \{1, \ldots, n\}$ of the marked points.
\end{itemize}
In order to have a nonzero contribution, the markings $s,t$ must go to different vertices, and below we assume that $s \in J'', t \in J'$.
Fixing such a graph $\Gamma$ there are two possible level-assignments $\Gamma^{\pm}$ on $\Gamma$, depending on the choice of the vertex of genus $g'$ going to level $0$ (in $\Gamma^+$) or level $-1$ (in $\Gamma^-$). We claim that depending on the input vector $a$, at most one of the two orientations can give a nonzero contribution to the sum \eqref{eqn:psiDRformula}. Indeed, denote by $c : \mathbb{Q}^n \to \mathbb{Q}$ the linear function defined by
\begin{equation}
c(a) = k(2g'-2+n') - \sum_{j \in J'} a_j\text{ for }k=\frac{\sum_{i=1}^n a_i}{2g-2+n} \text{ and }n'=|J'|+e\,.
\end{equation}
In Proposition \ref{prop:psiDRformula} we only consider $a$ such that the $k$ defined above is an integer, and then in the formula \eqref{eqn:psiDRformula} for $c(a) \geq 0$ we only see contributions $(\Gamma,I)$ with underlying level-graph $\Gamma^+$  (similarly for $\Gamma^-$ and $c(a) \leq 0$). For $c(a) \geq 0$, the possible twists $I$ on $\Gamma^+$ are enumerated by partitions $b_1 + \ldots + b_e = c(a)$ for positive integers $b_i$, and the sum of all contributions from graph $\Gamma$ is given by
\begin{equation} \label{eqn:PplusDRpoly}
P^+(a)=\sum_{b_1 + \ldots + b_e = c(a)}  \frac{1}{e!} b_1 \cdots b_e  \cdot \zeta_{\Gamma *}\left(\DR_{g'}(a_{J'}, \underline{b}), \otimes \DR_{g''}(a_{J''}, -\underline{b})\right)\,.
\end{equation}
Here $\underline{b}=(b_1, \ldots, b_n)$ and again we run through the partitions of $c(a)$ with the order of the summands taken into account. Compared to the original formula this is compensated by the fact that we divide by the size $e!$ of the full automorphism group of $\Gamma$ instead of the group of automorphisms fixing a given twist.
We also note that by our conventions of $s \in J'', t \in J'$ we have $f_{s,t}(\Gamma,I)=1$ here. On the other hand, at points $a$ with $c(a) \leq 0$ the sum of contributions is given by
\begin{equation} \label{eqn:PminusDRpoly}
P^-(a)=-\sum_{b_1 + \ldots + b_e = -c(a)}  \frac{1}{e!} b_1 \cdots b_e  \cdot \zeta_{\Gamma *}\left(\DR_{g'}(a_{J'}, -\underline{b}), \otimes \DR_{g''}(a_{J''}, \underline{b})\right)\,.
\end{equation}
Using the fact that the double ramification cycle is an even polynomial in its entries (i.e. that $\DR_{g_i}(-A) = \DR_{g_i}(A)$, as follows from its formula or from Invariance I of \cite{BHPSS20}) we can apply Lemma \ref{Lem:technicalsummationQ} to conclude that both $P^+$ and $P^-$ are given by polynomials on their respective half-spaces $\{a:c(a) \geq 0\}$ and $\{a:c(a)\leq 0\}$. Moreover, using again that the $\DR_{g_i}$-cycles are even, one sees from \eqref{eqn:PplusDRpoly} and \eqref{eqn:PminusDRpoly} that $P^-(-a)=-P^+(a)$. Combining this with the fact that $P^-$ is odd by Lemma \ref{Lem:technicalsummationQ}, i.e. $P^-(-a)=-P^-(a)$, it follows that $P^+ = P^-$. Thus indeed the contribution of $\Gamma$ to the formula \eqref{eqn:psiDRformula} is polynomial everywhere, which concludes the proof of the lemma.
\end{proof}

For the results involving intersection numbers of spin double ramification cycles, we also need the following lemma on sums of polynomials over odd integers.

\begin{lemma}\label{lem:polynomialodd}
Let $n,m \geq 0$ be integers and $P\in \QQ[x_1,\ldots,x_{n+m}]$ any polynomial. Then there exists a polynomial $Q\in \QQ[x_1,\ldots,x_{n},a]$, such that for $a \geq 0$ of the same parity as $m$, we have
\begin{equation} \label{eqn:oddsumlemma}
Q(x_1,\ldots, x_{n}, a)=\sum_{\begin{smallmatrix} j_1, \ldots, j_m >0 \text{ odd},\\ j_1 + \ldots + j_m=a\end{smallmatrix}} P(x_1,\ldots, x_{n},j_1, \ldots, j_m)\,.
\end{equation}
Moreover, for $m$ odd and all terms of $P$ having of odd total degree in $j_1, \ldots, j_m$, the polynomial $Q$ is divisible by $a$.
\end{lemma}
\begin{proof}
By decompositing $P$ into monomials and drawing out the factors $x_1, \ldots, x_n$, it is easy to reduce to the case $n=0$ and $P$ being a monomial of some degree $e$. Then, we begin with a preparatory remark: using similar techniques as in the proof of Lemma \ref{Lem:technicalsummationQ}, or alternatively a suitable version of Ehrhart reciprocity, it is possible to show that for any polynomial $P' \in \mathbb{Q}[x_1, \ldots, x_m]$ the assignment
\begin{equation} \label{eqn:ehrhartreciprocity}
\mathbb{Z} \to \mathbb{Q}, b \mapsto 
\begin{cases}
\sum_{\begin{smallmatrix} i_1, \ldots, i_m \geq 0,\\ i_1 + \ldots + i_m=b\end{smallmatrix}} P'(i_1, \ldots, i_m) & b \geq 0\\
(-1)^{m-1} \sum_{\begin{smallmatrix} i_1, \ldots, i_m < 0,\\ i_1 + \ldots + i_m=b\end{smallmatrix}} P'(i_1, \ldots, i_m) & b < 0
\end{cases}
\end{equation}
is given by a polynomial $Q' \in \mathbb{Q}[b]$.

Then, if the original polynomial $P$ is of pure degree $e$, we can parameterize the odd numbers $j_\ell$ in \eqref{eqn:oddsumlemma} as $j_\ell = 2i_\ell + 1$. Using moreover that
\[
P(2i_1+1, \ldots 2 i_m+1) = 2^e P(i_1 + \frac{1}{2}, \ldots, i_m + \frac{1}{2})\,,
\]
we can obtain the sum \eqref{eqn:oddsumlemma} from \eqref{eqn:ehrhartreciprocity} by choosing $P'(x_1, \ldots, x_m) = 2^e P(x_1+1/2, \ldots, x_m+1/2)$ and making the substitution b = (a-m)/2. This shows that $Q$ is indeed a polynomial.

Finally, going through the substitutions above one checks that for negative $a$ of the same parity as $m$, the polynomial $Q$ is given by
\begin{equation} \label{eqn:oddsumlemmaneg}
Q(a)=(-1)^{m-1} \sum_{\begin{smallmatrix} j_1, \ldots, j_m < 0 \text{ odd},\\ j_1 + \ldots + j_m=a\end{smallmatrix}} P(j_1, \ldots, j_m)\,.
\end{equation}
For $m$ and $P$ odd, it is then immediate that $Q(-a)=-Q(a)$, so that indeed $Q$ is divisible by $a$.
\end{proof}



\end{appendix}

\bibliographystyle{halpha}
\bibliography{biblio}

\end{document}